\documentclass[10pt]{amsart}  % eventually amsart

\usepackage{epsf,latexsym,amsmath,amssymb,amscd,datetime}
\usepackage{amsmath,amsthm,amssymb,enumerate,eucal,url,calligra,mathrsfs}
\usepackage{subcaption}
\usepackage{graphicx}
\usepackage{color}

\DeclareMathOperator{\SHom}{\mathscr{H}\text{\kern -3pt {\calligra\large om}}\,}
\DeclareMathOperator{\SExt}{\mathscr{E}\text{\kern -2pt {\calligra\large xt}}\,\,}

\newcommand{\naturals}{{\mathbb N}}

\newcommand{\mec}[1]{{\bf #1}}	% for vector 𝐤 and 𝐦  with 𝐤 ⋅𝐦  = k 
\newcommand{\bec}[1]{{\boldsymbol #1}}	% for greek letters

\usepackage{mathrsfs}
\usepackage{amssymb}
\usepackage{dsfont}
\usepackage{verbatim}
\usepackage{url}

\theoremstyle{plain}
\newtheorem{theorem}{Theorem}[section]
\newtheorem{lemma}[theorem]{Lemma}
\newtheorem{proposition}[theorem]{Proposition}
\newtheorem{corollary}[theorem]{Corollary}

\theoremstyle{definition}
\newtheorem{definition}[theorem]{Definition}

\newtheorem{xca}{Exercise}[section]
\newtheorem{example}[theorem]{Example}

\newtheorem{notation}[theorem]{Notation}

\newtheorem{convention}[theorem]{Convention}

\newtheorem{remark}[theorem]{Remark}

\newcommand{\isom}{\simeq} % later this should, perhaps, be changed to =
\newcommand{\ignore}[1]{}

\newcommand{\Hom}{{\rm Hom}}
\newcommand{\integers}{{\mathbb Z}}

\DeclareMathAlphabet{\mathcal}{OMS}{cmsy}{m}{n}

\newcommand\cA{\mathcal{A}}
\newcommand\cB{\mathcal{B}}
\newcommand\cC{\mathcal{C}}
\newcommand\cD{\mathcal{D}}

\newcommand\cF{\mathcal{F}}
\newcommand\cG{\mathcal{G}}

\newcommand\cI{\mathcal{I}}

\newcommand\cL{\mathcal{L}}
\newcommand\cM{\mathcal{M}}
\newcommand\cN{\mathcal{N}}
\newcommand\cO{\mathcal{O}}
\newcommand\cP{\mathcal{P}}

\newcommand\cS{\mathcal{S}}

\newcommand\frakm{\mathfrak{m}}

\newcommand\fraks{\mathfrak{s}}

\def\from{\colon}

\def\isom{\simeq}

\def\eqdef{\overset{\text{def}}{=}}

\DeclareMathOperator{\coker}{coker}

\DeclareMathOperator{\Ext}{Ext}

\def\Hom{\qopname\relax o{Hom}}

\def\implies{\Rightarrow}

\DeclareRobustCommand
  \rddots{\mathinner{\mkern1mu\raise\p@
    \vbox{\kern7\p@\hbox{.}}\mkern2mu
    \raise4\p@\hbox{.}\mkern2mu\raise7\p@\hbox{.}\mkern1mu}}

\newcommand\xhookrightarrow[2][]{\ext@arrow 0062{\hookrightarrowfill@}{#1}{#2}}
\def\hookrightarrowfill@{\arrowfill@\lhook\relbar\rightarrow}

\usepackage{relsize}
\usepackage{tikz}
\usetikzlibrary{math,matrix,arrows,decorations.pathmorphing}
\usepackage{tikz-cd}
\usetikzlibrary{cd}

\usepackage[pdftex,colorlinks,linkcolor=blue,citecolor=brown]{hyperref}
\usepackage{blkarray}
\usepackage{array}
\usetikzlibrary{shapes.misc}

\tikzset{cross/.style={cross out, draw=black, minimum size=2*(#1-\pgflinewidth), inner sep=0pt, outer sep=0pt},
cross/.default={1pt}}

\usepackage[colorinlistoftodos,prependcaption,textsize=small]{todonotes}
\begin{document}

\title[Duality and a Canonical Sheaf]
{Duality and a Canonical Sheaf in Periodic Riemann Functions}

\author{Nicolas Folinsbee}
\address{Department of Mathematics, University of British Columbia,
        Vancouver, BC\ \ V6T 1Z2, CANADA. }
\curraddr{}
\email{{\tt nicolasfolinsbee@gmail.com}}
\thanks{Research supported in part by an NSERC grant.}

\author{Joel Friedman}
\address{Department of Computer Science, 
        University of British Columbia, Vancouver, BC\ \ V6T 1Z4, CANADA. }
\curraddr{}
\email{{\tt jf@cs.ubc.ca}}
\thanks{Research supported in part by an NSERC grant.}

% \date{\today} % , at \currenttime  (get rid of time in final version)}
% \date{July 24, 2024} % , at \currenttime  (get rid of time in final version)}

% \subjclass[2010]{Primary: 05C38, 14H55.  Secondary: 55N30} from Nick's thesis


\begin{abstract}
Let $f\from\integers^2\to\integers$
be a Riemann function whose weight $W$ is a perfect
matching.  Then there is a family of sheaves of $k$-vector
spaces $\{\cM_{W,\mec d}\}_{\mec d\in\integers^2}$
on a five-point topological
that models $f$ in that $f(\mec d)=b^0(\cM_{W,\mec d})$
and that
$$
b^1(\cM_{W,\mec d})=
f^\wedge_\mec K(\mec d-\mec K)
$$
for any $\mec K\in\integers^2$.
Hence a Riemann-Roch formula for $f$ is equivalent to
an Euler characteristic computation of $\cM_{W,\mec d}$.

If $f$ and $W$ are $r$-periodic, then the sheaves
$\cM_{W,\mec d}$ become $\cO_r$-modules of finite type
for a natural sheaf of rings $\cO=\cO_r$.
We show that in this case there is a ``canonical 
$\cO$-module'' $\omega=\omega_W$ and a pairing
for $i=0,1$,
$$
H^i(\cM_{W,\mec 0}\otimes \cF)
\times
\Ext^{1-i}(\cF,\cM_{W^\wedge_\mec L,\mec K})\to 
H^1(\omega)\isom k
$$
that is perfect when $\mec L=\mec K+\mec 1$ and
$\cF$ is a certain type of line bundle or
a certain type of skyscraper sheaf.
In particular when $\cF$ is a line bundle, we realize
the above formula for $b^1(\cM_{W,\mec d})$ as a duality
theorem akin to Serre duality.

We show that canonical $\cO$-module $\omega_W$ is a rather exceptional
element in a family of
tensor products of two modules $\cM\otimes_\cO\cM'$, where
$\cM$ and $\cM'$ vary over $\cO_r$-modules of the form
$\cM_{W',\mec d}$.

This article doesn't assume any background in sheaf theory; rather
we describe all our sheaves as a ``diagrams of vector spaces,''
where each diagram is essentially a sheaf of vector spaces on a fixed
topological space of five points.
\end{abstract}

\maketitle
\setcounter{tocdepth}{3}
\tableofcontents

% This is started from scratch.  I'll probably borrow Sections_*

% Note: for cat_latex to create one file:

\section{Introduction}

In \cite{baker_norine}, Baker and Norine proved what they called a
``graph Riemann-Roch'' formula, which is an equality that
resembles the classical Riemann-Roch formula.
There is a large literature on this formula
\cite{cori_le_borgne,backman, Mohammadi,Caporaso,folinsbee_friedman_weights}
and related formulas, e.g., 
\cite{backman,Gathmann, Hladk, James, amini2013, manjunath2012,
amini_manjunath,Cools}.

The classical Riemann-Roch formula has a modern proof
where the formula represents an expression for the Euler characteristic
of a sheaf, namely a line bundle over the appropriate Riemann surface.
Baker and Norine \cite{baker_norine}
asked whether their formula could be
viewed as an ``Euler characteristic.''
Folinsbee and Friedman \cite{folinsbee_friedman_euler} gave a positive
answer with a few caveats.

Let us provide some details.
We emphasize that we do not assume the reader has seen sheaf theory:
all sheaves are described explicitly as a certain ``diagrams'' of vector
spaces in Section~\ref{se_main}.

Earlier, Folinsbee and Friedman \cite{folinsbee_friedman_weights} defined
a {\em Riemann function} to be any function 
$f\from\integers^n\to\integers$ such that
\begin{enumerate}
\item
$f(\mec d)=0$ whenever $\deg(\mec d)\eqdef d_1+\cdots+d_n$ is
sufficiently small, and
\item for some $C\in\integers$ --- called the {\em offset of $f$} --- we have
$f(\mec d)=\deg(\mec d)+C$ whenever $\deg(\mec d)$ is sufficiently
large.
\end{enumerate}
A Baker-Norine {\em rank function} \cite{baker_norine}
is an example, provided that we
add one to the function.
As such Riemann functions are vast generalizations --- with far less
structure --- of
the Baker-Norine rank functions and related functions.
The classical {\em Riemann's formula} (and {\em Riemann-Roch formula})
for algebraic curves
also give rise to Riemann functions.

Given a Riemann function, $f$, as above,
for any $\mec K\in\integers^n$, define
$f^\wedge_{\mec K}\from\integers^n\to\integers$ via the formula
\begin{equation}\label{eq_generalized_RR_formula_intro}
f(\mec d)-f_\mec K^\wedge(\mec K-\mec d) = \deg(\mec d)+C
\end{equation}
where $C$ is the offset of $f$; we easily see that
$$
f_\mec K^\wedge(\mec d)=
f(\mec K-\mec d) - C - \deg(\mec K-\mec d)
$$
and that $f_\mec K^\wedge$ is also a Riemann function.
We refer to \eqref{eq_generalized_RR_formula_intro} as a
{\em Riemann-Roch formula}.

The Baker-Norine formula (and classical Riemann-Roch formula)
implies that under certain conditions there exists a $\mec K\in\integers^n$
such that $f_\mec K^\wedge=f$.
Folinsbee-Friedman \cite{folinsbee_friedman_weights} wrote this equivalently as
follows: one easily sees that for any Riemann function 
$f\from\integers^n\to\integers$ there is a unique function
$W\from\integers^n\to\integers$ such that
$$
f(\mec d) = \sum_{\mec d'\le\mec d} W(\mec d').
$$
We call $W$ the {\em weight function} (or simply {\em weight}) of $f$.
It turns out that $f_\mec K^\wedge=f$ iff for $\mec L=\mec K+\mec 1$ 
(here $\mec 1=(1,\ldots,1)$) we
have $W^*_\mec L=(-1)^n W$ where
$W^*_\mec L$ is defined by $W^*_\mec L(\mec d)=W(\mec L-\mec d)$.
Given the relationship between $f_\mec K^\wedge$ and $f$, it is simpler
to write $W^*_\mec L=(-1)^n W$.
Folinsbee-Friedman \cite{folinsbee_friedman_weights} also showed that
if $G$ is a complete graph, then the weight functions associated to
the Baker-Norine rank functions have a very simple form.
This yielded a (second) simple formula for the Baker-Norine rank function
on complete graphs, $G$, as an alternative to the 
earlier formula of Cori and Le Borne formula
\cite{cori_le_borgne,cori_le_borgne2}.

Returning to \eqref{eq_generalized_RR_formula_intro}, we ask whether
there is a ``sheaf'' --- we leave this vague for now --- 
$\cM=\cM_\mec d=\cM_{\mec d,k}$,
such that if $b^i(\cM)$ is the $i$-th Betti number of $\cM$, then
\begin{equation}
\label{eq_betti_zero_cM_is_f_betti_one_is_f_wedge_sub_K}
b^0(\cM_\mec d)=f(\mec d), \quad
b^1(\cM_\mec d)=f_\mec K^\wedge(\mec K-\mec d),
\end{equation} 
and $b^i(\cM_\mec d)=0$ if $i\ne 0,1$.  If so, then the left-hand-side
of \eqref{eq_generalized_RR_formula_intro}
equals the {\em Euler characteristic} of $\cM_\mec d$,
and the right-hand-side is a simple expression of $\mec d$.

Folinsbee-Friedman \cite{folinsbee_friedman_euler} showed that if
$W$ is a {\em perfect matching}, meaning a non-negative
weight function of a Riemann function $f\from\integers^2\to\integers$,
then such $\cM_\mec d=\cM_{W,\mec d,k}$ exist, and can be built from $W$.
In fact, for every field, $k$, one can build sheaves of
$k$-vector spaces (on a five-point topological space\footnote{
  A sheaf of vector spaces on a five-point topological space is equivalent
  to the category of a certain
  type of diagram of five vector spaces with certain linear maps
  between these spaces.  Hence these sheaves are quite simple in nature.
  }), $\{\cM_{W,\mec d,k}\}_{\mec d\in\integers^2}$.
(Since the sheaves $\cM_{\mec d,k}$ depend on $k$ in a simple manner,
much like in \cite{friedman_memoirs_hnc},
we tend to suppress the field $k$.)
Moreover, the $\cM_\mec d=\cM_{W,\mec d,k}$
fit together into certain {\em short
exact sequences} involving skyscraper sheaves.
Finally, there is a duality theorem that describes an isomorphism
\begin{equation}\label{eq_duality_theorem_for_H_i_and_H_one_minus_i}
H^i(\cM_{W^*_\mec L,\mec K-d})'\to H^{1-i}(\cM_{W,\mec d}) 
\end{equation} 
where $\,'$ denotes the dual $k$-vector space.

Moreover,
Folinsbee and Friedman \cite{folinsbee_friedman_euler} showed that
all Riemann functions $\integers^n\to\integers$
and resulting Riemann-Roch formulas
express the Euler characteristic of a family of sheaves that can be
pieced together from those involving perfect matchings, 
provided
that one is willing to work with 
{\em virtual sheaves}, meaning formal differences of sheaves.
This is a fairly technical undertaking: there are many choices involved
in building general sheaves for 
Riemann functions $\integers^n\to\integers$
from those involving perfect matchings,
and one has to show that the resulting virtual sheaves
are independent of all choices.

Hence Riemann functions whose weights are perfect matchings
are the building blocks that can express the
Riemann-Roch formulas associated to any Riemann function
$\integers^n\to\integers$
as an Euler characteristic equation of virtual sheaves.

One ultimate goal of writing Baker-Norine formulas as expressing
the Euler characteristics of sheaves would be to get a ``simpler''
proof of their Riemann-Roch formula.
It is unclear if the sheaf theory
of \cite{folinsbee_friedman_euler} will be able to do this:
one reason is that its
``duality theorem'' for perfect matchings (akin
to Serre duality) gives an isomorphism
$$
H^1(\cM)\isom \Hom(\cM,\omega)'.
$$
where the ``canonical sheaf,'' $\omega$, is rather simple and independent of
the perfect matching.
Hence $\omega$ does not encode anything interesting about the
Riemann function; so in the case of Baker-Norine formulas, $\omega$
doesn't encode anything
interesting about the underlying graph.

The point of this article is to give a more remarkable
duality theorem, one that more closely resembles Serre duality.
We restrict our attention to Riemann
functions $f\from\integers^2\to\integers$ whose 
weights $W\from\integers^2\to\integers$ are perfect matchings; in
addition we assume that $f$ is
{\em $r$-periodic} in the sense of
\cite{folinsbee_friedman_weights}: namely,
$f(\mec d+(r,-r))=f(\mec d)$
for all $\mec d$, or equivalently
$W(\mec d+(r,-r))=W(\mec d)$.
In this case the sheaves $\cM_{W,\mec d}=\cM_{W,\mec d,k}$ turn out to be
$\cO_r=\cO_{r,k}$-modules of finite type
for a simple sheaf of rings $\cO_{r,k}$.\footnote{
  By contrast, the sheaves $\{\cM_{W,d}\}$ are not of finite type
  as $k$-diagrams.
  }

Let us briefly describe our duality theorem:
first, we show that if $W$ is an $r$-periodic perfect matching,
then for any $\cO_r$-module $\cF$, and fixed
$\mec K\in\integers^2$ and $\mec L=\mec K+\mec 1$,
for $i=0,1$ we have a
pairing
\begin{equation}\label{eq_pairing_intro}
H^i(\cM_{W,\mec 0}\otimes\cF) \times \Ext^{1-i}(\cF,\cM_{W^*_\mec L,\mec K})
\to H^1(\omega_{W,\mec K}) 
\end{equation} 
where
$$
\omega_{W,\mec K} \eqdef \cM_{W,\mec 0}\otimes\cM_{W^*_\mec L,\mec K}.
$$
This is constructed from: 
\begin{enumerate}
\item
a canonical map
$$
\Ext^{1-i}(\cF,\cM_{W^*_\mec L,\mec K})
\to
\Ext^{1-i}(\cM_{W,\mec 0}\otimes\cF,
\cM_{W,\mec 0}\otimes\cM_{W^*_\mec L,\mec K})
=
\Ext^{1-i}(\cM_{W,\mec 0}\otimes\cF,\omega_{W,\mec K}) ;
$$
\item
the Yoneda pairing
$$
H^i(\cM_{W,\mec 0}\otimes\cF)\times 
\Ext^{1-i}(\cM_{W,\mec 0}\otimes\cF,\omega_{W,\mec K})
\to H^1(\omega_{W,\mec K}) ;
$$
and
\item
the fact that $H^1(\omega_{W,\mec K})\isom k$.
\end{enumerate}
We will then show that for certain values of $\cF$,
namely certain line bundles, $\cL_\mec d=\cL_{\mec d,r,k}$
and certain skyscraper sheaves, 
\eqref{eq_pairing_intro} is a perfect pairing.
Applying this to $\cF=\cL_\mec d$ and using
$$
\cM_{W,\mec 0}\otimes\cL_\mec d\isom \cM_{W,\mec d}
$$
and
$$
\Ext^{1-i}(\cL_{\mec d},\cM_{W^*_\mec L,\mec K})\isom
\Ext^{1-i}(\cO,\cM_{W^*_\mec L,\mec K-\mec d})\isom
H^{1-i}(\cM_{W^*_\mec L,\mec K-\mec d})
$$
gives 
$$
H^i(\cM_{W,\mec d}) \isom 
H^{1-i}(\cM_{W^*_\mec L,\mec K-\mec d})',
$$
which again shows \eqref{eq_duality_theorem_for_H_i_and_H_one_minus_i},
although it is now realized as a very special case of
one choice of $\cF$ for which
\eqref{eq_pairing_intro} is a perfect pairing.

We wish to make a number of remarks on the above.

\begin{remark}
We say that an $\cO_r$-module, $\cF$, satisfies {\em strong
duality} if for $i=0,1$, 
\eqref{eq_pairing_intro} is a perfect pairing.
There is a ``two out of three principle,'' namely if
$0\to\cF_1\to\cF_2\to\cF_3\to 0$ is a short exact sequence,
and two of $\cF_1,\cF_2,\cF_3$ satisfy strong duality, then
so does the third one.
This can be used to give a class of $\cO_r$-modules that
satisfy strong duality larger than merely $\cL_{\mec d,r}$ and certain
skyscraper sheaves.
Hence strong duality seems to hold for a rich set of examples
$\cF$.
\end{remark}

\begin{remark}
If $W,W'$ are perfect matchings, and $\mec d,\mec d'\in\integers^2$,
then
\begin{equation}\label{eq_tensor_product_of_cM_W_and_cM_W_prime}
\cM_{W,\mec d}\otimes\cM_{W',\mec d}
\end{equation} 
has infinite zeroth Betti number, $b^0$; moreover it has vanishing
first Betti number $b^1$ unless $W'=W^*_\mec L$ for some
$\mec L\in\integers^2$.  Hence the canonical sheaf
$$
\omega_{W,\mec K} = \cM_{W,\mec 0}\otimes\cM_{W^*_\mec L,\mec K}
$$
has a property that is exceptional.
\end{remark}

\begin{remark}
If $W$ is an $r$-periodic matching, then 
$\cM_{W^*_\mec L,\mec K}$ is independent --- up to isomorphism ---
of $\mec K,\mec L$ with
$\mec L=\mec K+\mec 1$.  For this reason
$\omega_{W,\mec K}$ depends only $W$ up to isomorphism.
However, as $W$ varies, it seems likely that 
the $\omega_W$ take on non-isomorphic values.
In Subsection~\ref{su_omega_depends_on_W} we discuss this further.
\end{remark}

\begin{remark}
Theorem~3.1 of \cite{folinsbee_friedman_euler} shows that the any
$r$-periodic
Riemann function $f\from\integers^2\to\integers$ can be modeled by
a virtual sheaf built from sheaves $\cM_{W_i,\mec d}$ where the $W_i$
are $r$-periodic perfect matchings.
We don't know if the virtual sheaf one gets 
for $f$ is unique up to isomorphism as an $\cO_r$-module.
(This uniqueness is proven in \cite{folinsbee_friedman_euler} for
a general $f$ not assumed to be $p$-periodic, as diagrams of
$k$-vector spaces.)
Similarly for Riemann functions
$f\from\integers^n\to\integers$.
We hope to address this in a future article.
\end{remark}

Again, we emphasize that
in this article we assume no prior knowledge of sheaves and no knowledge
of algebraic geometry or topology.  We explain everything we do in
terms of ``diagrams'' of vector spaces, of rings, and of algebras
that we define from scratch
(as was done in \cite{folinsbee_friedman_euler}).
We include some
optional sections where we explain how our diagrams relate to modern sheaf
theory.

The rest of this article is organized as follows.
In Section~\ref{se_main} we state our main results.
In Section~\ref{se_basic_o_modules}
we give the diagrams (sheaves) of rings, $\cO_r=\cO_{r,k}$ which are
of interest to us, and give some basic facts about $\cO$-modules.
In Section~\ref{se_canonical} we prove various properties of
the canonical diagram $\omega_{W,\mec K}$.
In Section~\ref{se_duality} we prove a duality theorem that proves that
for $\cF=\cL_{\mec d}$, 
\eqref{eq_pairing_intro} is a perfect pairing for 
$i=1$, by a long computation that assumes that the Yoneda pairing
$$
H^1(\cF)\times\Hom_{\cO_r}(\cF,\cG) \to H^1(\cG)
$$
is given by a certain formula.
Section~\ref{se_strong_duality} formulates ``strong duality,'' which
is a stronger property than the duality in Section~\ref{se_duality}.
Section~\ref{se_strong_duality} then proves strong duality for
$\cF$ being $\cL_\mec d$ and certain skyscraper diagrams,
as well as proving the ``two out of three principle'' for strong
duality.
Section~\ref{se_strong_duality}
requires significantly
more background 
on the homological algebra of $\cO$-modules than the previous sections;
some of this background is relegated 
to Appendix~\ref{ap_yoneda_pairing_etc}.

We remark that some of the foundations that we develop in this article
are based on the unpublished manuscript 
\cite{folinsbee_friedman_two_vertex}.
\section{Main Results}
\label{se_main}

In this section we make our notation precise and state our main theorems.

\subsection{Basic Notation}

We use $\integers$ to denote the integers, and
$\naturals=\integers_{\ge 1}=\{1,2,\ldots\}$
for the natural numbers.
For $n\in\naturals$ we use $[n]$ to denote $\{1,\ldots,n\}$.
For 
$\mec d=(d_1,\ldots,d_n)\in\integers^n$,
the {\em degree} of $\mec d$ is defined as
$\deg(\mec d)=d_1+\cdots+d_n$, and endow $\integers^n$ with its usual
partial order, writing $\mec d'\le\mec d$ to mean $d_i'\le d_i$ for
all $i\in[n]$.

\subsubsection{Direct Sums and Maps}
\label{su_direct_sums_and_maps}

If $S$ is a set and $k$ a field, then, as usual, $k^{\oplus S}$ denotes the
set of maps $f\from S\to k$ such that $f(s)=0$ for all but finitely
many $s$ in $S$; we call $k^{\oplus S}$
the {\em direct sum of $S$ copies of $k$};
for $s\in S$, we
use $\mec e_s\in k^{\oplus S}$ to denote the element that
is $1$ on $s$ and $0$ elsewhere; we refer to $\mec e_s$ as the
{\em standard basis vector at $s$} (with $k^{\oplus S}$ understood)\footnote{
  If $s\in S$ and $S\subset S'$, then $\mec e_s$ can refer to an element
  of $k^{\oplus S}$ or $k^{\oplus S'}$; so we must always be clear on
  what is the
  ambient set containing $s$.
  }.
If $\pi\from S\to T$ a map of sets, then we use
$k^{\oplus \pi}$ 
to denote the unique
linear map $k^{\oplus S}\to k^{\oplus T}$ 
taking $\mec e_s$ to $\mec e_{\pi(s)}$.

\subsection{Riemann Functions}

We refer the reader to Section~2 of 
\cite{folinsbee_friedman_euler} and Section~2 of
\cite{folinsbee_friedman_weights} for proofs of the statements
in this subsection and for examples
of Riemann functions.

\begin{definition}\label{de_Riemann_function}
By a {\em Riemann function} we mean a function
$f\from\integers^n\to\integers$ such that:
\begin{enumerate}
\item 
$f(\mec d)=0$ for $\deg(\mec d)$ sufficiently small, and
\item
for some $C\in\integers$---called the {\em offset of $f$}---we have
$f(\mec d)=\deg(\mec d)+C$ for $\deg(\mec d)$ sufficiently large.
\end{enumerate}
\end{definition}

\begin{definition}
Let
$f\from\integers^n\to\integers$ be a Riemann function of offset $C$, and
$\mec K\in\integers^n$.  We define the {\em $\mec K$-dual of $f$} to be
the function
$f^\wedge_{\mec K}\from\integers^n\to\integers$ given by
$$
f_\mec K^\wedge(\mec d)\eqdef 
f(\mec K-\mec d) - C - \deg(\mec K-\mec d).
$$
\end{definition}
We easily see that $f_\mec K^\wedge$ is a Riemann function with offset
$-\deg(\mec K)-C$, and that
\begin{equation}\label{eq_Riemann_Roch_formula}
\forall\mec d\in\integers^n,\quad
f(\mec d)-f_\mec K^\wedge(\mec K-\mec d) =\deg(\mec d)+C.
\end{equation} 
We refer to \eqref{eq_Riemann_Roch_formula}
as a {\em Riemann-Roch formula for $f$}.
Hence $f_\mec K^\wedge(\mec K-\mec d)$ is independent of $\mec K$.

\begin{definition}
We say that a Riemann function $f\from\integers^n\to\integers$ is
{\em self-dual} if for some $\mec K$ we have $f^\wedge_{\mec K}=f$.
\end{definition}

The point of articles such as
\cite{baker_norine,amini_manjunath} is to study certain Riemann functions
of interest,
$f$, and determine if such $f$ are self-dual;
more precisely,
these articles define ``rank functions'' $r\from\integers^n\to\integers$
such that $f=1+r$ is a Riemann function.
In our approach we do not require a Riemann function to be self-dual;
hence we consider a much wider class of functions $\integers^n\to\integers$
with far less structure.

Our motivation
for the term {\em Riemann function} is the classical {\em Riemann's theorem}
for curves.

For examples of Riemann functions we refer the reader to
Section~2.3 of \cite{folinsbee_friedman_euler}
and
Sections~2.5 and~2.6 of \cite{folinsbee_friedman_weights}.

\subsection{Weights of Riemann Functions}

For details and proofs of the material in this subsection, we refer
the reader to Section~3 of \cite{folinsbee_friedman_weights}.

We say that a function $g\from\integers^n\to\integers$
is {\em initially zero} if $g(\mec d)=0$ when $\deg(\mec d)$ is sufficiently
small; e.g., a Riemann function is initially zero.
If $f\from\integers^n\to\integers$ is any initially zero function, then
there is a unique initially zero
function $W\from\integers^n\to\integers$ such that
$$
f(\mec d) = \sum_{\mec d'\le\mec d} W(\mec d');
$$
we call $W$ the {\em weight of $f$}.

More formally, for any initially zero function 
$W\from\integers^n\to\integers$ we define the function 
$\sigma W\from\integers^n\to\integers$ via
$$
(\fraks W)(\mec d) = \sum_{\mec d'\le\mec d} W(\mec d');
$$
for any function $f\from\integers^n\to\integers$ we define
$\frakm f\from\integers^n\to\integers$ via 
$$
(\frakm f)(\mec d) = \sum_{I\subset[n]} (-1)^{|I|} f(\mec d-\mec e_I)
$$
where $\mec e_I$ is the vector of $1$'s and $0$'s that is $1$ on the
components in $I$ (see Proposition~17 \cite{folinsbee_friedman_weights}).
Proposition~17 of \cite{folinsbee_friedman_weights} says that if
$f$ and $W$ are initially zero, then $f=\fraks W$ iff $W=\frakm f$.
We will later use Theorem~30 of \cite{folinsbee_friedman_weights} which
states that for all Riemann functions $f\from\integers^n\to\integers$
and $\mec L,\mec K\in\integers^n$ such that $\mec L=\mec K+\mec 1$
($\mec 1$ is the all $1$'s vector),
\begin{equation}\label{eq_weight_of_f_dual_is_W_dual}
\frakm f^\wedge_\mec K = (-1)^n W^*_\mec L,
\end{equation} 
where $W^*_\mec L$ from $\integers^n\to\integers$ is given by
\begin{equation}\label{eq_define_dual_weight}
W^*_\mec L(\mec d)=W(\mec L-\mec d).
\end{equation}

\subsection{Periodic Functions}

\begin{definition}
For $r\in\naturals$, we say that $f\from\integers^n\to\integers$
is {\em $r$-periodic} if
for all $i,j\in[n]$ we have
$$
\forall \mec d\in\integers^n,\ i,j\in[n],
\quad
f(\mec d)=f(\mec d+r\mec e_i-r\mec e_j).
$$
\end{definition}

\begin{proposition}
A function $f$ that is initially zero is $r$-periodic iff its weight, $W$ is.
\end{proposition}
\begin{proof}
See Proposition~2.3 \cite{folinsbee_friedman_euler}; the basic idea
is that $f$ and $W$ have the same set of translations, since any translation
commutes with $\fraks$ and with $\frakm$.
\end{proof}

\begin{remark}
Riemann functions based on the classical Riemann's Theorem
(see Section~2.6 of \cite{folinsbee_friedman_weights}) are not
generally $r$-periodic.
By contrast, the Baker-Norine rank functions and related functions
(see Section~2.5 of \cite{folinsbee_friedman_weights}) are 
$r$-periodic.
\end{remark}

\begin{remark}
If $f\from\integers^n\to\integers$ is a Riemann function, then
for any $\mec K\in\integers^n$ and $\mec L=\mec K+\mec 1$
(where $\mec 1$ is the all $1$'s vector), we have
$f_\mec K^\wedge=f$ iff $W_\mec L^*=(-1)^n W$, where
$W_\mec L^*$ is the function defined by
$$
W_\mec L^*(\mec d) = W(\mec L-\mec d)
$$
(see Theorem~30 of \cite{folinsbee_friedman_weights}).
Of course, $(W_\mec L^*)_\mec L^*=W$.
Section~6 of
\cite{folinsbee_friedman_weights} shows that the Baker-Norine rank
function for complete graphs has a remarkably simple weight function.
\end{remark}

\subsection{Riemann functions $\integers^2\to\integers$ that are 
Perfect Matchings}

We model Riemann functions $\integers^2\to\integers$ by starting
with a particularly simple case of functions, related to what we call
{\em perfect matchings}.

\begin{definition}
\label{de_perfect_matching}
We say that $W\from\integers^2\to\integers$ is {\em of bounded support}
if for some $C\in\naturals$ we have $W(\mec d)=0$
whenever $|\deg(\mec d)|>C$.
We say that $W\from\integers^2\to\integers$ is a {\em perfect matching}
if (1) $W$ is of bounded support, and
(2) there is a bijection $\pi\from\integers\to\integers$ such that
$$
W(i,j)= 
\left\{ \begin{array}{ll} 1 & \mbox{if $j=\pi(i)$, and} \\
0 & \mbox{otherwise.}
\end{array} \right. 
$$
If so, we
call $\pi$ the {\em bijection associated to $W$}, and $W$ the
{\em weight function} (or merely {\em weight}) {\em associated to $\pi$}.
\end{definition}
We easily see that in the above definition, $W$ is $r$-periodic iff
$\pi$ satisfies $\pi(i+r)=\pi(i)-r$ for all $i\in\integers$.
We also see that if $W$ is a perfect matching, then $\pi(a)+a$
is bounded above and below (above by $C$ and below by $-C$ for
the same $C$ in Definition~\ref{de_perfect_matching}).

\begin{remark}
Proposition~2.4 of \cite{folinsbee_friedman_euler} characterizes the
Riemann functions $f\from\integers^2\to\integers$ such that $W=\frakm f$
is a perfect matching.
It also shows that $f$ is a Riemann function, then
$W=\frakm f$ is everywhere non-negative iff $W$ is a perfect matching.
\end{remark}

\begin{remark}
If $G$ is a graph consisting of two vertices joined by $r$ edges,
and $f$ is one plus the Baker-Norine rank function, then $f$
is $r$-periodic.  Moreover $W=\frakm f$ is a perfect matching
whose associated bijection $\pi$ is given uniquely by
$\pi(i)=i$ for $i=0,1,\ldots,r-1$ and $\pi(i+r)=\pi(i)-r$ for all
$i\in\integers$.
This was the motivation for a lot of the theory developed
in \cite{folinsbee_friedman_two_vertex}.
\end{remark}

\subsection{$k$-Diagrams and Perfect Matchings}

We now introduce the main structure that is the way we speak of
Betti numbers, skyscrapers, etc.
The following definition assumes no experience with sheaf theory,
although we will soon explain this connection.

We begin by reviewing some of the definitions and main results
in Section~4 of
\cite{folinsbee_friedman_euler}; see also Section~4 of
\cite{folinsbee_friedman_two_vertex}.

\subsubsection{$k$-Diagrams}

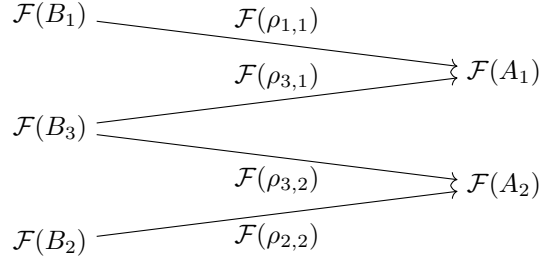
\begin{figure}
$$
\begin{tikzpicture}[scale=0.75]
% \node at (8,2) {$B=B_1\oplus B_3\oplus B_2$};
% \node at (8,0) {$A=A_1\oplus A_2$};
\node (B1) at (0,2) {$\cF(B_1)$};
\node (B2) at (0,-2) {$\cF(B_2)$};
\node (B3) at (0,0) {$\cF(B_3)$};
\node (A1) at (8,1) {$\cF(A_1)$};
\node (A2) at (8,-1) {$\cF(A_2)$};
\draw [->] (B1) -- (A1) node [midway,above] {$\cF(\rho_{1,1})$} ;
\draw [->] (B2) -- (A2) node [midway,below] {$\cF(\rho_{2,2})$} ;
\draw [->] (B3) -- (A1) node [midway,above] {$\cF(\rho_{3,1})$} ;
\draw [->] (B3) -- (A2) node [midway,below] {$\cF(\rho_{3,2})$} ;
\end{tikzpicture}
$$
\caption{Our Diagrams}
\label{fi_our_diagrams_new_again}
\end{figure}

\begin{definition}\label{de_diagram_k_vs}  % vs = vector space
Let $k$ be a field.
By a {\em diagram of $k$-vector spaces}, or simply a
{\em $k$-diagram} we mean a collection, $\cF$,
of data consisting of:
\begin{enumerate}
\item
five $k$-vector spaces, 
$$
\cF(B_1),\cF(B_2),\cF(B_3),\cF(A_1),\cF(A_2)
$$
called the {\em values} of $\cF$;
and
\item
$k$-linear maps $\cF(\rho_{i,j})\from \cF(B_i)\to \cF(A_j)$ for the pairs 
$(i,j)$ where
\begin{equation}\label{eq_restriction_pairs_that_exist}
(i,j) \in \{ (1,1),(2,2),(3,1),(3,2) \}
\end{equation} 
(i.e., $\cF(\rho_{1,2})$ and $\cF(\rho_{2,1})$ don't exist); we call the
$\cF(\rho_{ij})$ the {\em restriction maps} of $\cF$.
\end{enumerate}
We similarly define a {\em diagram of rings} 
replacing ``$k$-vector spaces'' above by ``rings,''
and ``$k$-linear maps'' above by ``morphisms of rings.''
We similarly define a {\em diagram of $k$-algebras}, 
{\em of abelian groups}, etc.
To a diagram of $k$-vector spaces, $\cF$, we associate the $k$-vector spaces
$$
\cF(B)=\cF(B_1)\oplus\cF(B_2)\oplus\cF(B_3),\quad
\cF(A)=\cF(A_1)\oplus\cF(A_2),
$$
and the
linear transformation
$\cF(\partial)\from \cF(B)\to \cF(A)$, called the
{\em differential of $\cF$}, given by
\begin{equation}\label{eq_formula_for_cF_partial}
\cF(\partial)(b_1,b_2,b_3) =
\bigl(
\cF(\rho_{1,1})(b_1)-\cF(\rho_{3,1})(b_3) ,
% \cF(\rho_{3,2})(b_3)-\cF(\rho_{2,2})(b_2)
\cF(\rho_{2,2})(b_2) - \cF(\rho_{3,2})(b_3)
\bigr),
\end{equation} 
and define the {\em zeroth and first cohomology groups of $\cF$}
to be, respectively
$$
H^0(\cF) \eqdef \ker\bigl( \cF(\partial) \bigr), 
\quad
H^1(\cF) \eqdef \coker\bigl( \cF(\partial) \bigr),
$$
i.e.,  
the kernel and cokernel of
$\cF(\partial)$.
By a {\em global section of $\cF$} we mean any tuple
$(b_1,b_2,b_3,a_1,a_2)$ such that
$\cF(\rho_{ij})b_i=a_j$ whenever $\cF(\rho_{ij})$ is defined;
we use $\Gamma(\cF)$ to denote the set of global sections of $\cF$.
We easily see that the map
$(b_1,b_2,b_3,a_1,a_2)\mapsto (b_1,b_2,b_3)$ gives an isomorphism
$\Gamma(\cF)\to H^0(\cF)$.
%
% If $(b_1,b_2,b_3)\in H^0(\cF)$, and $a_j=\cF(\rho_{jj})b_j$ for 
% $j=1,2$, 
% then the tuple $(b_1,b_2,b_3,a_1,a_2)$ satisfies
% $\cF(\rho_{ij})b_i=a_j$ whenever $\cF(\rho_{ij})$ is defined,
% and we refer to $(b_1,b_2,b_3,a_1,a_2)$ as a
% {\em global section of $\cF$}.
We define the {\em zeroth Betti number} and {\em first Betti number of
$\cF$} to be, respectively
$$
b^0(\cF) = \dim\bigl( H^0(\cF) \bigr),
\quad
b^1(\cF) = \dim\bigl( H^1(\cF) \bigr),
$$
and the {\em Euler characteristic of $\cF$} to be
$$
\chi(\cF) \eqdef b^0(\cF) - b^1(\cF).
$$
\end{definition}

% It follows that global sections are in one-to-one correspondence with
% elements of $H^0(\cF)$.

\begin{convention}
When we speak of a {\em $k$-vector space}
or a {\em $k$-diagram} without prior reference to $k$, we understand
$k$ to be an arbitrary field.
\end{convention}

\begin{convention}
Let $\cF$ be a $k$-diagram.
We say that $\rho_{i,j}$ and $\cF(\rho_{i,j})$ {\em exist}
if \eqref{eq_restriction_pairs_that_exist} holds; otherwise
they {\em do not exist}.
At times we refer to $\cF(\rho_{i,j})$ as the {\em restriction $\rho_{i,j}$}
when $\cF$ is clear.  At times we drop the comma between $i$ and $j$
in $\rho_{i,j}$ and $\cF(\rho_{i,j})$.
\end{convention}

At this point we have given the abstract definition, with no motivation.
The motivation is that we can model perfect matchings with $k$-diagrams,
as we now see.

\subsubsection{$k$-Diagrams in Perfect Matchings}

Recall the definitions and notation in 
Subsubsection~\ref{su_direct_sums_and_maps}.

\begin{definition}
\label{de_cM_W_d_when_W_is_a_perfect_matching}
Let $W\from \integers^2\to\integers$ be a perfect matching.
For each $\mec d\in\integers^2$ we use $\cM_{W,\mec d}$ to denote the
following $k$-diagram:
\begin{enumerate}
\item 
for $i=1,2$,
$\cM_{W,\mec d}(B_3)=\cM_{W,\mec d}(A_i)=k^{\oplus\integers}$,
and $\cM_{W,\mec d}(B_i)=k^{\oplus\integers_{\le d_i}}$ 
where $\integers_{\le d}$ denotes the set of integers $\le d$.
\item
For $i=1,2$, 
if $\iota_i$ denotes the inclusion $\integers_{\le d_i}\to\integers$,
then
$\rho_{i,i}=k^{\oplus \iota_i}$.
\item 
For $i=1,2$,
$\rho_{3,i}$ is the identity map.
\end{enumerate}
\end{definition}
We depict these $k$-diagrams in Figure~\ref{fi_cM_W_d}.
\begin{figure}
$$
\begin{tikzpicture}[scale=0.60]
\node (B1) at (0,2) {$\cM_{W,\mec d}(B_1)=k^{\oplus\integers_{\le d_1}}$};
\node (B2) at (0,-2) {$\cM_{W,\mec d}(B_2)=k^{\oplus\integers_{\le d_2}}$};
\node (B3) at (0,0) {$\cM_{W,\mec d}(B_3)=k^{\oplus \integers}$};
\node (A1) at (10,1) {$k^{\oplus\integers}=\cM_{W,\mec d}(A_1$)};
\node (A2) at (10,-1) {$k^{\oplus\integers}=\cM_{W,\mec d}(A_2$)};
\draw [->] (B1) -- (A1) node [midway,above] {$\rho_{1,1}=k^{\rm inclusion}$} ;
\draw [->] (B2) -- (A2) node [midway,below] {$\rho_{2,2}=k^{\rm inclusion}$} ;
\draw [->] (B3) -- (A1) node [midway,above] {$\rho_{3,1}={\rm identity}$} ;
\draw [->] (B3) -- (A2) node [midway,below] {$\rho_{3,2}={\rm identity}$} ;
\end{tikzpicture}
$$
\caption{The $k$-Diagram $\cM_{W,\mec d}$.
Notice that this description of the values and
restrictions of $\cM_{W,\mec d}$ as a $k$-diagram is simplest.
However, when we work with $\cO_r$-modules, it will be easier
to slightly rename these values and the restriction maps.
Compare with Figure~\ref{fi_M_W_d_as_cO_r_modules}.
}
\label{fi_cM_W_d}
\end{figure}

\begin{remark}
Figure~\ref{fi_cM_W_d} is the simplest way to understand the
$\cM_{W,\mec d}$ as $k$-diagrams.  However when we think of
$\cM_{W,\mec d}$ is an $\cO_r$-module, it will be convenient to
modify the values of $\cM_{W,\mec d}$ to reflect their
$\cO_r$-module structure
(see Subsection~\ref{su_cO_modules}
and Figure~\ref{fi_M_W_d_as_cO_r_modules}).
\end{remark}

\begin{theorem}
\label{th_perfect_matching_betti_numbers}
Let $W\from\integers^2\to\integers$ be a perfect matching, and
$\mec d\in\integers^2$.
Let $f=\fraks W$ (which is therefore a Riemann function)
and let $C$ be the offset of $f$.
Then for any $\mec K\in\integers^2$ we have
\begin{align}
\label{eq_b_zero_of_cM_W_mec_d_equals_f_of_mec_d}
b^0\bigl( \cM_{W,\mec d} \bigr) &= (\fraks W)(\mec d) =f(\mec d), \\
\label{eq_b_one_of_cM_W_mec_d_equals_f_sup_wedge_K_etc}
b^1\bigl( \cM_{W,\mec d} \bigr) &= f_\mec K^\wedge(\mec K-\mec d), \\
\nonumber
\chi\bigl( \cM_{W,\mec d} \bigr) &=  \deg(\mec d)+C.
\end{align}
\end{theorem}
\begin{proof}
See Theorems~4.2 (Subsection~4.6) of \cite{folinsbee_friedman_euler}.
\end{proof}

\subsubsection{Duality Theorems}

There is a large class of theorems that can be called ``duality theorems.''
Here we formalize the minimal requirements of such a theorem
in our context.

Let $W$ be a perfect matching and $\mec K,\mec L\in\integers^2$
with $\mec L=\mec K+\mec 1$.
If we apply
\eqref{eq_b_zero_of_cM_W_mec_d_equals_f_of_mec_d}
with $W$ replaced by $W^*_\mec L$ and $\mec d$ replaced with
$\mec K-\mec d$, and we 
\eqref{eq_b_one_of_cM_W_mec_d_equals_f_sup_wedge_K_etc}, we get that
$$
f_\mec K^\wedge(\mec K-\mec d) = b^0(\cM_{W^*_\mec L,\mec K-\mec d});
$$
combining this with
\eqref{eq_b_one_of_cM_W_mec_d_equals_f_sup_wedge_K_etc},
we get that
\begin{equation}\label{eq_duality_in_terms_of_Betti_numbers}
b^1\bigl( \cM_{W,\mec d} \bigr) = 
b^0\bigl( \cM_{W^*_\mec L,\mec K-\mec d}).
\end{equation} 
By a {\em duality theorem} we mean any theorem that provides a
perfect pairing of $k$-vector spaces:
\begin{equation}\label{eq_duality_theorem_means_this}
H^1(\cM_{W,\mec d}) \times H^0\bigl( \cM_{W^*_\mec L,\mec K-\mec d}) \to k
\end{equation} 
for all perfect matchings $W$ and $\mec K,\mec L,\mec d\in\integers^2$
with $\mec L=\mec K+\mec 1$;
since the vector spaces in
\eqref{eq_duality_theorem_means_this} are finite dimensional,
it is equivalent to give an
isomorphism
\begin{equation}\label{eq_duality_in_main_results_section}
H^1(\cM_{W,\mec d})' \xrightarrow{\isom}
H^0\bigl( \cM_{W^*_\mec L,\mec K-\mec d})
\end{equation} 
where $\,'$ denotes the dual of the $k$-vector space;
such an isomorphism immediately implies 
\eqref{eq_duality_in_terms_of_Betti_numbers}.

\subsubsection{Mimicking the classical Riemann-Roch theorem}

[These next two subsubsections can be skipped; they describe the
relation of the above material from
\cite{folinsbee_friedman_euler} with classical sheaf theory and
the Riemann-Roch theorem.]

One interesting part of the proof of Theorem~4.2 of
\cite{folinsbee_friedman_euler} is that 
it mimics the modern proof of the Riemann-Roch theorem.
Namely for any $W,\mec d$ as above
for $i=1,2$ it is easy to see
that there is a {\em short exact sequence} of $k$-diagrams
(we define this in the next section)
$$
0 \to \cM_{W,\mec d} \to \cM_{W,\mec d+\mec e_i} \to \cS_i \to 0
$$
where $\cS_i$ is a {\em skyscraper $k$-diagram} 
(Definition~\ref{de_cS_one_two} below);
this immediately implies that
\begin{equation}\label{eq_euler_characteristic_comparison}
\chi\bigl( \cM_{W,\mec d+\mec e_i} \bigr)
=
\chi\bigl( \cM_{W,\mec d} \bigr)+ 1.
\end{equation} 

This mimics the equation in the middle of page~296 of \cite{hartshorne},
where there is an exact sequence
$$
0 \to \cL(D) \to \cL(D+P) \to k(P) \to 0,
$$
from which one concludes that 
$$
\chi\bigl( \cL(D+P) \bigr) = 
\chi\bigl( \cL(D) \bigr) + 1.
$$

\subsubsection{Connection to Classical Sheaf Theory and Grothendieck's
Sheaf Theory}

The reader familiar with sheaf theory can see that $k$-diagrams are
just sheaf theory on a certain five-point topological space.
Let us give some details; the reader can refer to
Subsubsection~10.8.3 of \cite{folinsbee_friedman_euler} or
Section~4 of \cite{folinsbee_friedman_two_vertex} for more details.

It may help to understand some general underlying principle:
if $X$ is any finite
set, and $\cO$ is a topology on $X$, then each point $P\in X$ lies
in a unique smallest open subset $U_P$ containing $P$.
% [For intuition it may be
% simpler to think of 
% those $(X,\cO)$ that are $T_0$-spaces, i.e.,
% where $U_P=U_Q$ implies $P=Q$.] 
The sets $\{U_P\}_{P\in X}$ form a basis for the topology, and
the $\{U_P\}_{P\in X}$ also become the objects of category $\cC_X$,
with an arrow $U_P\to U_Q$ iff $U_P\subset U_Q$.
Now any sheaf, $\cF$, of $k$-vector spaces on $(X,\cO)$ restricts
to a contravariant functor $\cF|_{\cC_X}$ from $\cC_X$ to the category of 
$k$-vector spaces (which is Grothendieck's notion of a presheaf of
$k$-vector spaces on the category $\cC_X$);
it is not hard to see that the map $\cF\mapsto \cF|_{\cC_X}$ is
an equivalence of categories
(for a proof see \cite{friedman_cohomology}, Theorem~2.1, which
explains this as a special case of the Comparison Lemma 
\cite{sga4.1} Exp.~III, Theorem~4.1).

As a special case of the last paragraph,
let
$X={A_1,A_2,B_1,B_2,B_3}$; for $i=1,2$ let
$$
U_{A_i} = \{ A_i \}, \quad
U_{B_i} = \{ A_i,B_i\},
$$ 
and let
$$
U_{B_3} = \{A_1,A_2,B_3\}.
$$
Then if $\cO$ is the topology on $X$ with 
basis $U_P$ with $P=A_1,A_2,B_1,B_2,B_3$,
then $\{A_1\}$ and $\{A_2\}$ are open subsets, and
$\{B_1\},\{B_2\},\{B_3\}$ are closed subsets.
Let $\cC$ be the category whose objects are $U_P$ with $P\in X$.
Then for any sheaf, $\cF$, of $k$-vector spaces on $(X,\cO)$,
for all $P\in X$, $\cF(U_P)$ is a vector space, and for each
inclusion $U_P\to U_Q$ (i.e., $U_{A_i}\to U_{B_j}$ for $i=j$ or $j=3$)
we get a map $\cF(U_Q)\to \cF(U_P)$.
This gives a diagram of $k$-vector spaces, $\cF|_\cC$.
The previous paragraph implies that
the map $\cF\to\cF|_\cC$ is an equivalence
of vector spaces (i.e., knowing $\cF|_\cC$, one can determine
$\cF$ up to isomorphism).

\subsection{Morphisms and Exact Sequences}

To explain the duality theorem of \cite{folinsbee_friedman_euler},
we need to
define {\em morphisms} of $k$-diagrams;
we refer the reader to 
Section~5 of \cite{folinsbee_friedman_euler} for more details and
examples.

\begin{definition}
\label{de_morphism_k_diagrams}
Let $\cF,\cG$ be two $k$-diagrams.
By a {\em morphism $\phi\from\cF\to\cG$}
we mean the data, $\phi$, of linear maps from each
value of $\cF$ to the corresponding
value on $\cG$ in a way that commutes with the restriction maps:
i.e., $\phi$ consists of
$k$-linear maps $\phi(B_i)\from \cF(B_i)\to\cG(B_i)$ for
$i=1,2,3$ and $\phi(A_j)\from\cF(A_j)\to\cG(A_j)$ for $j=1,2$
such that $\cG(\rho_{ij})\phi(B_i)=\phi(A_j)\cF(\rho_{ij})$ whenever
$\cF(\rho_{ij}),\cG(\rho_{ij})$ exist (i.e., $i=j$ or $i=3$ and any $j$).
For $k$-diagrams $\cF,\cG$ we use $\Hom(\cF,\cG)$ to denote the
set of morphism $\cF\to\cG$.
We similarly define a morphism $\cF\to\cG$ when $\cF,\cG$ are
two diagrams of rings, two diagrams of abelian groups, etc.
\end{definition}
If $\cF,\cG$ are two $k$-diagrams,
then $\Hom(\cF,\cG)$ becomes a $k$-vector space in an evident fashion,
by the $k$-vector space structure on each linear map $\cF(P)\to\cG(P)$
for $P\in\{A_1,A_2,B_1,B_2,B_3\}$.

We illustrate a morphism of $k$-diagrams 
in Figure~\ref{fi_morphism_of_diagrams}.
\begin{figure}
$$
\begin{tikzpicture}[scale=0.50,font=\small]
\node (B1) at (0,4) {$\cF(B_1)$};
\node (B2) at (0,-4) {$\cF(B_2)$};
\node (B3) at (0,0) {$\cF(B_3)$};
\node (A1) at (8,2) {$\cF(A_1)$};
\node (A2) at (8,-2) {$\cF(A_2)$};
\draw [->] (B1) -- (A1) node [midway,above] {$\cF(\rho_{1,1})$} ;
\draw [->] (B2) -- (A2) node [midway,below] {$\cF(\rho_{2,2})$} ;
\draw [->] (B3) -- (A1) node [midway,above] {$\cF(\rho_{3,1})$} ;
\draw [->] (B3) -- (A2) node [midway,below] {$\cF(\rho_{3,2})$} ;
\node (BB1) at (15,4) {$\cG(B_1)$};
\node (BB2) at (15,-4) {$\cG(B_2)$};
\node (BB3) at (15,0) {$\cG(B_3)$};
\node (AA1) at (23,2) {$\cG(A_1)$};
\node (AA2) at (23,-2) {$\cG(A_2)$};
\draw [->] (BB1) -- (AA1) node [midway,above] {$\cG(\rho_{1,1})$} ;
\draw [->] (BB2) -- (AA2) node [midway,below] {$\cG(\rho_{2,2})$} ;
\draw [->] (BB3) -- (AA1) node [midway,above] {$\cG(\rho_{3,1})$} ;
\draw [->] (BB3) -- (AA2) node [midway,below] {$\cG(\rho_{3,2})$} ;
\draw [->,ultra thick] (B1) -- (BB1) node [pos=0.75,above] {$\phi(B_1)$};
\draw [->,ultra thick] (B2) -- (BB2) node [pos=0.75,above] {$\phi(B_2)$};
\draw [->,ultra thick] (B3) -- (BB3) node [pos=0.75,above] {$\phi(B_3)$};
\draw [->,ultra thick] (A1) -- (AA1) node [pos=0.2,above] {$\phi(A_1)$};
\draw [->,ultra thick] (A2) -- (AA2) node [pos=0.2,above] {$\phi(A_2)$};
\node (F) at (4,-6) {\Huge$\cF$};
\node (G) at (19,-6) {\Huge$\cG$};
\draw [->,ultra thick] (F) -- (G) node [midway,above] {\Large$\phi$};
\end{tikzpicture}
$$
\caption{A morphism of diagrams $\phi\from\cF\to\cG$, depicted in thick lines}
\label{fi_morphism_of_diagrams}
\end{figure}

[The following remark can be skipped until Section~\ref{se_strong_duality}.]
If $\phi\from\cF\to\cG$ is a morphism, then we define the
{\em kernel of $\phi$}, the {\em cokernel of $\phi$}, and
the {\em image of $\phi$} as their value-by-value $k$-diagrams;
e.g., $\ker(\phi)$ is the sub-$k$-diagram of $\cF$ whose value at
$P=\{A_1,A_2,B_1,B_2,B_3\}$
is $\ker(\phi(P))$.\footnote{
  In Section~\ref{se_strong_duality} we will use some more technical
  points of homological algebra.  At that point we want to see that
  the category of $k$-diagrams (whose objects are $k$-diagrams and whose
  morphisms are as described above) is an Abelian category in which kernels,
  cokernels, and images are computed in value-by-value.
  The reader can either check this by hand, or refer to
  \cite{sga4.1}, the sentence following Proposition~I.3.1.
  See also \cite{friedman_memoirs_hnc}.
}
In particular, a sequence of morphisms $0\to\cF_1\to\cF_2\to\cF_3\to 0$
is a short exact sequence iff it is value-by-value, i.e.,
iff for all $P=\{A_1,A_2,B_1,B_2,B_3\}$, the sequence
$$
0 \to \cF_1(P)\to\cF_2(P)\to\cF_3(P)\to 0
$$
is a short exact sequence.
The ``strong duality'' theorem of Section~\ref{se_strong_duality}
will make extensive use of short exact sequences and the resulting
long exact sequences of Ext groups.

\subsection{Duality for $k$-Diagrams}

In \cite{folinsbee_friedman_euler}, the following duality theory was
described for $k$-diagrams.
Let $\underline k$ and $\omega=\omega_k$ be the following diagrams:
\newcommand{\ficonsttikzpiece}[5]{
\begin{tikzpicture}[scale=0.40]
\node (B1) at (0,2) {$#1$};
\node (B2) at (0,-2) {$#2$};
\node (B3) at (0,0) {$k$};
\node (A1) at (3,1) {$k$};
\node (A2) at (3,-1) {$k$};
\draw [->] (B1) -- (A1) ;
\draw [->] (B2) -- (A2) ;
\draw [->] (B3) -- (A1) ;
\draw [->] (B3) -- (A2) ;
\node at (1.5,-4){$#3$};
% \node at (1.5,-6){$\myrevision{b^0=#4}$};
% \node at (1.5,-7){$\myrevision{b^1=#5}$};
\end{tikzpicture}
}
$$
\ficonsttikzpiece{k}{k}{\underline{k}}{1}{0}
\quad
\ficonsttikzpiece{0}{0}{\omega=\omega_k}{0}{1}
$$
(The $k$-diagram $\omega$ is also called $\underline k_{/B_1,B_2}$ in 
\cite{folinsbee_friedman_euler}.)
Theorem~9.1 of \cite{folinsbee_friedman_euler} showed that
for any $k$-diagram, $\cF$, we have
\begin{equation}\label{eq_k_vec_space_omega_represents_H_one_dual}
H^1(\cF)' \isom \Hom(\cF,\omega).
\end{equation} 
Hence $\omega$ ``represents'' the functor
$\cF\mapsto H^1(\cF)'$
(and plays the role of the canonical sheaf in
curve theory, e.g., \cite{hartshorne}, middle of page~295).

Theorem~9.2 of \cite{folinsbee_friedman_euler} then shows that 
if $W$ is a perfect matching, and if
$\mec K,\mec L\in\integers^2$ satisfy $\mec L=\mec K+\mec 1$, then
a 
natural isomorphism
$$
\Hom(\omega,\omega)\to\Hom(\underline k,\underline k)
$$
yields an isomorphism
$$
\Hom(\cM_{W,\mec d},\omega) 
\to
\Hom(\underline k,\cM_{W^*_\mec L,\mec K-\mec d}) 
\isom
H^0(\cM_{W^*_\mec L,\mec K-\mec d}).
$$
Combining this with \eqref{eq_k_vec_space_omega_represents_H_one_dual} gives
an isomorphism
$$
H^1(\cM_{W,\mec d})'\to H^0(\cM_{W^*_\mec L,\mec K-\mec d}).
$$
This is a ``duality theorem,'' for then this shows that
$$
b^1(\cM_{W,\mec d}) =  b^0(\cM_{W^*_\mec L,\mec K-\mec d}),
$$
which proves
the formula
$b^1(\cM_{W,\mec d}) = f^\wedge_\mec K(\mec K-\mec d)$
in Theorem~\ref{th_perfect_matching_betti_numbers}
as a consequence of
\eqref{eq_k_vec_space_omega_represents_H_one_dual}
and the fact that $b^0(\cM_{W,\mec d})=f(\mec d)$ in
Theorem~\ref{th_perfect_matching_betti_numbers}.

\eqref{eq_k_vec_space_omega_represents_H_one_dual} represents a
``weaker'' duality theory in that the canonical sheaf
$\omega$ is independent of $W$ and does not reflect any special structure
of $W$.
We expect that a ``stronger'' duality theory would
have a canonical diagram that would involve the special structure of $W$.
This is what we do next, when $W$ is a periodic perfect matching.

\subsection{$\cO$-Modules}
\label{su_cO_modules}

The notion of an {\em $\cO_X$-module} is a standard notion in the context
of a {\em ringed space}, $(X,\cO_X)$ (see, e.g.,
\cite{hartshorne} Section~II.5, or
\cite{ega1} Section~0.4);
in this subsection we explain what this amounts to in our context.
The reason we do this is that if $W\from\integers^2\to\integers$ is
an $r$-periodic perfect matching, then 
the $\cM_{W,\mec d}=\cM_{W,\mec d,k}$ become {\em $\cO_r$-modules}
(of finite type)
where $\cO_r=\cO_{r,k}$ is an interesting diagram of rings;
this is crucial to the rest of this paper, i.e., to our newer duality
theory.
The rough idea is explained in Subsection~10.10
of \cite{folinsbee_friedman_euler}, but is not further pursued there.

\begin{definition}
Let $k$ be a field and $r\in\naturals$.  We define the $k$-diagram
$\cO=\cO_{r,k}$ (or $\cO_r$ with $k$ understood), as follows:
\begin{equation}\label{eq_bigger_O_module_for_r_periodic_values}
\forall i=1,2,\quad
\cO(A_i)=k[x_i,1/x_i], \ \cO(B_i)=k[y_i],
\ \cO(B_3)=k[v,1/v],
\end{equation}
and 
whose restriction maps take $v$ to $x_1^r,x_2^{-r}$ and take $y_i$ to $1/x_i$,
i.e.,
\begin{equation}\label{eq_bigger_O_module_for_r_periodic_restrictions}
\cO(\rho_{i,i})(y_i)=1/x_i\mbox{\ (for $i=1,2$)}, \quad
\cO(\rho_{3,1})(v)=x_1^r, \quad
\cO(\rho_{3,2})(v)=x_2^{-r} 
\end{equation}
(and all restriction maps are the identity on $k$).
\end{definition}
Since the values of $\cO_{r,k}$ are rings, and its restriction maps
are morphisms of rings, we view $\cO_{r,k}$ as a diagram of rings or,
more precisely, a diagram of $k$-algebras.
We depict $\cO_{r,k}$ in 
Figure~\ref{fi_diagrams_rings_short_description_section}.

\begin{figure}
$$
\begin{tikzpicture}[scale=0.75]
% \node (B1) at (0,2) {$\cO_r(B_1)=k[1/x_1]$};
\node (B1) at (0,2) {$\cO_r(B_1)=k[y_1]$};
\node (B2) at (0,-2) {$\cO_r(B_2)=k[1/x_2]$};
\node (B3) at (0,0) {$\cO_r(B_3)=k[v,1/v]$};
\node (A1) at (8,1) {$\cO_r(A_1)=k[x_1,1/x_1]$};
\node (A2) at (8,-1) {$\cO_r(A_2)=k[x_2,1/x_2]$};
% \draw [->] (B1) -- (A1) node [midway,above] {${\rm inclusion}$};
\draw [->] (B1) -- (A1) node [midway,above] {$y_1\mapsto 1/x_1$};
% \draw [->] (B2) -- (A2) node [midway,below] {${\rm inclusion}$};
\draw [->] (B2) -- (A2) node [midway,below] {$y_2\mapsto 1/x_2$};
\draw [->] (B3) -- (A1) node [midway,above] {$v\mapsto x_1^r$};
\draw [->] (B3) -- (A2) node [midway,below] {$v\mapsto x_2^{-r}$};
\end{tikzpicture}
$$
\caption{The Diagram of Rings, $\cO_r=\cO_{r,k}$: this
$k$-diagram has more structure: its values are rings,
and restriction maps are also morphisms
of rings.  Hence $\cO_r=\cO_{r,k}$ is much larger and more structured
than $\underline k$ (the constant diagram whose values are $k$);
hence for any $\cO$-modules $\cF,\cG$,
$\Hom_\cO(\cF,\cG)$ is much smaller than 
$\Hom_{\underline k}(\cF,\cG)$.  This smallness is crucial if
we want to get a stronger form of Serre duality.}
\label{fi_diagrams_rings_short_description_section}
\end{figure}

\begin{definition}
Let $\cO$ be a diagram of rings (recall 
Definition~\ref{de_diagram_k_vs}).
By an $\cO$-module we mean any diagram of abelian groups, $\cF$, such that
for each $P\in X=\{A_1,A_2,B_1,B_2,B_3\}$, 
$\cF(P)$ is endowed with the structure of an $\cO(P)$-module such that
the restriction maps respect the ring structure,
i.e., for each restriction map $\rho\from \cF(B)\to \cF(A)$, for each
$r\in \cO(B)$ and $m\in\cF(B)$ we have
$$
\cF(\rho)( rm ) = \bigl(\cO(\rho)(r)\bigr) \bigl( \cF(\rho)(m)\bigr).
$$
\end{definition}
For two $\cO$-modules $\cF,\cG$ we define $\Hom_\cO(\cF,\cG)$
to be the set of morphisms of diagrams $\phi$ such that for
for each $P\in X=\{A_1,A_2,B_1,B_2,B_3\}$, 
$\phi(P)$ is a morphism of $\cO(P)$-modules.

Note that $\cO_r=\cO_{r,k}$ are $k$-algebras, and hence 
any $\cO_{r,k}$ is automatically a $k$-diagram of vector spaces.

It will be helpful to keep a few examples of $\cO$-modules.
We start with the main example, which justifies the definition
of $\cO_{r,k}$.

\begin{example}
\label{ex_cM_W_d_as_O_r_module_W_is_a_perfect_matching}
Let $\cO=\cO_{r,k}$.  Then for any $r$-periodic perfect matching, $W$,
$\cM_{W,\mec d}=\cM_{k,W,\mec d}$ is an $\cO_{r,k}$-module as follows
(here we are slightly renaming the values of $\cM_{W,\mec d}$):
\begin{enumerate}
\item
We identify $\cM_{W,\mec d}(B_1)=k^{\oplus \integers_{\le d_1}}$
with the $k[y_1]$-module $k[y_1]y_1^{-d_1}$ in the evident fashion,
i.e., if $f\from \integers_{\le d_1}\to k$ is an element of
$k^{\oplus \integers_{\le d_1}}$, then we identify $f$ with
$\sum_i f(i)y_1^{-i}$.
Hence $\cM_{W,\mec d}(B_1)$ is a $k[y_1]$-module,
which makes it a $\cO_r(B_1)$-module.
\item
Similarly for $\cM_{W,\mec d}(B_2)$.
\item
We identify $\cM_{W,\mec d}(A_1)=k^{\oplus\integers}$ with
$k[x_1]$ by the map taking $f$ to $\sum_i f(i)x_1^i$.
The restriction map $\cM_{W,\mec d}(\rho_{1,1})$ takes
$y_1$ to $1/x_1$.
\item
Similarly for $\cM_{W,\mec d}(A_2)$ and $\cM_{W,\mec d}(\rho_{2,2})$.
\item
We identify $\cM_{W,\mec d}(B_3)=k^{\oplus\integers}$
with $\cM_{W,\mec d}(A_1)=k^{\oplus\integers}$ via the identity map,
and $\rho_{3,1}$ is the identity map.
Since 
$\cM_{W,\mec d}(B_3)$ is a $k[x_1,1/x_1]$-module,
it is also a $k[x_1^r,1/x_1^r]$-module, and identifying
$v$ with $x_1^r$, $\cM_{W,\mec d}(B_3)$ is a $k[v,1/v]$-module.
This makes $\cM_{W,\mec d}(B_3)$ isomorphic (non-canonically) to
with $k[v,1/v]^{\oplus r}$.
To define the restriction
$\cM_{W,\mec d}(\rho_{32})$ we take $x_1^a\in \cM_{W,\mec d}(B_3)$
to $x_2^b$ where $b$ is the unique integer with $W(a,b)=1$.
(Of course, we may equally well identify 
$\cM_{W,\mec d}(B_3)$ with $\cM_{W,\mec d}(A_2)$.)
\end{enumerate}
\end{example}
Another way to define $\cM_{W,\mec d}(B_3)$ is to (non-canonically)
choose integers $i_1,\ldots,i_r$, one in each class mod $r$
(for example, choose some $i_s\equiv s\pmod{r}$),
and then for $s\in[r]$ let $j_s$ be the unique integer
with $W(i_s,j_s)$.
We can then identify $\cM_{W,\mec d}(B_3)$ with 
$k[v,1/v]^{\oplus r}$, and define the $\rho_{3,1}$ restriction by mapping
$$
\bigl( f_1(v),\ldots,f_r(v) \bigr) \mapsto
\sum_{s=1}^r f_s(v) x_1^{i_s}
$$
and the $\rho_{3,2}$ restriction by mapping
$$
\bigl( f_1(v),\ldots,f_r(v) \bigr) \mapsto
\sum_{s=1}^r f_s(1/v) x_2^{j_s} .
$$
We depict this in Figure~\ref{fi_M_W_d_as_cO_r_modules}.

\begin{example}
Let $\cO$ be an arbitrary diagram of rings, and let $\cF,\cG$ be 
two $\cO$-modules.
Then the {\em tensor product of $\cF$ and $\cG$ over $\cO$},
denoted $\cF\otimes_\cO\cG$ is the $\cO$-module whose value
at $P\in\{A_1,A_2,B_1,B_3,B_2\}$ is
$\cF(P)\otimes_{\cO(P)}\cG(P)$, and whose restriction maps
are the evident tensor product of restriction maps.
We similarly define $\cF\oplus_\cO\cG$.
\end{example}

\begin{example}
Let $V$ be a $k$-vector space. The {\em constant diagram $V$},
denoted $\underline V$,
refers to the $k$-diagram all of whose values are $V$ and all of whose
restrictions are the identity maps.
Similarly for a ring $R$ and the constant diagram of rings,
$\underline R$.
If $k$ is a field, then it also is a ring.  Then
$\cO=\underline k$ is a diagram of rings.
We easily see that a $\underline k$-module is the same thing as a diagram
of $k$-vector spaces.
\end{example}

For the geometers we
remark that for
$r=1$, $\cO_r$ resembles the Riemann sphere with two closed points;
then the $\cM_{W,\mec d}$ are
$\cO_r$ line bundles. 
However, $r\ge 2$ this is a
much more mysterious space.\footnote{We thank Ehud de Shalit for a
  discussion of $\cO_{r,k}$, and Luc Illusie for questions regarding
  $\cM_{W,\mec d}$ and $\cO_{r,k}$-modules.
  }
[The $r\ge 2$ case of $\cO_r$ is not an orbifold; perhaps it is a 
cover of an orbifold reflecting a \v{C}ech cohomology computation or something
related.]

\begin{figure}
\begin{tikzpicture}[scale=0.75]
\node (B1) at (0,2) {$\cM_{W,\mec d}(B_1)=k[y_1]y_1^{-d_1}$};
\node (B2) at (0,-2) {$\cM_{W,\mec d}(B_2)=k[1/y_2]y_2^{-d_2}$};
\node (B3) at (0,0) {$\cM_{W,\mec d}(B_3)=k[v,1/v]^{\oplus r}$};
\node (A1) at (8,1) {$\cM_{W,\mec d}(A_1)=k[x_1,1/x_1]$};
\node (A2) at (8,-1) {$\cM_{W,\mec d}(A_2)=k[x_2,1/x_2]$};
\node at (-5,2) {Ring maps};
\node at (-5,1) {$y_1\mapsto 1/x_1$};
\node at (-5,0) {$y_2\mapsto 1/x_2$};
\node at (-5,-1) {$v\mapsto x_1^r$};
\node at (-5,-2) {$v\mapsto x_2^{-r}$};
\draw [->] (B1) -- (A1) node [midway,above] {$y_1^{-a}\mapsto x_1^a$};
\draw [->] (B2) -- (A2) node [midway,below] {$y_2^{-a}\mapsto x_2^a$};
\draw [->] (B3) -- (A1) node [midway,above] {$v^a\mec e_s\mapsto x_1^{i_s+ar}$};
\draw [->] (B3) -- (A2) node [midway,below] {$v^a\mec e_s\mapsto x_2^{j_s-ar}$};
\end{tikzpicture}
\caption{If $W$ is an $r$-periodic perfect matching, then
$\cM_{W,\mec d}$ has a natural structure as an $\cO_r$-module:
for $k\in[r]$ we set }
\label{fi_M_W_d_as_cO_r_modules}
\end{figure}

\subsection{Our First Duality Theorem}
\label{su_our_first_duality_theorem}

We now wish to describe our first duality theory, that is the subject
of Section~\ref{se_duality}.

Fix an $r$-periodic perfect matching $W$, and fix a
$\mec K\in\integers^2$ and set $\mec L=\mec K+\mec 1$.
Section~\ref{se_duality} begins by describing
a pairing for any $\cO_r$-module $\cF$ of the form:
$$
H^1(\cM_{W,\mec 0}\otimes\cF)\times
\Hom_{\cO_r}(\cF,\cM_{W^*_\mec L,\mec K})\to k .
$$
Later in Section~\ref{se_duality} we show
that this is a perfect pairing when $\cF$
is equal to certain {\em line bundles}
$\cL_\mec d=\cL_{\mec d,r,k}$
(which we define in Definition~\ref{de_cL_r_mec_d}).

The above pairing is built from the following two general
constructions steps, valid for any
diagram of rings $\cO$:
\begin{enumerate}
\item
For any $\cO$-modules $\cA,\cB,\cC$ there is a natural map
$$
\Hom_\cO(\cA,\cB)\xrightarrow{{\rm id}_\cC\otimes\cdot}
\Hom_\cO(\cC\otimes\cA,\cC\otimes\cB)
$$
where ${\rm id}_\cC$ is the identity map on $\cC$.
\item
For any $\cO$-modules $\cF,\cG$, there is a natural ``Yoneda pairing''
$$
H^1(\cF)\times \Hom_\cO(\cF,\cG) \to H^1(\cG).
$$
\end{enumerate}
Hence for any $\cO_r$-module $\cF$ we have a natural map
$$
\Hom(\cF,\cM_{W^*_\mec L,\mec K})
\to
\Hom\bigl(\cM_{W,\mec 0}\otimes\cF,
\cM_{W,\mec 0}\otimes \cM_{W^*_\mec L,\mec K}\bigr).
$$
So setting
$$
\omega_W = \cM_{W,\mec 0}\otimes \cM_{W^*_\mec L,\mec K},
$$
we get a natural map
$$
\Hom(\cF,\cM_{W^*_\mec L,\mec K})
\to \Hom\bigl(\cM_{W,\mec 0}\otimes\cF,\omega\bigr).
$$
Next, we have a Yoneda pairing
$$
H^1(\cM_{W,\mec 0}\otimes\cF)\times
\Hom\bigl(\cM_{W,\mec 0}\otimes\cF,\omega\bigr) \to H^1(\omega).
$$
It turns out that 
$$
H^1( \omega )\isom k
$$
(this is an involved computation).
Hence fixing an isomorphism to $k$, 
we get a natural pairing
$$
H^1(\cM_{W,\mec 0}\otimes\cF)\times
\Hom(\cF,\cM_{W^*_\mec L,\mec K})\to 
H^1(\cM_{W,\mec 0}\otimes\cF)\times
\Hom(\cM_{W,\mec 0}\otimes\cF,\omega)\to 
H^1(\omega)\to k.
$$
Next we verify that this pairing is a perfect pairing
when $\cF$ is a certain type of {\em line bundle},
$\cL_\mec d=\cL_{r,k,\mec d}$ (this is a straightforward but
tedious computation).
While the $\cL_\mec d$ are defined in 
Definition~\ref{de_cL_r_mec_d}, let us mention some of their basic properties:
$$
\cL_\mec d\otimes\cL_{\mec d'} \isom \cL_{\mec d+\mec d'},
\quad
\cL_\mec 0 \isom \cO_r,
\quad
\cM_{W,\mec d'}\otimes \cL_\mec d \isom
\cM_{W,\mec d'+\mec d},
$$
and for any $\cO_r$-modules $\cF,\cG$,
$$
\Hom_{\cO_r}(\cF,\cG)\isom 
\Hom_{\cO_r}(\cF\otimes\cL_\mec d,\cG\otimes\cL_\mec d).
$$
Since 
$$
H^1(\cM_{W,\mec 0}\otimes\cF)\times
\Hom_{\cO_r}(\cF,\cM_{W^*_\mec L,\mec K})\to  k
$$
is a perfect pairing for $\cF=\cL_\mec d$, this easily implies that there
is an isomorphism
$$
H^1(\cM_{W,\mec d})'\to \Hom_{\cO_r}(\cO_r,\cM_{W^*_\mec L,\mec K-\mec d})
\isom
H^0(\cM_{W^*_\mec L,\mec K-\mec d}).
$$
Hence we get a duality theorem, implying that
$$
b^1( \cM_{W,\mec d} ) 
=
b^0( \cM_{W^*_\mec L,\mec K-\mec d}).
$$

\begin{remark}
If $W,W'$ are $r$-periodic perfect matchings, then we have
that for any $\mec d,\mec d'\in\integers^2$
$$
b^1\bigl( \cM_{W,\mec d}\otimes \cM_{W',\mec d'} \bigr) = 0
$$
unless $W'=W^*_\mec L$ for some $\mec L\in\integers^2$.
Hence the definition of $\omega$ is a rather exceptional choice for
a ``canonical'' $\cO$-module.  
We also remark that
$$
b^0(\omega)=\infty,
$$
so there are properties of $\omega$ that are not mirrored in the classical
Riemann-Roch theorem (classically
the canonical sheaf is a line bundle, and therefore
has finite $b^0$ and $b^1$).
\end{remark}

\subsection{How Canonical is the Canonical $\cO_r$-Module $\omega_W$?}

It turns out that $\cM_{W^*_\mec L,\mec K}$, ranging over all
$\mec K,\mec L$ with $\mec L=\mec K+\mec 1$, is independent --- up to
isomorphism --- of the
particular pair $(\mec K,\mec L)$.  Hence
$$
\omega_{W,\mec K} = \cM_{W,\mec 0}\otimes \cM_{W^*_\mec L,\mec K}
$$
depends only on $W$.  It is an interesting question to see how
$\omega_W$
depends on $W$.
We will briefly addresss this in
Subsection~\ref{su_omega_depends_on_W}.
% It is easy to see that if $W'$ is a translate of $W$ or if 
% $W^*_\mec L$, then $\omega_W\isom\omega_{W'}$.
% More generally, the question of whether or not
% $\omega_W\isom\omega_{W'}$ as $\cO_r$-modules is related to the
% {\em difference multisets} of $W$ and $W'$.
% In particular, generally
% $\omega_W\not\isom\omega_{W'}$ for different $W,W'$, but there are 
% non-trivial cases 
% of $W,W'$ such that
% $\omega_W\isom\omega_{W'}$ but $W$ is not a translate of $W$ or 
% $W^*_\mec L$.  

\subsection{A Stronger Duality Theorem}

In Section~\ref{se_strong_duality} we prove a stronger duality
theorem; this requires more of a background in homological algebra
as it applies to $\cO_r$-modules, part of which is given in
Appendix~\ref{ap_yoneda_pairing_etc}.
Let us describe the main result.

If $\cO$ is any diagram of rings, and $\cA,\cB,\cC$ are any $\cO$-modules,
then the natural map
(of Subsection~\ref{su_our_first_duality_theorem})
$$
\Hom_\cO(\cA,\cB) \to 
\Hom_\cO(\cC\otimes_\cO\cA,\cC\otimes_\cO\cB)  
$$
extends to a map
of $\delta$-functors
$$
\Ext^i_\cO(\cA,\cB)\to
\Ext^i_\cO(\cC\otimes_\cO\cA,\cC\otimes_\cO\cB) 
$$
provided that $\cC$ is a {\em flat $\cO$-module}.
It is easy to see that if $W$ is an $r$-periodic perfect matching,
then $\cM_{W,0}$ is a flat $\cO_r$-module, and hence we get a map
$$
\Ext^i_\cO(\cF,\cM_{W^*_\mec L,\mec K})
\to
\Ext^i_\cO(\cM_{W,0}\otimes\cF,\omega_{W,\mec K}).
$$
Now there is a Yoneda pairing
(generalizing that of Subsection~\ref{su_our_first_duality_theorem})
$$
\Ext^i_\cO(\cA,\cB)\times \Ext^j_\cO(\cB,\cC)
\to
\Ext^{i+j}_\cO(\cA,\cC)
$$
which therefore gives a pairing
\begin{equation}\label{eq_strong_duality_pairing_in_main_tex}
H^i(\cM_{W,\mec 0}\otimes\cF)\times
\Ext^{1-i}_{\cO_r}(\cF,\cM_{W^*_\mec L,\mec K})\to 
H^1(\omega_{W,\mec K})\isom k
\end{equation} 
for any $\cO_r$-module, $\cF$.
We say that $\cF$ satisfies {\em strong duality}
if \eqref{eq_strong_duality_pairing_in_main_tex} is a 
perfect pairing for $i=0,1$.

It follows from general principles and the five-lemma that if
$0\to\cF_1\to\cF_2\to\cF_3\to 0$ is a short exact sequence, and if
two of $\cF_1,\cF_2,\cF_3$ satisfy strong duality, then so does
the third.

In Section~\ref{se_strong_duality}
we establish that a number of $\cF$ satisfy strong duality.
For example, in Section~\ref{se_duality} we've seen that
\eqref{eq_strong_duality_pairing_in_main_tex} is a perfect pairing for
$\cF=\cL_\mec d$ and $i=1$.
[This is true modulo some technicalities regarding the Yoneda pairing
that we are hiding for now.]
But for $\deg(\mec d)$ sufficiently small we can see that for
$\cF=\cL_\mec d$, both
$H^i(\cM_{W,\mec 0}\otimes\cF)$ and
$\Ext^{1-i}_{\cO_r}(\cF,\cM_{W^*_\mec L,\mec K})$
vanish for $i=0$.
Hence $\cL_\mec d$ satisfy strong duality for $\deg(\mec d)\ll 0$.
Now we use an exact sequence
$$
0 \to \cL_\mec d \to \cL_{\mec d+\mec e_1} \to \cS_1 \to 0
$$
where $\cS_1$ is a certain {\em skyscraper sheaf}
to show that (1) $\cS_1$ satisfies strong duality, and then
(2) $\cL_\mec d$ satisfies strong duality for all $\mec d$.

The ``two out of three principle'' shows that, for example, if
$\cF$ is the cokernel of a map of finite direct sums of
$\cO_r$-modules of the form $\cL_\mec d$, then $\cF$ also satisfies
strong duality.
Such $\cF$ are examples of $\cO_r$-modules that are {\em coherent}
in the usual sense of $\cO$-modules on a ringed topological 
$(X,\cO)$.\footnote{
  Hartshorne's textbook \cite{hartshorne} defines a coherent $\cO$-module
  only in the context of schemes; the more general definition 
  can be
  found in, for example, EGA~I, Section 0.5.3.
  }

\subsection{The Interest in Perfect Matchings}

It seems quite restrictive to study
the set of $f\from\integers^2\to\integers$ whose weight $W$
is a perfect matching.
As mentioned before,
the interest in such $f$ and $W$ is that to any Riemann function
$f\from\integers^n\to\integers$ one can associate a
{\em virtual sheaf} that is formally the sum and difference of
sheaves $\cM_{W,\mec d}$ where $W$ is a perfect matching.
Let us roughly describe the procedure, referring to 
\cite{folinsbee_friedman_euler} Subsections~1.3 and~1.4 for details.

First,
any Riemann function $f\from\integers^2\to\integers$ has
a weight $W$ that can be written as
\begin{equation}\label{eq_W_as_sum_and_difference_of_perfect_matchings_W}
W = W_1 + \cdots + W_\ell - \tilde W_1 - \cdots - \tilde W_{\ell-1}
\end{equation} 
for some $\ell$ and perfect matchings
$W_1,\ldots,W_\ell,\tilde W_1,\ldots,\tilde W_{\ell-1}$.
Assuming we work with {\em virtual} sheaves --- meaning
formal sums and differences --- of the sheaves 
$\cM_{W_i,\mec d}$ and $\cM_{\tilde W_i,\mec d}$, one 
gets $f(\mec d)$ and $f^\wedge_\mec K(\mec K-\mec d)$ as the
zeroth and first Betti numbers of a virtual sheaf built from
perfect matchings.
One then proves the virtual sheaf obtained is independent of how
one writes $W$ as the RHS of
\eqref{eq_W_as_sum_and_difference_of_perfect_matchings_W}.

Second, if $f\from\integers^n\to\integers$ is any Riemann function,
for any $\mec d\in\integers^n$ one gets a Riemann function
$\integers^2\to\integers$ obtained by varying two of the $n$ variables
in $\mec d$; i.e., for distinct $i,j\in[n]$, we consider
$g(a_1,a_2)\eqdef f(\mec d+a_1\mec e_i+a_2\mec e_j)$, where
$\mec e_i,\mec e_j$ are standard basis vectors.
Since $g\from\integers^2\to\integers$ is also a Riemann function,
one gets a virtual sheaf for $g(0,0)=f(\mec d)$.
One then proves that the virtual sheaf is independent of the choice
of $i,j$.

As alluded to in the introduction, the above was shown to hold in
\cite{folinsbee_friedman_euler} in the context of $k$-diagrams,
not $\cO_r$-modules.
If $W$ is an $r$-periodic weight function, then
Lemma~3.1 of \cite{folinsbee_friedman_euler} 
shows that 
\eqref{eq_W_as_sum_and_difference_of_perfect_matchings_W}
holds where the $W_i$ and $\tilde W_i$ are also $r$-periodic.
Hence we can define $\cM_{W,\mec d}$ as a formal difference of
sums of $\cO_r$-modules $\cM_{W_i,\mec d}$ and $\cM_{W_i',\mec d}$.
But \cite{folinsbee_friedman_euler} does not prove that
the resulting virtual (or formal difference of)
$\cO_r$-module is independent of these choices;
it only proves that for two different decompositions of
\eqref{eq_W_as_sum_and_difference_of_perfect_matchings_W},
the two formal difference are isomorphic as virtual $k$-diagrams
(not as virtual $\cO_r$-modules).
We hope to address this point in a future work.

\subsection{Additional Remarks and Future Work}

A good challenge for future work is to develop models that
explain generalized Riemann-Roch formulas as a type of sheaf
or diagram of $k$-vector spaces that does not have all the ad hoc
choices we make, and that does not need to pass to virtual
diagrams or virtual sheaves.

Another---perhaps independent challenge---is to use the theory
of diagrams or sheaves to give proofs of self-duality, such as
in the Baker-Norine formula \cite{baker_norine} and some more 
general situations, such as those studied by
Amini and Manjunath \cite{amini_manjunath}.

\section{Some Basic $\cO$-Modules}
\label{se_basic_o_modules}

In this section we describe some basic $\cO_r$-modules and some
operations on them that we will need.
In this section $k$ is a fixed field.

\subsection{Line Bundles and Tensor Products}

Let $\cO$ be an arbitrary $k$-diagram of rings.  Then if
$\cM_1,\cM_2$ are two $\cO$-modules, we define
$\cM_1\oplus\cM_2$ and
$\cM_1\otimes_\cO \cM_2$, respectively, to be the
$\cO$-modules whose values at
$P\in X=\{A_1,A_2,B_1,B_2,B_3\}$ are
$\cM_1(P)\oplus\cM_2(P)$ and $\cM_1(P)\otimes_{\cO(P)}\cM_2(P)$,
respectively,
and whose restriction maps are those induced by those of $\cM_1$ and $\cM_2$.

\begin{definition}\label{de_cL_r_mec_d}
Let $r\in\naturals$ and $\mec d\in\integers^2$.
The {\em line bundle of $\cO_{r,k}$ with $\mec d$ shift}, denoted
$\cL_{\mec d}=\cL_{r,\mec d,k}$ is the $\cO_{r,k}$-module $\cL$ whose values
are those of $\cO_{r,k}$, and with the following restriction maps:
\begin{enumerate}
\item $\rho_{1,1}$ takes $1\in\cL(B_1)$ to $x_1^{d_1}\in\cL(A_1)$
[and therefore, by the map $\cO_{r,k}(B_1)$ to $\cO_{r,k}(A_1)$,
takes $p(y_1)\in\cL(B_1)$ to $x_1^{d_1}p(1/x_1)\in\cL(A_1)$];
\item $\rho_{2,2}$ takes $1\in\cL(B_2)$ to $x_2^{d_2}\in\cL(A_2)$
[and therefore takes $p(y_2)\in\cL(B_2)$ to $x_2^{d_2}p(1/x_2)\in\cL(A_2)$].
\item $\rho_{3,1}$ takes $1\in\cL(B_3)$ to $1\in\cL(A_1)$
[and therefore takes $p(v)\in k[v^{\pm}]=\cL(B_3)$ to 
$p(x_1^r)\in k[x^{\pm}]=\cL(A_1)$;
and
\item $\rho_{3,2}$ takes $1\in\cL(B_3)$ to $1\in\cL(A_2)$.
\end{enumerate}
\end{definition}
Since $k$ is fixed in this section, we will generally omit it as a subscript
and simply write $\cO_r$ and $\cL_{\mec d,r}$.

We depict the $\cL_{r,\mec d}$ in 
Figure~\ref{fi_diagrams_cL_line_bundles}.
Note that $\cL_{r,\mec 0}=\cO_{r}$ (we do mean equality, not merely being
isomorphic).

\begin{figure}
$$
\begin{tikzpicture}[scale=0.75]
\node (B1) at (0,2) {$\cL_{r,\mec d}(B_1)=k[y_1]$};
\node (B2) at (0,-2) {$\cL_{r,\mec d}(B_2)=k[1/x_2]$};
\node (B3) at (0,0) {$\cL_{r,\mec d}(B_3)=k[v,1/v]$};
\node (A1) at (8,1) {$\cL_{r,\mec d}(A_1)=k[x_1,1/x_1]$};
\node (A2) at (8,-1) {$\cL_{r,\mec d}(A_2)=k[x_2,1/x_2]$};
\draw [->] (B1) -- (A1) node [midway,above] {$1\mapsto x_1^{d_1}$};
\draw [->] (B2) -- (A2) node [midway,below] {$1\mapsto x_2^{d_2}$};
\draw [->] (B3) -- (A1) node [midway,above] {$v\mapsto x_1^r$};
\draw [->] (B3) -- (A2) node [midway,below] {$v\mapsto x_2^{-r}$};
\end{tikzpicture}
$$
\caption{The $\cO_r$-Modules, $\cL_{r,\mec d}$.
These are invertible $\cO_r$-modules: they have the same values as
$\cO_r$, and we have $\cL_{r,\mec 0}=\cO_r$.
Since $y_1\mapsto x_1^{-1}$ in $\cO_r$,
the fact that $\cL_{r,\mec d}(\rho_{1,1})$ takes $1$ to $x_1^{d_1}$
implies that
$p(y_1)\in\cL_{r,\mec d}(B_1)$ is mapped to $x_1^{d_1}p(x_1^{-1})$
under $\cL_{r,\mec d}(\rho_{1,1})$.
Hence writing $1\mapsto x_1^{d_1}$ is more concise than describing
the entire map $\cL_{r,\mec d}(\rho_{1,1})\from k[y_1]\to k[x_1^{\pm}]$.
This concise notation is very helpful for the $\cO_r$-modules in this
article.}
\label{fi_diagrams_cL_line_bundles}
\end{figure}

\begin{remark}
Notice that since $\rho_{1,1}$ takes $1$ to $x_1^{d_1}$, by the
$\cO_r$-module structure on $\cL_{\mec d}$, $1$ times $p(y_1)$ must
be taken to $1$ times $p(1/x_1)$.
More generally, to specify a restriction map of an $\cO$-module, $\cM$, 
to specify $\rho_{i,j}$ it suffices to choose a generating set
of $\cM(B_i)$ and to say where each set element is mapped to in
$\cM(A_j)$.
It is usually far simpler to specify restriction maps this way, and
we generally do so in what follows; at times we write the entire map.
\end{remark}

The following proposition is easy.

\begin{proposition}
\label{pr_easy_cL_is_invertible_tensoring_with_cL_etc}
For any $\mec d,\mec d'\in\integers^2$
\begin{equation}\label{eq_tensor_product_of_Ls}
\cL_{r,\mec d}\otimes_{\cO_r} \cL_{r,\mec d'} \isom \cL_{r,\mec d+\mec d'},
\end{equation} 
and in particular
\begin{equation}\label{eq_cL_r_mec_d_has_inverse_same_minus_mec_d}
\cL_{r,\mec d}\otimes_{\cO_r} \cL_{r,-\mec d} \isom \cO_r.
\end{equation} 
If $W$ is any $r$-periodic perfect matching, then
\begin{equation}\label{eq_cM_W_mec_d_tensored_with_cL_mec_d_prime}
\cM_{W,\mec d}\otimes_{\cO_r} \cL_{r,\mec d'} = 
\cM_{W,\mec d+\mec d'}.
\end{equation} 
If $\cF,\cG$ are any $\cO_r$-modules, then
\begin{equation}\label{eq_hom_L_otimes_on_both_sides_does_nothing}
\Hom_{\cO_r}(\cF,\cG)
\isom
\Hom_{\cO_r}
(\cL_{r,\mec d}\otimes_{\cO_r}\cF,\cL_{r,\mec d}\otimes_{\cO_r}\cG) ,
\end{equation} 
and
\begin{equation}\label{eq_hom_L_otimes_cF_move_L_to_other_side}
\Hom_{\cO_r}(\cL_{r,\mec d}\otimes_{\cO_r}\cF,\cG) 
\isom
\Hom_{\cO_r}(\cF,\cL_{r,-\mec d}\otimes_{\cO_r}\cG) .
\end{equation} 
\end{proposition}
\begin{proof}
For $P\in X=\{A_1,A_2,B_1,B_2,B_3\}$ we have
$$
\cL_{r,\mec d}\otimes_{\cO_r} \cL_{r,\mec d'}(P)=\cO(P)\otimes_{\cO_r(P)}\cO(P)
$$
and hence there is an isomorphism for each $P$:
$$
\phi(P)\from \cO(P)\otimes_{\cO_r(P)}\cO(P) \to \cO(P)
\quad\mbox{given by}\quad
\phi(P)(1\otimes 1)= 1.
$$
Let us prove that the maps $\{\phi(P)\}_{P\in X}$ form an isomorphism
$\phi\from\cL_{r,\mec d}\otimes_{\cO_r} \cL_{r,\mec d'}
\to\cL_{r,\mec d+\mec d'}$ of $k$-diagrams.
Glancing at Figure~\ref{fi_morphism_of_diagrams}, we see that it suffices
to show that for each $i,j$ such that $\rho_{i,j}$ exists, we have a
commuting diagram
$$
\begin{tikzpicture}
\node (1) at (0,2) {$\cL_{r,\mec d}\otimes_{\cO_r} \cL_{r,\mec d'}(B_i)$};
\node (2) at (4,2) {$\cL_{r,\mec d+\mec d'}(B_i)$};
\node (3) at (0,0) {$\cL_{r,\mec d}\otimes_{\cO_r} \cL_{r,\mec d'}(A_j)$};
\node (4) at (4,0) {$\cL_{r,\mec d+\mec d'}(A_j)$};
\draw[->] (1) -- (2) node [midway,above] {\Small $\phi(B_i)$};
\draw[->] (3) -- (4) node [midway,below] {\Small $\phi(A_j)$};
\draw[->] (1) -- (3) node [midway,left]  {\Small $\cL_{r,\mec d}\otimes_{\cO_r} \cL_{r,\mec d'}(\rho_{i,j})$};
\draw[->] (2) -- (4) node [midway,right] {\Small $\cL_{r,\mec d+\mec d'}(\rho_{i,j})$};
\end{tikzpicture}
$$
So consider $\rho_{1,1}$, i.e., $i=j=1$:
the restriction map $\cL_{r,\mec d}\otimes_{\cO_r} \cL_{r,\mec d'}(\rho_{1,1})$
takes $1\otimes 1\in k[y_1]$ to
$$
x_1^{d_1}\otimes x_1^{d_1'} = x_1^{d_1+d_1'} (1\otimes 1)
$$
which therefore agrees with $\cL_{r,\mec d+\mec d'}$;
similarly for $\rho_{2,2}$.
The case of $\rho_{3,1}$ is similar, since
$\cL_{r,\mec d}(\rho_{3,1})$ and $\cL_{r,\mec d}$ each take $1$ to $1$;
and similarly for $\rho_{3,2}$.
This shows that $\{\phi(P)\}_{P\in X}$ gives an isomorphism
as in
\eqref{eq_tensor_product_of_Ls}.

Since $\cL_{r,\mec 0}=\cO_r$, 
\eqref{eq_tensor_product_of_Ls} implies
\eqref{eq_cL_r_mec_d_has_inverse_same_minus_mec_d}.

Similarly for \eqref{eq_cM_W_mec_d_tensored_with_cL_mec_d_prime},
whose details we leave to the reader.

[The rest, \eqref{eq_hom_L_otimes_on_both_sides_does_nothing} and
\eqref{eq_hom_L_otimes_cF_move_L_to_other_side}, 
are general properties of invertible
sheaves (e.g., \cite{hartshorne}, Exercise~II.5.1); however, for 
self-containment, we now give the details.]

To check \eqref{eq_hom_L_otimes_on_both_sides_does_nothing},
let $\phi\in\Hom(\cF,\cG)$ and, for brevity,
let $\cL=\cL_{r,\mec d}$; we therefore get maps
$\phi(P)\from \cF(P)\to \cG(P)$ for all
$P\in X=\{A_1,A_2,B_1,B_2,B_3\}$
that make Figure~\ref{fi_morphism_of_diagrams} commutative.
Since $\cL(P)=\cO(P)$ for all $P$, 
for each $P\in X$ there is a natural bijection between maps
$$
\phi(P)\from\cF(P)\to \cG(P)
$$
to maps
$$
{\rm id}_{\cO_r(P)}\otimes_{\cO_r(P)} \phi(P) 
\from \cL(P)\otimes_{\cO_r(P)}\cF(P)\to\cL(P)\otimes_{\cO_r(P)}\cG(P),
$$
namely the map taking 
$$
1\otimes f \mapsto 1\otimes \bigl(\phi(P)\bigr)(f).
$$
We claim that the family of maps $\{{\rm id}\otimes\phi(P)\}_{P\in X}$ 
form a morphism
of $\cO$-modules:
indeed, looking at Figure~\ref{fi_morphism_of_diagrams}, we see that
for each $\rho_{ij}$ there is a commutativity that we need to verify.
So consider, for example, the
$\rho_{1,1}$ commutativity: so consider the two squares:
$$
\begin{tikzpicture}
\node (1) at (0,2) {$\cF(B_1)$};
\node (2) at (3,2) {$\cG(B_1)$};
\node (3) at (0,0) {$\cF(A_1)$};
\node (4) at (3,0) {$\cG(A_1)$};
\draw[->] (1) -- (2) node [midway,above] {\Small $\phi(B_1)$};
\draw[->] (3) -- (4) node [midway,below] {\Small $\phi(A_1)$};
\draw[->] (1) -- (3) node [midway,left]  {\Small $\cF(\rho_{1,1})$};
\draw[->] (2) -- (4) node [midway,right] {\Small $\cG(\rho_{1,1})$};
\end{tikzpicture}
\qquad
\begin{tikzpicture}
\node (1) at (0,2) {$\cL\otimes\cF(B_1)$};
\node (2) at (3,2) {$\cL\otimes\cG(B_1)$};
\node (3) at (0,0) {$\cL\otimes\cF(A_1)$};
\node (4) at (3,0) {$\cL\otimes\cG(A_1)$};
\draw[->] (1) -- (2) node [midway,above] {\Small ${\rm id}\otimes\phi(B_1)$};
\draw[->] (3) -- (4) node [midway,below] {\Small ${\rm id}\otimes\phi(A_1)$};
\draw[->] (1) -- (3) node [midway,left]  {\Small $\cL\otimes\cF(\rho_{1,1})$};
\draw[->] (2) -- (4) node [midway,right] {\Small $\cL\otimes\cG(\rho_{1,1})$};
\end{tikzpicture}
$$
we claim that the left square commutes iff the right square does:
indeed, say that $f\in \cF(B_1)$ gives rise to $f',g,g'$ as follows:
$$
\begin{tikzpicture}
\node (1) at (0,2) {$f$};
\node (2) at (4,2) {$g=\phi(B_1)f$};
\node (3) at (0,0) {$f'=\cF(\rho_{1,1})f$};
\node (4) at (4,0) {$g'=\cG(\rho_{1,1})g$};
\draw[->] (1) -- (2) node [midway,above] {\Small $\phi(B_1)$};
\draw[->,dotted] (3) -- (4) node [midway,below] {\Small $\phi(A_1)$};
\draw[->,dotted] (3) -- (4) node [midway,above] {\Small ???};
\draw[->] (1) -- (3) node [midway,left]  {\Small $\cF(\rho_{1,1})$};
\draw[->] (2) -- (4) node [midway,right] {\Small $\cG(\rho_{1,1})$};
\end{tikzpicture}
$$
then the above square commutes iff $\phi(A_1)$ takes $f'$ to $g'$,
or, equivalently, $\phi(A_1)$ takes $x^{d_1}f'$ to $x^{d_1}g'$
(since $\cF,\cG$ are $\cO_r$-modules), which gives a commuting square
of values
$$
\begin{tikzpicture}
\node (1) at (0,2) {$1\otimes f$};
\node (2) at (3,2) {$1\otimes g$};
\node (3) at (0,0) {$x_1^{d_1}\otimes f'$};
\node (4) at (3,0) {$x_1^{d_1}\otimes g'$};
\draw[->] (1) -- (2) node [midway,above] {\Small ${\rm id}\otimes\phi(B_1)$};
\draw[->] (3) -- (4) node [midway,below] {\Small ${\rm id}\otimes\phi(A_1)$};
\draw[->] (1) -- (3) node [midway,left]  {\Small $\cL\otimes\cF(\rho_{1,1})$};
\draw[->] (2) -- (4) node [midway,right] {\Small $\cL\otimes\cG(\rho_{1,1})$};
\end{tikzpicture}
$$
conversely, a square of values as above gives 
a square of values
$$
\begin{tikzpicture}
\node (1) at (0,2) {$f$};
\node (2) at (3,2) {$g$};
\node (3) at (0,0) {$f'$};
\node (4) at (3,0) {$g'$};
\draw[->] (1) -- (2) node [midway,above] {\Small $\phi(B_1)$};
\draw[->] (3) -- (4) node [midway,below] {\Small $\phi(A_1)$};
\draw[->] (1) -- (3) node [midway,left]  {\Small $\cF(\rho_{1,1})$};
\draw[->] (2) -- (4) node [midway,right] {\Small $\cG(\rho_{1,1})$};
\end{tikzpicture}
$$
Doing similarly for the other $\rho_{i,j}$
proves \eqref{eq_hom_L_otimes_on_both_sides_does_nothing}.

Taking \eqref{eq_hom_L_otimes_on_both_sides_does_nothing} with
$\cF$ as is and $\cG$ replaced with $\cL_{r,-\mec d}\otimes_{\cO_r} \cG$,
and using
$$
\cL_{r,\mec d}\otimes_{\cO_r}\cL_{r,-\mec d}\otimes_{\cO_r}\cG \isom \cO \otimes_{\cO_r}\cG
\isom \cG
$$
establishes \eqref{eq_hom_L_otimes_cF_move_L_to_other_side}.
\end{proof}

\begin{remark}
More generally, if $\cO$ is an arbitrary diagram of rings, one can define
an {\em invertible $\cO$-module} (sometimes {\em $\cO$-line bundle})
as any $\cO$-module $\cL$ such that
there exists $\cL^{-1}$ such that $\cL\otimes\cL^{-1}\isom\cO$.
Hence the line bundles $\cL_{r,\mec d}$
above are special cases of an invertible
$\cO_r$-module.  However, a general invertible $\cO_r$-module, $\cL$,
does not require that $\cL(\rho_{3,1})$ and $\cL(\rho_{3,2})$ be
the identity maps, which is true for $\cL=\cL_{r,\mec d}$.
\end{remark}

\subsection{The Modules $\cM_{W,\mec d}$ for General Non-negative $W$}

For any function
$W\from\integers^2\to\integers_{\ge 0}\cup\{\infty\}$ of bounded support
(i.e., $W(\mec d)=0$ for $|\deg(\mec d)|$ sufficiently large),
and any $\mec d\in\integers^2$,
\cite{folinsbee_friedman_euler} defined a $k$-diagram
$\cM_{W,\mec d}$.
Let us review the construction, since we will need it
to define the canonical diagram $\omega_W$.
When $W$ is a perfect matching, then it will turn out that
the definition of $\cM_{W,\mec d}$ below is equivalent to the
definition in Section~\ref{se_main}.

\begin{definition}
Let $k$ be a field.  If $S$ is a set, we use $k^{\oplus S}$ to
denote the $k$-vector space that is direct sum of one copy of $k$ for
each element of $S$, i.e., whose elements are collections
$\{v_s\}_{s\in S}$ with $v_s\ne 0$ for at most finitely many
values of $s$; for $s\in S$, we use $\mec e_s$ to denote the
vector that is $1$ in component $s$ and $0$ elsewhere.
If $T$ is another set and $\alpha\from S\to T$ a map of sets, then
$\alpha$ gives rise to a unique
$k$-linear transformation, denoted $k^{\oplus\alpha}$,
from $k^{\oplus S}\to k^{\oplus T}$
taking $\mec e_s$ to $\mec e_{\alpha(s)}$.
If $S\subset T$, then the inclusion map $\iota\from S\to T$
gives an injection $k^{\oplus\iota}$ which
we call the {\em inclusion map (of $k^{\oplus S}$ to $k^{\oplus T}$)}.
\end{definition}
In the above one easily checks that if 
$\alpha$ is an injection, surjection, or bijection, then 
the same is true of $k^{\oplus\alpha}$.
Next we fix a convention for multisets (any reasonable 
convention would suffice).

\begin{definition}
Let $S_1,S_2$ be sets, and $W\from S_1\times S_2\to 
\integers_{\ge 0}\cup\{\infty\}$.
The {\em multiset on $S_1\times S_2$ with multiplicities $W$}
refers to the set
\begin{equation}\label{eq_multiset_convention}
{\rm Multi}(W)
= \{ (s_1,s_2,i) \in S_1\times S_2\times\naturals \ | \ i\le W(s_1,s_2) \} ,
\end{equation} 
where if $W(s_1,s_2)=\infty$, then we view all $i$ as satisfying 
$i\le W(s_1,s_2)$.
We refer to the maps 
${\rm Multi}(W)\to S_1$ and 
${\rm Multi}(W) \to S_2$
taking $(s_1,s_2,i)$ to, respectively, $s_1$ and $s_2$,
as, respectively, the {\em first and second projections}.
We use the notation $k^{\oplus W}$ to denote $k^{\oplus{\rm Multi}(W)}$,
which comes with maps 
\begin{equation}\label{eq_oplus_W_notation}
{\rm proj}_i \from k^{\oplus W}\to k^{\oplus S_i}
\end{equation} 
induced by the first and second projections.
The {\em support of $W$} is the set of $(s_1,s_2)\in S_1\times S_2$
such that $W(s_1,s_2)\ge 1$.
When $W$ takes on only the values $\{0,1\}$, then with mild abuse of notation
we may identify ${\rm Multi}(W)$ with its support, which is a subset of
$S_1\times S_2$, since in this case $W$ is determined by its support.
For $r\ge 1$ we say that $W$ is {\em $r$-periodic} if for all
$i,j\in\integers$ we have
$$
W(i,j) = W(i+r,j-r).
$$
\end{definition}

\begin{example}
If $W\from\integers^2\to\integers$ is a perfect matching,
and $\pi\from\integers\to\integers$ is its associated bijection,
then $k^{\oplus W}$ has one copy of $k$ for each pair
$(a_1,\pi(a_1))\in\integers^2$ varying over all $a_1\in\integers$.
In this case we may identify $k^{\oplus W}$ with 
$k^{\oplus\integers}$,
where the first projection is the identity map on 
$k^{\oplus\integers}$,
and the second projection is the map 
$k^{\oplus\integers}\to k^{\oplus\integers}$,
takes $\mec e_{a_1}$ to $\mec e_{\pi(a_1)}$.
Hence both maps
\eqref{eq_oplus_W_notation} are isomorphisms.
\end{example}

\begin{definition}\label{de_cM_W_d}
Let $k$ be a field,
$W\from\integers^2\to\integers_{\ge 0}\cup\{\infty\}$, 
and $\mec d\in\integers^2$.  
We use $\cM_{W,\mec d}$ to denote the following $k$-diagram
(Definition~\ref{de_diagram_k_vs}): 
\begin{enumerate}
\item 
for $i=1,2$,
$\cM_{W,\mec d}(B_i)=k^{\oplus\integers_{\le d_i}}$, 
$\cM_{W,\mec d}(A_i)=k^{\oplus\integers}$,
$\rho_{i,i}$ is the inclusion,
\item 
$B_3=k^{\oplus W}$, and for $j=1,2$, $\rho_{3,j}$ are
the projection maps
(as in \eqref{eq_oplus_W_notation}).
\end{enumerate}
\end{definition}
We depict these $k$-diagrams in Figure~\ref{fi_cM_W_d_sect_basic}.

\begin{figure}
$$
\begin{tikzpicture}[scale=0.60]
\node (B1) at (0,2) {$\cM_{W,\mec d}(B_1)=k^{\oplus\integers_{\le d_1}}$};
\node (B2) at (0,-2) {$\cM_{W,\mec d}(B_2)=k^{\oplus\integers_{\le d_2}}$};
\node (B3) at (0,0) {$\cM_{W,\mec d}(B_3)=k^{\oplus W}$};
\node (A1) at (10,1) {$k^{\oplus\integers}=\cM_{W,\mec d}(A_1$)};
\node (A2) at (10,-1) {$k^{\oplus\integers}=\cM_{W,\mec d}(A_2$)};
\draw [->] (B1) -- (A1) node [midway,above] {$\rho_{1,1}={\rm inclusion}$} ;
\draw [->] (B2) -- (A2) node [midway,below] {$\rho_{2,2}={\rm inclusion}$} ;
\draw [->] (B3) -- (A1) node [midway,above] {$\rho_{3,1}$} ;
\draw [->] (B3) -- (A2) node [midway,below] {$\rho_{3,2}$} ;
\end{tikzpicture}
$$
\caption{The $k$-Diagram $\cM_{W,\mec d}$.}
\label{fi_cM_W_d_sect_basic}
\end{figure}

Note that if $W$ is $r$-periodic, then $\cM_{W,\mec d}$ can be viewed as
a ring of $\cO_r$-modules: in other words, for
$P\in X=\{A_1,A_2,B_1,B_2,B_3\}$, we can view 
$\cM_{W,\mec d}(P)$ as an $\cO_r(P)$ module as follows:
\begin{enumerate}
\item
For $i=1,2$ we can identity $\cM_{W,\mec d}(A_i)=k^{\oplus \integers}$
with $k[x_i^{\pm}]$, i.e., where the element $\alpha\from\integers\to k$ in
$k^{\oplus \integers}$ is taken to
$$
\sum_{j\in\integers} \alpha(j)x_i^j.
$$
\item
Similarly, for
$i=1,2$ we can identity $\cM_{W,\mec d}(B_i)=k^{\oplus\integers_{\le d_i}}$
with the
$k[y_i]$ module $k[y_i]y^{d_1}\subset k[x_i^{\pm}]$ where
$y_i$ is identified with $1/x_i$; hence
$\beta\from\integers_{\le d_i}\to k$ is taken to
$$
\sum_{j\in\integers} \beta(j)y_i^j.
$$
\item
If $W$ is $r$-periodic, then for all $i,j$ we have
$$
W(i,j)=W(i+r,j-r).
$$
This gives us a map on ${\rm Multi}(W)$ 
\eqref{eq_multiset_convention} to itself taking
$(s_1,s_2,i)$ to $(s_1+r,s_2-r,i)$,
and therefore a map $k^{\oplus W}$ to itself.
We view $v\in k[v^\pm]$ as acting on $k^{\oplus W}$ as taking
$W(i,j)$ to $W(i+r,j-r)$; for any $q\in\integers$, this gives a
map $v^q$, taking $W(i,j)$ to $W(i+qr,j-qr)$, from which we get
a $k[v^\pm]$ action on $k^{\oplus W}$.
\end{enumerate}
Since our ring maps of $\cO_r$ take $y_i$ to $1/x_i$ and
$v$ to $x_1^r$ and $x_2^{-r}$, $\cM_{W,\mec d}$ becomes an $\cO_r$
module.

\begin{example}
If $W\from\integers^2\to\{0,1\}$ is a perfect matching, then
the definition of $\cM_{W,\mec d}$ in Definition~\ref{de_cM_W_d}
agrees with its definition in Section~\ref{se_main}
(i.e., Definition~\ref{de_cM_W_d_when_W_is_a_perfect_matching},
and its $\cO_r$-module structure explained in
Example~\ref{ex_cM_W_d_as_O_r_module_W_is_a_perfect_matching}).
\end{example}

\subsection{Tensor Products of Certain $\cO_r$-modules}

The point of this subsection is to explicitly describe
$\cM_{W,\mec d}\otimes_{\cO_r}\cM_{W',\mec d'}$.
We will do this in greater generality.

\begin{definition}
We say that a function $W\from\integers^2\to\integers$ 
has {\em bounded support} if for some $C_1,C_2\in\integers$ we have
$$
\forall \mec d\in\integers^2,
\qquad
W(\mec d)\ne 0 
\quad\implies\quad
C_1\le\deg(\mec d)\le C_2.
$$
\end{definition}

\begin{definition}
If $(i,j)\in\integers^2$, we use $\delta_{i,j}$ to denote the function
$\integers^2\to\integers_{\ge 0}$ that is $1$ on $(i,j)$ and $0$ elsewhere.
By the {\em simple $r$-periodic weight at $(i,j)$} we mean the
function $\integers^2\to\integers_{\ge 0}$ given by
$$
{\rm Simple}_{r;i,j} = \sum_{q\in\integers} 
\delta_{i+rq,j-rq}
$$
(and we often just write ${\rm Simple}_{i,j}$ 
when $r$ is understood).
\end{definition}
Hence ${\rm Simple}_{r;i,j}$ is $r$-periodic, and for any
$q\in\integers$
$$
{\rm Simple}_{r;i,j} = {\rm Simple}_{r;i+qr,j-qr}.
$$
Moreover, the only function $f\from\integers^2\to\integers_{\ge 0}$
that satisfies $f\le {\rm Simple}_{r;i,j}$
is either $f={\rm Simple}_{r;i,j}$ or $f=0$.
For these reasons, functions of the
form ${\rm Simple}_{r;i,j}$ are the ``simplest'' nonzero
$r$-periodic weight functions.

\begin{example}
Let $W$ be the $2$-periodic matching with $W(0,0)=W(1,1)=1$.
(Hence $W(2k,-2k)=W(2k+1,-2k+1)=1$ for all $k\in\integers$.)
Then
$$
W={\rm Simple}_{2;0,0}+{\rm Simple}_{2;1,1}
={\rm Simple}_{2;(0,1),(0,1)}.
$$
We depict these functions in Figure~\ref{fi_some_two_periodic_examples}.
\end{example}

\newcommand{\plotSomeWeightExamples}[4]{   
  \begin{tikzpicture}[scale=#1]
  \foreach \x in {-5,...,5} {
    \draw (\x * #2, -5.5 * #2) -- (\x * #2, 5.5 * #2);
    \draw (-5.5 * #2, \x * #2) -- (5.5 * #2,\x * #2);
  }
  \draw[very thick] (0 * #2, -5.5 * #2) -- (0 * #2, 5.5 * #2);
  \draw[very thick] (-5.5 * #2, 0 * #2) -- (5.5 * #2,0 * #2);
  \node at (0,-14) {#3};
  #4
  \end{tikzpicture}
}
\begin{figure}
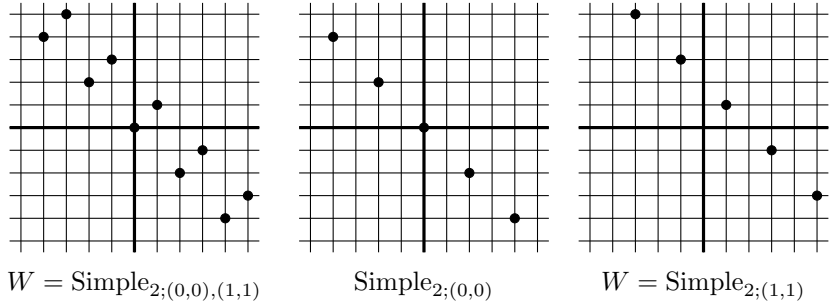

$$
\plotSomeWeightExamples{0.15}{2}{$W={\rm Simple}_{2;(0,0),(1,1)}$}{
  \filldraw (0,0) circle (0.4);
  \filldraw (2,2) circle (0.4);
  \filldraw (4,-4) circle (0.4);
  \filldraw (4+2,-4+2) circle (0.4);
  \filldraw (8,-8) circle (0.4);
  \filldraw (8+2,-8+2) circle (0.4);
  \filldraw (-4,4) circle (0.4);
  \filldraw (-4+2,4+2) circle (0.4);
  \filldraw (-8,8) circle (0.4);
  \filldraw (-8+2,8+2) circle (0.4);
}
\quad
\plotSomeWeightExamples{0.15}{2}{${\rm Simple}_{2;(0,0)}$}{
  \filldraw (0,0) circle (0.4);
  \filldraw (4,-4) circle (0.4);
  \filldraw (8,-8) circle (0.4);
  \filldraw (-4,4) circle (0.4);
  \filldraw (-8,8) circle (0.4);
}
\quad
\plotSomeWeightExamples{0.15}{2}{$W={\rm Simple}_{2;(1,1)}$}{
  \filldraw (2,2) circle (0.4);
  \filldraw (4+2,-4+2) circle (0.4);
  \filldraw (8+2,-8+2) circle (0.4);
  \filldraw (-4+2,4+2) circle (0.4);
  \filldraw (-8+2,8+2) circle (0.4);
}
$$
\caption{$W$ is $2$-periodic that satisfies $W(2k,-2k)=W(2k+1,-2k+1)=1$ for
$k\in\integers$ and otherwise $W=0$.  We depict $W$ by showing in bold
points where $W=1$.  $W$ can be written as the sum of two $2$-periodic
functions, ${\rm Simple}_{2;(0,0)}$ plus ${\rm Simple}_{2;(1,1)}$, each
of which we depict above.  This is the unique way of writing $W$ as the
sum of two simple $2$-periodic functions.
Of course, ${\rm Simple}_{2;(0,0)}={\rm Simple}_{2;(2k,-2k)}$ for all
$k\in\integers$, so simple functions aren't uniquely described as
${\rm Simple}_{2;(i,j)}$.
}
\label{fi_some_two_periodic_examples}
\end{figure}

The following is easy.

\begin{proposition}
Let $W\from\integers^2\to\integers_{\ge 0}$ be an $r$-periodic degree
bounded
function.  Then $W$ is a finite sum of $r$-periodic simple functions
$$
W = \sum_{i=0}^{r-1} \sum_{j\in\integers} W(i,j){\rm Simple}_{r;i,j}.
$$
\end{proposition}
\begin{proof}
The sum on the left is finite, since $W$ is a weight function,
and hence for each $0\le i\le r-1$, there are only finitely many $j$
at which $W(i,j)$ is nonzero.
The sum on the left clearly agrees with $W$ at all $(i',j')$ with
$0\le i'\le r-1$ and $j'$ arbitrary.  Therefore the two sides agree
on all $(i',j')$ by the $r$-periodicity of both sides.
\end{proof}

Now we define a $\cO_r$-module for any $W$ that is the sum of simple,
$r$-periodic functions.

\begin{definition}
Let $r\in\naturals$, $\ell\in\integers_{\ge 0}$, $k$ be a field, 
$\mec i=(i_1,\ldots,i_\ell),\mec j=(j_1,\ldots,j_\ell)\in\integers^\ell$,
and $\mec d\in\integers^2$.
We define
the $\cO_r$-module
$$
\cN=\cN_{k,r;\mec i,\mec j;\mec d}
$$
to be the 
module such that:
\begin{enumerate}
\item 
for $P\in \{A_1,A_2,B_1,B_2\}$, $\cN(P)=\cO_r(P)$;
\item 
$\cN(B_3)= k[v^{\pm}]^{\ell}$;
\item 
the restriction map $\rho_{1,1}$ takes $1$ to $x_1^{d_1}$;
\item 
the restriction map $\rho_{2,2}$ takes $1$ to $x_2^{d_1}$;
\item 
for $1\le q\le \ell$,
the restriction map $\rho_{3,1}$ takes $\mec e_q$, the vector
$(0,\ldots,0,1,0,\ldots,0)\in k[v^{\pm}]^{\ell}$ to
$x_1^{i_q}$; and
\item
similarly $\rho_{3,2}$ takes $\mec e_q$ to $x_2^{j_q}$.
\end{enumerate}
\end{definition}

\begin{example}
$$
\cL_{r;\mec d}=\cN_{(0,0);\mec d},
$$
since for $\cN=\cN_{(0,0);\mec d}$ we have $\ell=1$, $\mec i=(0)$,
$\mec j=(0)$, and $\cN(B_3)=k[v^{\pm}]$ and takes $1$ to $x_1^0=1$ in 
$\cN(A_1)$ and to $x_2^0=1$ in $\cN(A_2)$.
\end{example}

\begin{example}
Let $W$ be an $r$-periodic perfect matching, with associated permutation
$\pi$.  Then
$$
\cM_{W,\mec d} = \cN_{(0,1,\ldots,r-1),(\pi(0),\ldots,\pi(r-1));\mec d}.
$$
\end{example}

If $\mec i\in\integers^\ell$ and $\mec i'\in\integers^{\ell'}$
we use the notation $\mec i{\bec +}\mec i'\in\integers^{\ell \ell'}$
to denote the vector whose components are the sum of those of
$\mec i$ and $\mec i'$:
$$
\mec i{\bec +}\mec i' = (i_1 + i'_1,i_2 + i'_1,\ldots,i_\ell + i'_{\ell'})
\in \integers^{\ell\ell'}
$$
The following proposition follows directly from the definitions.

\begin{proposition}
Let $r\in\naturals$, $\ell,\ell'\in\integers_{\ge 0}$, 
$\mec i,\mec j\in\integers^\ell$,
$\mec i',\mec j'\in\integers^{\ell'}$,
and
$\mec d,\mec d'\in\integers^2$.  Then
$$
\cN_{r;\mec i,\mec j;\mec d}
\otimes_{\cO_r}
\cN_{r;\mec i',\mec j';\mec d'}
\isom
\cN_{r;\mec i{\bec +}\mec i',\mec j{\bec +}\mec j';\mec d + \mec d'} .
$$
\end{proposition}
We illustrate this proposition in a case of interest in 
Figure~\ref{fi_tensor_product_cM_W_and_cM_W_prime_example}.

\begin{figure}
$$
\plotSomeWeightExamples{0.12}{2}{\Small $W={\rm Simple}_{2;(0,0),(1,1)}$}{
  \filldraw (0,0) circle (0.4);
  \filldraw (2,2) circle (0.4);
  \filldraw (4,-4) circle (0.4);
  \filldraw (4+2,-4+2) circle (0.4);
  \filldraw (8,-8) circle (0.4);
  \filldraw (8+2,-8+2) circle (0.4);
  \filldraw (-4,4) circle (0.4);
  \filldraw (-4+2,4+2) circle (0.4);
  \filldraw (-8,8) circle (0.4);
  \filldraw (-8+2,8+2) circle (0.4);
}
\quad
\plotSomeWeightExamples{0.12}{2}{\Small $W'={\rm Simple}_{2;(0,1),(1,0)}$}{
  \filldraw (0,2) circle (0.4);
  \filldraw (2,0) circle (0.4);
  \filldraw (4,-4+2) circle (0.4);
  \filldraw (4+2,-4) circle (0.4);
  \filldraw (8,-8+2) circle (0.4);
  \filldraw (8+2,-8) circle (0.4);
  \filldraw (-4,4+2) circle (0.4);
  \filldraw (-4+2,4) circle (0.4);
  \filldraw (-8,8+2) circle (0.4);
  \filldraw (-8+2,8) circle (0.4);
}
\quad
\plotSomeWeightExamples{0.12}{2}{\Small $W\star_2 W'$}{
  \filldraw (0,2) circle (0.4);
  \filldraw (2,0) circle (0.4);
  \filldraw (4,-4+2) circle (0.4);
  \filldraw (4+2,-4) circle (0.4);
  \filldraw (8,-8+2) circle (0.4);
  \filldraw (8+2,-8) circle (0.4);
  \filldraw (-4,4+2) circle (0.4);
  \filldraw (-4+2,4) circle (0.4);
  \filldraw (-8,8+2) circle (0.4);
  \filldraw (-8+2,8) circle (0.4);
  \filldraw (2+0,2+2) circle (0.4);
  \filldraw (2+2,2+0) circle (0.4);
  \filldraw (2+4,2-4+2) circle (0.4);
  \filldraw (2+4+2,2-4) circle (0.4);
  \filldraw (2+8,2-8+2) circle (0.4);
%  \filldraw (2+8+2,2-8) circle (0.4);
  \filldraw (2-4,2+4+2) circle (0.4);
  \filldraw (2-4+2,2+4) circle (0.4);
%  \filldraw (2-8,2+8+2) circle (0.4);
  \filldraw (2-8+2,2+8) circle (0.4);
}
$$
\caption{
Here is an example of $W$ and $W'$ that are both $2$-periodic
weights; since $W'$ is also $1$-periodic, we have
$W\star_2 W'$ is also $1$-periodic.
$W,W'$ are both the sum of two simple $2$-periodic functions,
and $W\star_2 W'$ a sum of four.
Later we will see that since
$W'$ is not of
the form $W^*_\mec L$ for any $\mec L\in\integers^2$,
it follows that 
$\cF=\cM_{W,\mec 0}\otimes_{\cO_r}\cM_{W',\mec 0}$
cannot serve as a {\em canonical $\cO_r$-module} since
for any $\mec d$ (no matter how small), $\cG=\cF\otimes_{\cO_r}\cL_{r,\mec d}$
has $H^1(\cG)=0$.
}
\label{fi_tensor_product_cM_W_and_cM_W_prime_example}
\end{figure}

\begin{corollary}
Say that $W,W'$ from $\integers^2\to\integers_{\ge 0}$
have bounded support and are $r$-periodic.  Say that
$$
W = \sum_{q=1}^\ell {\rm Simple}_{r;i_q,j_q},
\quad
W' = \sum_{q=1}^{\ell'} {\rm Simple}_{r;i_q',j_q'}.
$$
Then for any $\mec d,\mec d'\in\integers^2$ we have
$$
\cM_{W,\mec d}\otimes_{\cO_r} \cM_{W',\mec d'} \isom \cM_{W'',\mec d+\mec d'},
$$
where
$$
W'' = \sum_{q=1}^\ell \sum_{q'=1}^{\ell'} 
{\rm Simple}_{r;i_q+i'_{q'},j_q+j'_{q'}}.
$$
\end{corollary}

\begin{corollary}
\label{co_r_periodic_convolution_gives_tensor_product_over_cO_r}
Say that $W,W'$ from $\integers^2\to\integers_{\ge 0}$
have bounded support and are $r$-periodic.
Define the {\em $r$-periodic convolution of $W$ and $W'$},
denoted $W\star_r W'$, to be the function
$\integers^2\to\integers_{\ge 0}$ give by
$$
W\star_r W'(\mec y) 
= \sum_{a_1=0}^{r-1}\sum_{a_2\in\integers} W(\mec a) W'(\mec y-\mec a).
$$
Then for any $\mec d,\mec d'\in\integers^2$ we have
$$
\cM_{W,\mec d}\otimes_{\cO_r}\cM_{W',\mec d'}
\isom
\cM_{W\star_r W',\mec d+\mec d'}.
$$
\end{corollary}

\begin{corollary}
Say that $W,W'$ are two $r$-periodic perfect matchings, with associated
bijections $\pi,\pi'$.
Then
$$
\cM_{W,\mec d}\otimes_{\cO_r} \cM_{W',\mec d'} \isom \cM_{W'',\mec d+\mec d'},
$$
where
$$
W'' = \sum_{q=0}^{r-1} \sum_{q'=0}^{r-1} 
{\rm Simple}_{r;q+q',\pi(q)+\pi'(q')}.
$$
\end{corollary}
Hence $\cM_{W'',\mec d+\mec d'}(B_3)$ is isomorphic to
$k[v^{\pm}]^{r^2}$.

\begin{remark}
\label{re_W_W_prime_tensor_produce_H_one_independent_of_cO}
If $W,W'$ are two perfect matchings that are not necessarily periodic,
we can define the tensor product of $k$-diagrams
$$
\cM_{W,\mec d}\otimes \cM_{W',\mec d'}
=\cM_{W,\mec d}\otimes_{\underline k} \cM_{W',\mec d'}.
$$
However, the $B_3$ value of this diagram is
$$
k^{\oplus W}\otimes k^{\oplus W'} ,
$$
and we easily see that this is isomorphic to $k^{\oplus W''}$ where
$W''=W*W'$ where 
$$
(W*W')(\mec s) \eqdef \sum_{\mec a\in\integers} W(\mec a) W'(\mec s-\mec a),
$$
and the sum above can be infinite.
(In fact, since $W,W'$ is of bounded support, $W*W'$ must be infinite.)
However, if $W,W'$ are $r$-periodic, then we can see that
\begin{equation}\label{eq_W_star_r_W_prime_positive_iff_asterisk}
\forall \mec s\in\integers^2,
\qquad
W\star_r W'(\mec s)>0 
\ \iff\ %
W* W'(\mec s)>0 .
\end{equation} 
In the next section we will see that $H^1(\cM_{W,\mec d})$ depends
only on the $\mec s$ where $W(\mec s)>0$, not on the particular value
of $W(\mec s)$ (even if $W(\mec s)=\infty$).
Hence it turns out that
$$
H^1(\cM_{W,\mec d}\otimes_{\underline k} \cM_{W',\mec d'})
=
H^1(\cM_{W,\mec d}\otimes_{\cO_r} \cM_{W',\mec d'}).
$$
Similarly if $W,W'$ are both $r$-periodic and $r'$-periodic, then
$$
H^1(\cM_{W,\mec d}\otimes_{\cO_r} \cM_{W',\mec d'})
=
H^1(\cM_{W,\mec d}\otimes_{\cO_{r'}} \cM_{W',\mec d'}).
$$
\end{remark}

\section{A Canonical $\cO$-Module}
\label{se_canonical}

Again, in this section $k$ will be a fixed field, and we will often
suppress $k$, writing, e.g., $\cO_r$ instead of $\cO_{r,k}$.

The basis for our duality theory will be an $\cO_r$-module, $\omega$, such that
\begin{enumerate}
\item $H^1(\omega)\isom k$; 
\item $\omega=\cM\otimes \cM'$ where $\cM=\cM_W$ and
$\cM_{W'}$ where $W$ is an arbitrary $r$-periodic perfect matching
and $W'$ is an appropriately chosen $r$-periodic perfect matching;
\item
$\omega$ allows up to set up a satisfactory ``duality theory.''
\end{enumerate} 
Although we leave (3) above vague for now, there is a rather striking
result for $\omega$ satisfying~(1) and~(2).

\begin{theorem}\label{th_H_one_vanishes_tensor_prod}
Let $W,W'$ be $r$-periodic perfect matchings.
If for all $\mec L\in\integers^2$ we have $W'\ne W^*_\mec L$,
then for all $\mec d,\mec d'\in\integers^2$ we have
$$
H^1( \cM_{W,\mec d} \otimes_{\cO_r} \cM_{W',\mec d'} ) = 0,
$$
\end{theorem}
This theorem will be proved in
Subsection~\ref{su_proof_of_first_main_theorem}.

This theorem tells us that the only possible $\omega$ as above come from
$W'$ of the form $W^*_\mec L$.
Moreover we have the following result.

\begin{theorem}\label{th_H_one_equals_k_in_rare_case}
Let $W$ be a periodic perfect matchings, and let $r$ be the
period of $W$ (i.e., the minimum $r\ge 1$ such that $W$ is $r$-periodic).
Let $\mec K\in\integers^2$, $\mec L=\mec K+\mec 1$, 
$W'=W^*_\mec L$, and let
$\omega_{W,\mec K}=\cM_{W,\mec 0} \otimes \cM_{W',\mec K}$.  Then
$$
b^1( \omega_{W,\mec K} ) = 1,
$$
and 
a basis for (the one-dimensional) $H^1(\omega)$ is given by
$$
(x^{L_1},0)=(0,-x_2^{L_2})\in H^1(\omega)
\isom \omega(A)/\omega(\partial).
$$
More generally, for any $\mec d,\mec d'\in \integers^2$ 
let $\mec K'=\mec d+\mec d'$;
we have that
$$
b^1( \cM_{W,\mec d} \otimes \cM_{W',\mec d'} ) 
$$
equals the size of the set
$$
Q=\{ q\in\integers\ | \ L_1+qr > K_1' \mbox{\ and\ } L_2-qr > K_2' \}.
$$
Morever:
\begin{enumerate}
\item
A basis for $H^1(  \cM_{W,\mec d} \otimes \cM_{W',\mec d'} )$
is the set of
$(x_1^{L_1+qr},0)$ ranging over all $q\in Q$;
and similarly with
$(0,x_2^{L_2-qr})$ over $q\in Q$.
\item
If $(x_1^a,0)\in\omega_W(A)$ and $a\ne L_1+qr$ with $q\in Q$, then
the image of $(x_1^a,0)$ in $H^1(  \cM_{W,\mec 0} \otimes \cM_{W',\mec K'} )$
is zero; similarly for $(0,x_2^a)$ with $a\ne L_2-qr$ with $q\in Q$.
\end{enumerate}
\end{theorem}
This theorem will be proven in 
Subsection~\ref{su_proof_of_second_main_theorem}.

\begin{remark}
Theorem~\ref{th_H_one_equals_k_in_rare_case} also shows that
$$
b^1( \cM_{W,\mec 0}\otimes \cM_{W',\mec K-\mec d}) = 1
$$
provided that $\mec d\in[r'-1,r'-1]$, where $r'$ is the period of $W$;
hence for $r'\ge 2$, there would be other candidates for our
canonical diagram (if our only criteria was that $b^1$ should equal $1$).
However, it turns out that $\mec d=\mec 0$ is the right choice
of canonical diagram for our duality theorem.
[Of course, we also have that 
$$
H^1( \cM_{W,\mec 0}\otimes \cM_{W',\mec K-\mec d'}) = 1
$$
whenever $\mec d'$ equals one of the $\mec d$ above plus a
multiple of $(r,-r)$.]
\end{remark}

% \begin{remark}
% In the above theorem, if $\mec d,\mec d'\in\integers^2$
% are such that $\mec d+\mec d'=\mec K'$, then
% $$
% \cM_{W,\mec d}\otimes_{\cO_r}\cM_{W',\mec d'}
% \isom
% \cM_{W,\mec 0}\otimes_{\cO_r}\cL_{\mec d}\otimes_{\cO_r}\cM_{W',\mec d'}
% \isom
% \cM_{W,\mec 0}\otimes_{\cO_r}\cM_{W',\mec d+\mec d'}
% \isom
% \cM_{W,\mec 0}\otimes_{\cO_r}\cM_{W',\mec K'}
% $$
% Hence Theorems~\ref{th_H_one_vanishes_tensor_prod}
% and~\ref{th_H_one_equals_k_in_rare_case} apply to
% $\cM_{W,\mec d}\otimes_{\cO_r}\cM_{W',\mec d'}$.
% \end{remark}

\begin{remark}
The Folinsbee-Friedman algorithm in
Subsection~\ref{su_folinsbee_friedman} below, along with
\eqref{re_W_W_prime_tensor_produce_H_one_independent_of_cO},
show that 
if $W,W'$ are $r$-periodic perfect matchings, then
$H^1(\cM_{W,\mec d}\otimes\cM_{W',\mec d})$ whether one takes
the tensor product as $\cO_r$-modules or as $\underline k$-modules
(i.e., $k$-diagrams).
\end{remark}

\begin{definition}
Let $W$ be an $r$-periodic perfect matchings.
Let $\mec K\in\integers^2$ and $\mec L=\mec K+\mec 1$.
Then the {\em canonical $\cO$-module} on $W$ and $\mec K$ refers to
$$
\omega_{W,\mec K} =\omega_{W,\mec K,r}
\eqdef \cM_{W,\mec 0} \otimes_{\cO_r}
\cM_{W^*_\mec L,\mec K}.
$$
\end{definition}

We alert the reader that the only result we need in
Sections~\ref{se_duality} and~\ref{se_strong_duality}
is that
$$
H^1(\omega_{W,\mec K})\isom k.
$$
Hence Theorem~\ref{th_H_one_vanishes_tensor_prod} will not be needed there.
(The purpose of this theorem is to explain the special nature of
the canonical $\cO_r$-module.)
(However, the proofs of Theorem~\ref{th_H_one_vanishes_tensor_prod}
and~\ref{th_H_one_equals_k_in_rare_case} are both based on the same
fundamental 
Lemma~\ref{le_fundamental_two_bijections}.)

\begin{remark}
Since $H^1(\omega_{W,\mec K})$ is generated $(x^{L_1},0)$,
we see that $\omega_{W,\mec K}$ generally depends on $\mec K$.
However, we easily that for any $\mec K'\in\integers$ 
and $\mec L'=\mec K+\mec 1$ there is an
isomorphism 
$$
\phi\from\cM_{W^*_\mec L,\mec K} \xrightarrow{\quad\isom\quad}
\cM_{W^*_{\mec L'},\mec K'} ;
$$
namely for $i=1,2$, $\phi(B_i)$ takes $y_i^a$ to $y_i^{a+K_i-K_i'}$,
$\phi(A_i)$ takes $x_i^a$ to $x_i^{a-K_i'+K_i}$,
and $\phi(B_3)$ takes $\mec d$ to $\mec d+\mec L'-\mec L$;
to check that $\phi$ is an isomorphism we use the facts that
$$
W^*_\mec L(\mec d)=W(\mec L-\mec d)=W(\mec L'+\mec L-\mec L'-\mec d)
=
W^*_{\mec L'}(\mec d+\mec L'-\mec L)
$$
and that $\mec L'-\mec L=\mec K'-\mec K$.
It follows that $\omega_{W,\mec K}$ and $\omega_{W,\mec K'}$
are isomorphic.
In Subsection~\ref{su_omega_depends_on_W}
we give some remarks about how 
the isomorphism class of $\omega_{W,\mec K}$ is independent of $\mec K$
but might depend on $W$.
\end{remark}

In the next two subsections we prove the above theorems.
Recall \cite{folinsbee_friedman_euler} that if $W,W'$ are any
$r$-periodic perfect matchings, then
$$
\cM_W \otimes \cM_{W'} = \cM_{W''}
$$
where $W''$ is an $r$-periodic function $\integers^2\to\integers_{\ge 0}$
that can be written as the sum
$$
W'' \isom W_1+ \cdots + W_r
$$
where $W_1,\ldots,W_r$ are $r$-periodic perfect matchings.
As such it is easy to see that for $r\ge 2$ we have
$$
H^0(\cM_{W''})=\infty.
$$
Hence these $\cO$-modules are more chaotic than their counterparts
in algebraic geometry.

\begin{remark}
We also remark that for any $W,W'$,
$\gamma=\cM_W\otimes\cM_{W'}$ has the same values as
$\cM_{W,\mec d}$ (and as $\cO$) except at $B_3$; 
and 
$\gamma(B_3)=k[v^{\pm}]^{\oplus r^2}$
and $\cM(B_3)=k[v^{\pm}]^{\oplus r}$.
A crude dimension count suggests that the image of $\cM(B)$ in $\cM(A)$
is roughly 
``of the same size'' as $\cM(A)$:
indeed, the image of $\cM(B_3)$ in $\cM(A)$ gives a bijection between
$\cM(A_1)$ and $\cM(A_2)$, both isomorphic to $k^{\oplus\integers}$;
for $i=1,2$, the image of $\cM(B_i)$ in $\cM(A_i)$ is
$k^{\oplus \integers_{\le d_i}}$; hence the image of $\cM(B)$ in
$\cM(A)$ is roughly one copy of $k^{\oplus\integers}$ from $\cM(B_3)$
and roughly half a copy of $k^{\oplus\integers}$ for each $\cM(B_i)$.
However, for $r\ge 2$, since
$\gamma$ has the same values as $\cM_{W,\mec d}$ except that it has
``$(r-1)$ more copies of $k^{\oplus\integers}$ at $B_3$,''
it seems clear that $H^0(\gamma)=\infty$ and likely that
$H^1(\gamma)$ should equal $0$.
Hence one might expect $H^1(\gamma)$ to be $0$.
For this reason Theorem~\ref{th_H_one_equals_k_in_rare_case} seems
quite remarkable to us.
\end{remark}

\subsection{The Algorithm of Folinsbee-Friedman
\cite{folinsbee_friedman_euler} To
Compute $H^i(\cM_{W,\mec d})$}
\label{su_folinsbee_friedman}

To prove the main theorems in this section, we review the
method of Folinsbee-Friedman \cite{folinsbee_friedman_euler} to compute
$H^i(\cM_{W,\mec d})$ where 
$W\from\integers^2\to\integers_{\ge 0}\cup\{\infty\}$ 
is an arbitrary function.  
This is stated as Theorem~4.3 there, although the basis for
$H^1(\cM_{W,\mec d})$, which we need in this article,
is only apparent from the proof of Theorem~4.3, not its statement.

Recall that by definition,
$H^1(\cM_{W,\mec d})$ is the cokernel of the map
$$
\cM_{W,\mec d}(\partial) \from
\cM_{W,\mec d}(B_1)\oplus
\cM_{W,\mec d}(B_2)\oplus
\cM_{W,\mec d}(B_3)\to
\cM_{W,\mec d}(A_1)\oplus
\cM_{W,\mec d}(A_2),
$$
which, setting $\tau=\cM_{W,\mec d}(\partial)$, is the map
\begin{equation}\label{eq_tau_the_map_of_cM_W_prime_prime_partial}
\tau\from 
k^{\oplus\integers_{\le d_1}}\oplus
k^{\oplus\integers_{\le d_2}}\oplus
k^{{\rm Multi}(W)}
\to
k^{\oplus\integers}\oplus
k^{\oplus\integers}
\isom k^{\oplus\integers\times\{1,2\}}.
\end{equation} 
% $$
% \tau\from k^{\oplus( 
% \integers_{\le d_1}\,\amalg\,
% \integers_{\le d_2}\,\amalg\,
% {\rm Multi}(W))}
% \to
% k^{\oplus(\integers\,\amalg\,\integers)}.
% $$

\begin{theorem}
\label{th_folinsbee_friedman_algorithm}
Let 
$W\from\integers^2\to\integers_{\ge 0}\cup\{\infty\}$ 
be an arbitrary function.
Let
% according to Theorem~4.7 and its proof, we have the following algorithm:
\begin{enumerate}
\item 
$G'$ be the (bipartite) graph whose vertex set is $\integers\times\{1,2\}$
such that the vertices $(s_1,1)$ and $(s_2,2)$ are joined by $W(s_1,s_2)$
edges (this is therefore a multigraph);
\item
let $G$ be the graph obtained from $G'$ by collasping the vertices of
$\integers_{\le d_1}\times\{1\}$ and
$\integers_{\le d_2}\times\{2\}$ into a singe vertex $v_0$
(keeping as a self-loop any edge from
$\integers_{\le d_1}\times\{1\}$ to
$\integers_{\le d_2}\times\{2\}$).
\end{enumerate}
(Hence $G$ may have self-loops and is not generally bipartite.)
Also set $V_{\rm first},V_{\rm second}$ to be the subsets
$\integers_{\ge d_1+1}\times\{1\}$ and
$\integers_{\ge d_2+1}\times\{2\}$, so that $V$ is partitioned as
$\{v_0\}\amalg V_{\rm first}\amalg V_{\rm second}$.
Then
\begin{align}
\label{eq_betti_one_and_betti_zero_G}
b^1(\cM_{W,\mec d})& =b^0(G)-1 \\
\label{eq_betti_zero_and_betti_one_G}
b^0(\cM_{W,\mec d})& =b^1(G)  .
\end{align}
Moreover, let $\{G_i\}_{i\in I}$ be the connected components of $G$,
where $0\in I$ and $G_0$ is the connected component of $G$ containing
$v_0$; set $I'=I\setminus\{0\}$.
(Hence $G_i$ with $i\in I'$ are precisely the connected components of
$G'$ whose vertex set is disjoint from
$\integers_{\le d_1}\times\{1\}\cup
\integers_{\le d_2}\times\{2\}$.)
For each $i\in I'$, and $G_i=(V_i,E_i)$,
choose some $v_i\in V_i$.
Then 
$$
H^1(\cM_{W,\mec d}) \isom \bigoplus_{i\in I'} H^0(G_i)
\isom k^{\oplus I'},
$$
and a basis for $H^1(\cM_{W,\mec d})$ is 
$\{\mec e_{v_i}\}_{i\in I'}$,
the standard basis vectors for $k^{\oplus \integers\times\{1,2\}}$.
\end{theorem}
We depict $G'$ and $G$ in 
Figure~\ref{fi_G_G_prime_from_W_prime_prime}.

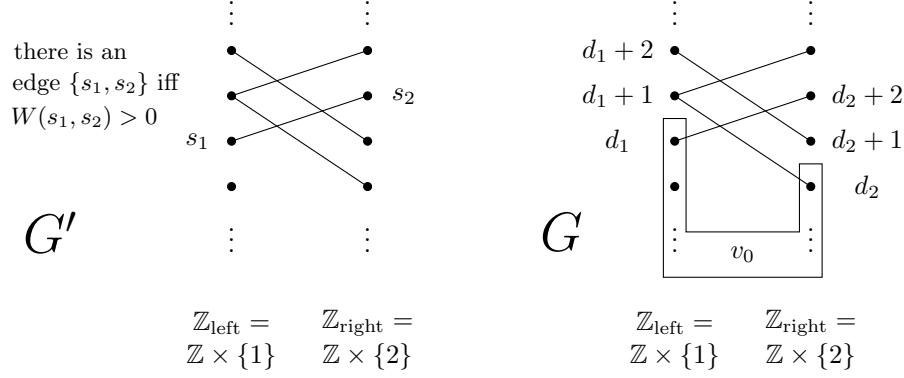
\begin{figure}
$$
\begin{tikzpicture}[scale=0.30]
\foreach \i in {-1,...,2} {
  \filldraw(0,\i*2) circle (5pt);
  \filldraw(6,\i*2) circle (5pt);
}
\node at (0,6) {$\vdots$};
\node at (6,6) {$\vdots$};
\node at (0,-4) {$\vdots$};
\node at (6,-4) {$\vdots$};
\node at (0,-8) {$\integers_{\rm left}=$};
\node at (0,-9.5) {$\integers\times\{1\}$};
\node at (6,-8) {$\integers_{\rm right}=$};
\node at (6,-9.5) {$\integers\times\{2\}$};
\node at (-1.5,0) {$s_1$};
\node at (7.6,2) {$s_2$};
\draw (0,0) -- (6,2);
\node at (-7.1,4) {\small there is an};
\node at (-6,2.5) {\small edge $\{s_1,s_2\}$ iff};
\node at (-6.5,1) {\small $W(s_1,s_2)>0$};
\node at (-8,-4) {\Huge $G'$};
\draw (0,2) -- (6,-2);
\draw (0,2) -- (6,4);
\draw (0,4) -- (6,0);
\end{tikzpicture}
\qquad\qquad
\begin{tikzpicture}[scale=0.30]
\foreach \i in {-1,...,2} {
  \filldraw(0,\i*2) circle (5pt);
  \filldraw(6,\i*2) circle (5pt);
}
\node at (0,6) {$\vdots$};
\node at (6,6) {$\vdots$};
\node at (0,-4) {$\vdots$};
\node at (6,-4) {$\vdots$};
\node at (0,-8) {$\integers_{\rm left}=$};
\node at (0,-9.5) {$\integers\times\{1\}$};
\node at (6,-8) {$\integers_{\rm right}=$};
\node at (6,-9.5) {$\integers\times\{2\}$};
% \node at (-1.5,0) {$s_1$};
% \node at (7.6,2) {$s_2$};
\draw (0,0) -- (6,2);
% \node at (-6,1) {\small if $W(s_1,s_2)>0$};
\node at (-5,-4) {\Huge $G$};
\draw (0,2) -- (6,-2);
\draw (0,2) -- (6,4);
\draw (0,4) -- (6,0);
\draw (-0.5,1) -- (0.5,1) -- (0.5,-4) -- (5.5,-4) -- (5.5,-1) -- (6.5,-1) -- (6.5,-6) -- (-0.5,-6) -- (-0.5,1);
\node at (3,-5) {$v_0$};
\node at (-2.5,0) {$d_1$};
\node at (-2.5,2) {$d_1+1$};
\node at (-2.5,4) {$d_1+2$};
\node at (8.5,-2) {$d_2$};
\node at (8.5,0) {$d_2+1$};
\node at (8.5,2) {$d_2+2$};
\end{tikzpicture}
$$
\caption{The graph $G'=(V',E')$: $V'$ consists of two copies of $\integers$,
one on the left, one on the right, plus $W(s_1,s_2)$ edges between 
$s_1\in\integers_{\rm left}$ and $s_2\in\integers_{\rm right}$.
$G$ is obtained from $G'$ by collasping the left vertices $\le d_1$
and the right vertices $\le d_2$ into a single vertex, $v_0$.
Therefore $G$ generally has self-loops about $v_0$, and $G$ is not
generally bipartite.
To compute 
$H^1(\cM_{W,\mec d})$ with $G'$ and $G$, 
multiple edges have no effect; so it is
enough to put an edge from $s_1$ on the left to $s_2$ on the right
when $W(s_1,s_2)>0$.
}
\label{fi_G_G_prime_from_W_prime_prime}
\end{figure}

\begin{proof}
Theorem~4.3 of \cite{folinsbee_friedman_euler} explicitly
describes the formulation of $G$ and $G'$, and states
\eqref{eq_betti_one_and_betti_zero_G}
and \eqref{eq_betti_zero_and_betti_one_G}.
The rest can be inferred from the proof of Theorem~4.3: let us review
the main
point.
Because for $i=1,2$ we have
$$
\coker\bigl(\cM_{W,\mec d}(\rho_{ii}) \bigr)
\isom\cM_{W,\mec d})(A_i)/\cM_{W,\mec d})(B_i)
\isom k^{\oplus\integers_{\ge d_i+1}},
$$
and because $\cM_{W,\mec d}(B_3)=k^{{\rm Multi}(W)}$, it follows that
the cokernel of $\tau$ above is the same as the cokernel
of the map
$$
\tilde\tau\from
% \cM_{W,\mec d}(B_3)=
k^{{\rm Multi}(W)}
\to
k^{\oplus\integers_{\ge d_1+1}}\oplus
k^{\oplus\integers_{\ge d_2+1}}
$$
each copy of $k$ coming from
a pair $(s_1,s_2)$ with $W(s_1,s_2)\ne 0$
is sent to $(-\mec e_{s_1},-\mec e_{s_2})$,
if $s_1\ge d_1+1$ and $s_2\ge d_2+1$,
but is sent to $0$ otherwise.
But this is the same as the map
$$
\bigoplus_{i\in I'} k^{\oplus E_i}
\to
\bigoplus_{i\in I'} k^{\oplus V_i}
$$
where each $e\in E_i$ corresponding to an $(s_1,s_2)$ with $W(s_1,s_2)>0$
is taken to $(-\mec e_{s_1},-\mec e_{s_2})$.
But each $G_i=(V_i,E_i)$ with $i\in I'$ is also bipartite, so
by scaling this map by negating the $V_i\cap \integers_{\ge d_1+1}$
vertices, this is equivalent to the map taking 
$e\in E_i$ as above to $(\mec e_{s_1},-\mec e_{s_2})$;
hence we get a direct sum of indicence matrices for $G_i$.
\end{proof}

\begin{notation}
In the formulation of $G'$ and $G$ above from $W$, we will use 
$\integers_{\rm left},\integers_{\rm right}$ to denote
$\integers\times\{1\},\integers\times\{2\}$, respectively,
and refer to vertices in $\integers_{\rm left}$ as {\em left vertices}
and similarly for $\integers_{\rm right}$.
\end{notation}

\begin{remark}
In computing $H^1(\cM_{W,\mec d})$ from $G$, we can replace any
multiple edge of $G$ with a single edge, since this doesn't change
the connected components of $G$.
\end{remark}

\begin{remark}
In Theorem~\ref{th_folinsbee_friedman_algorithm}, may be simpler to
think of $I'$ and the $G_i$ with $i\in I'$ as representing
the connected components of $G'$ that have no vertices
in $\integers_{\le d_1}$ on the left and none in
$\integers_{\le d_2}$ on the right.
This is because $G'$ is a simpler graph than $G$.
In fact, we will state our fundamental lemma in terms of $G'$.
\end{remark}

The following proposition is straightforward; it is involved in
how we state our fundamental lemma.

\begin{proposition}
\label{pr_equivalent_forms_of_left_right_connectivity}
Say that $W,W'$ are $r$-periodic
perfect matchings whose associated permutations
are $\pi,\pi'$.
Let $W''=W\star_r W'$.
\begin{enumerate}
\item
For all $\mec s\in\integers^2$, 
the following are equivalent:
\begin{enumerate}
\item
$W''(\mec s)>0$;
\item
there exists $\mec a\in\integers^2$
such that $W(\mec a)>0$ and $W'(\mec s-\mec a)>0$; and
% \item
% there exists $a_1\in\integers$ such that
% the unique $a_2$ satisfying $W(a_1,a_2)>0$ also satisfies
% $W'(s_1-a_1,s_2-a_2)>0$;
% \item
% there exists $a_1\in\integers$ such that
% $a_2=\pi(a_1)$ also satisfies
% $W'(s_1-a_1,s_2-a_2)>0$;
% \item
% there exists $a_2\in\integers$ such that
% the unique $a_1$ satisfying $W(a_1,a_2)>0$ satisfies
% $W'(s_1-a_1,s_2-a_2)>0$; and
\item
there exist $a,a'\in\integers$ such that
$$
s_1 = a+a', \quad
s_2 = \pi(a)+\pi'(a').
$$
\end{enumerate}
\item
For any $L_1,L_2\in\integers$, the following are equivalent:
\begin{enumerate}
\item
$\forall a\in\integers$, $\pi(a)+\pi'(L_1-a)=L_2$;
\item
$W''(L_1,s_2)>0$ implies that $s_2=L_2$;
\item
$W''(s_1,L_2)>0$ implies that $s_1=L_1$;
and
\item
$G'$ in Theorem~\ref{th_folinsbee_friedman_algorithm}
with $W=W''$ has a connected component consisting whose vertex set
consists entirely of
vertex $L_1$ on the left and vertex $L_2$ on the right.
\end{enumerate}
\end{enumerate}
\end{proposition}
\begin{proof}
(1): 
Corollary~\ref{co_r_periodic_convolution_gives_tensor_product_over_cO_r}
shows that $W''(\mec s)>0$ iff for some $a_1=0,\ldots,r-1$ and
$a_2\in\integers$ we have 
\begin{equation}\label{eq_W_at_a_positive_W_prime_at_etc}
W(\mec a)>0 \quad \mbox{and} \quad W'(\mec s- \mec a)>0.
\end{equation} 
The latter implies the (1b).
Now assume (1b), 
i.e., that
$\tilde a_1\in\integers$ and $\tilde a_2\in\integers$ satisfy
$$
W(\tilde a_1,\tilde a_2)>0
\quad\mbox{and}\quad
W(s_1-\tilde a_1,s_2-\tilde a_2)>0
$$
then for some $q\in\integers$ we have $0\le \tilde a_1-qr \le r-1$.
Since $W,W'$ are $r$-periodic, the latter implies that
$$
W(\tilde a_1-qk,\tilde a_2+qk)>0, 
\quad\mbox{and}\quad
W'(s_1-\tilde a_1+qk,s_2-\tilde a_2-qk)>0
$$
which implies 
\eqref{eq_W_at_a_positive_W_prime_at_etc}
for $a_1=\tilde a_1-qr$ and $a_2=\tilde a_2+qk$,
which makes $0\le a_1\le r-1$.
Hence (1a) and (1b) are equivalent.

% (1b) and~(1c) are equivalenet because for each $a_1\in\integers$
% there is a unique $a_2\in\integers$ such that $W(a_1,a_2)$>0
% (since $W$ is a perfect matching).
% 
% Similarly, (1b) and (1d) are equivalent.

Now we show that~(1b) and~(1c) are equivalenet: (1b) is equivalent to
$a_2=\pi(a_1)$ and $s_2-a_2=\pi'(s_1-a_1)$;
and setting $a_1'=s_1-a_1$, this equivalent the existence of $a_1,a_1'$
such that
$$
s_1=a_1+a_1' \quad\mbox{and}\quad
s_2=a_2 +\pi'(s_1-a_1);
$$
but in these last equations
$$
s_2=a_2+\pi'(s_1-a_1) = \pi(a_1)+\pi'(a_1') .
$$
Hence~(1b) and~(1c) are equivalent.

Hence (1a)--(1c) are equivalent.

(2): 
Let us show that (2b) and (2a) are equivalent, by showing that
their negations are equivalent:
the negation of (2b) is that there exists $s_2\ne L_2$ such that
$W''(L_1,s_2)>0$,
which by~(1) is equivalent to there exist $a,a'\in\integers$ such that
$L_1=a+a'$ and $\pi(a)+\pi(a')=s_2\ne L_2$; since $a'$ is determined
as $a'=L_1-a$, this is equivalent to the existence of
$a$ such that $\pi(a)+\pi'(L_1-a)\ne L_2$.
But this is the negation of (2a); hence
(2a) and (2b) are equivalent.

Similarly, the negation of (2c) is equivalent to the existence of
$a,a'\in\integers$ such that
$a+a'\ne L_1$ and $\pi(a)+\pi'(a')=L_2$.
But this implies $\pi(a)+\pi'(L_1-a)\ne L_2$, since $\pi'$ is a bijection.
Hence, by~(1c), $W(L_1,s_2)>0$ where
$$
s_2 = \pi(a)+\pi'(L_1-a) \ne L_2.
$$
Hence the negation of (2c) implies the negation of (2b).
By the symmetric argument, the negation of (2b) implies the
negation of (2c).
Hence (2b) and (2c) are equivalent, which are therefore also equivalent 
to (2a).

(2b) and (2c) are equivalent to (2d), by the definition of $G'$.
Hence (2a)--(2d) are equivalent.
\end{proof}

\subsection{The Fundamental Lemma}

The discussion in the last subsection 
motivates the fundamental lemma we now give.

\begin{lemma}\label{le_fundamental_two_bijections}
Let $\pi,\pi'$ be permutations of bounded support
(i.e., for some constants $C,C'$ we have
$|\pi(a)+a|\le C$ and $|\pi'(a)+a|\le C'$ for all $a\in\integers$).
Let $G'$ be the following bipartite graph:
\begin{enumerate}
\item
its vertex set is $V=\integers_{\rm left}\amalg\integers_{\rm right}$, 
where $\integers_{\rm left},\integers_{\rm right}$ are each a copy
of $\integers$;
\item 
each $a,a'\in\integers$ there is an edge joining 
$a+a'$ on the left (i.e., $a+a'\in\integers_{\rm left}$) to
$\pi(a)+\pi'(a')$ on the right (i.e., in $\integers_{\rm right}$)
\end{enumerate}
(multiple edges are unimportant here).
Then for any $L_1\in\integers$, either
\begin{enumerate}
\item
all edges from $L_1$ on the left are incident upon a single vertx $L_2$ on 
the right (in which case $\pi(x)+\pi'(L_1-x)=L_2$, i.e., if $W,W'$
are the weights associated to $\pi,\pi'$, then $W'=W_\mec L^*$), or
\item
for any $k\in\integers$, there is a path from $L_1$ on the left 
to some left or right vertex
whose value is at most $k$.
\end{enumerate}
\end{lemma}
\begin{proof}
Say that the hypotheses~(1) and~(2) hold, but for some $L_1$ neither~(1) 
nor~(2) holds; we will derive a contradiction.
So let $S_1$ be all the left vertices in the connected component
of $L_1$ in $G'$, and $S_2$ all the right vertices.
Then $S_1$ must have a minimum element; but since $\pi,\pi'$ are
of bounded support, then $S_1$ must have a maximum element, for 
otherwise $L_1$ has a path to an arbitrarily large left vertex and
therefore a path to an arbitrarily small right vertex.
Let $L_1^{\rm min}$ and $L_1^{\rm max}$ be the minimum and maximum
values of $S_1$; 
similarly $S_2$ has a minimum and a maximum value which we respectively
denote $L_2^{\rm min}$ and $L_2^{\rm max}$.

Let 
$$
\Delta_1 = L_1^{\rm max}-L_1^{\rm min},\quad
\Delta_2 = L_2^{\rm max}-L_2^{\rm min};
$$
let us prove that $\Delta_1=\Delta_2$:
we have that for any $a\in\integers$
$$
\pi(a)+\pi'(L_1^{\rm min} - a) \le L_2^{\rm max},
$$
and
$$
\pi(a)+\pi'(L_1^{\rm max}-a)\ge L_2^{\rm min}.
$$
Subtracting these two equations gives
$$
\forall a\in\integers,
\quad 
\pi'(L_1^{\rm max}-a) - \pi'(L_1^{\rm min} - a) \ge -\Delta_2 ;
$$
substituting $b=L_1^{\rm min} - a$, we have that
$$
\forall b\in\integers,
\quad 
\pi'(b+\Delta_1) - \pi'(b) \ge -\Delta_2 .
$$
Taking $b$ to be 
$0,\Delta_1,\ldots,(\ell-1)\Delta_1$ for some $\ell\in\naturals$ and adding,
we have
$$
\pi'(\ell \Delta_1)-\pi'(0) \ge -\ell \Delta_2 .
$$
Since $\pi'$ is of bounded support, we have
$$
-\ell\Delta_1 - \pi'(\ell \Delta_1) \ge C'
$$
for some constant $C'$
(depending only on $\pi'$).  Adding the last two inequalities we get
$$
-\ell\Delta_1 -\pi'(0) \ge -\ell\Delta_2 + C'.
$$
Dividing by $\ell$ and
taking $\ell\to\infty$ we get $-\Delta_1\ge-\Delta_2$
or $\Delta_1\le \Delta_2$.
The symmetric argument shows that $\Delta_2\le \Delta_1$, and
hence $\Delta_1=\Delta_2$.

Now let $\Delta=\Delta_1=\Delta_2$;
since conclusion~(1) of the lemma does not hold, we have $\Delta>0$.
Let us show that this is impossible.

By assumption, $L_1^{\rm max}$ on the left is connected only to
$\ell$ on the right with $\ell\ge L_2^{\rm min}$, i.e., 
$$
\forall y\in\integers,
\quad
\pi(y)+\pi'(L_1^{\rm max}-y) \ge L_2^{\rm min}.
$$
Letting $x=y-\Delta$, we have 
\begin{equation}\label{eq_L_1_max_imposs}
\forall x\in\integers,
\quad
\pi(x+\Delta)+\pi'(L_1^{\rm min}-x) \ge L_2^{\rm min}.
\end{equation} 
Since $L_1^{\rm min}$ only has right vertices somewhere between
$L_2^{\rm min}$ and $L_2^{\rm max}$, if
$$
X_0=\{ x \ | \ \pi(x)+\pi'(L_1^{\rm min}-x) \le  L_2^{\rm max} - 1 \}
$$
then 
\begin{enumerate}
\item
$x\in X_0$ implies that
$$
\pi(x)+\pi'(L_1^{\rm min}-x) \le  L_2^{\rm max} - 1
$$
and hence, in view of \eqref{eq_L_1_max_imposs},
$$
\pi(x)-\pi(x+\Delta) \le -\Delta-1;
$$
and
\item
otherwise, 
\begin{equation}\label{eq_pi_and_pi_prime_connect_L_one_to_L_two}
\forall x\notin X_0,
\quad\pi(x)+\pi'(L_1^{\rm min}-x) = L_2^{\rm max}.
\end{equation}
and hence
$$
\forall x\notin X_0,
\quad
\pi(x)-\pi(x+\Delta) = -\Delta.
$$
\end{enumerate}
We aim to show that $X_0=\emptyset$.

First, let us show that $|X_0|$ is finite:
for any $a,B\in\naturals$,
if 
$$
s=s_{a,B} 
= \Bigl| \{a-B\Delta,\ a-(B-1)\Delta, \ldots, a+B(-1)\Delta\}\cap X_0 \Bigr|
$$
then
$$
\pi(a-B\Delta)-\pi(a+B\Delta) \le -(2B)\Delta-s.
$$
Hence
$$
s\le -(2B)\Delta + \pi(a+B\Delta) - \pi(a-B\Delta).
$$
Since $\pi$ is of bounded support, the right hand side above is bounded
for any fixed $a$ and $B\to\infty$;
but 
$$
s_{a,+\infty} 
\eqdef \lim_{B\to\infty} s_{a,B}
= \{ x\in X_0 \ | \ x\equiv a \pmod{\Delta} \},
$$
and each is bounded for fixed $a$.
Taking $a=0,1,\ldots,B-1$, we have
$$
|X_0| = s_{0,+\infty} + \cdots s_{B-1,+\infty}
$$
which is bounded.
Hence $X_0$ is finite.

Now we prove that $X_0$ is empty:  indeed, let
$$
\tilde\pi(y)=\pi(y+\Delta)+\Delta.
$$
Since $\pi$ is a permutation, so is $\tilde\pi$.  But
$\tilde\pi(x)=\pi(x)$ for $x\notin X_0$, and 
$\tilde\pi(x)\le\pi(x)-1$ for $x\in X_0$.
Hence, since $X_0$ is finite,
$$
\sum_{x\in X_0} \tilde\pi(x) \le -|X_0|+\sum_{x\in X_0}\pi(x) ;
$$
and so
\begin{equation}\label{eq_X_naught_bound}
|X_0| \le \sum_{x\in X_0}\pi(x) - \sum_{x\in X_0} \tilde\pi(x).
\end{equation} 
But
since $\tilde\pi$ and $\pi$ agree precisely on the complement of $X_0$,
we have 
$$
\tilde\pi(\integers\setminus X_0) = \pi(\integers\setminus X_0)
$$
(an equality of sets), and therefore we have the equality of sets
$$
\tilde\pi(X_0) = \pi(X_0).
$$
Hence
$$
\sum_{x\in X_0} \tilde\pi(x) = \sum_{x\in X_0}\pi(x),
$$
and hence \eqref{eq_X_naught_bound} implies that $|X_0|=0$.

Hence $X_0$ is empty.
Then \eqref{eq_pi_and_pi_prime_connect_L_one_to_L_two} implies that
$$
\forall x\in\integers,
\quad
\quad\pi(x)+\pi'(L_1^{\rm min}-x) = L_2^{\rm max}.
$$
Hence it follows that all edges from $L_1^{\rm min}$ on the left
join $L_2^{\rm max}$ on the right.
Hence $L_1$ can only be connected to $L_1^{\rm min}$ and $L_2^{\rm max}$
if $L_1=L_1^{\rm min}$.
Hence $\Delta=0$, which is precisely conclusion~(1) of the
lemma.
(Which is a contradiction, since we assumed both conclusions~(1)
and~(2) do not hold.)
\end{proof}

\subsection{An Easy Consequence of the Fundamental Lemma}

To use
Lemma~\ref{le_fundamental_two_bijections}
to prove Theorem~\ref{th_H_one_equals_k_in_rare_case}, we will need
some easy facts; we will combine them with
Lemma~\ref{le_fundamental_two_bijections} in the statement of
Lemma~\ref{le_consequence_of_fundamental_lemma} below.

\begin{definition}
Let $\pi$ be a bijection $\integers\to\integers$ of bounded support.
For any $r\in\integers$, we say that $\pi$ is $r$-periodic if
$$
\forall x\in\integers,\quad
\pi(x+r) = \pi(x) - r.
$$
The {\em period} of $\pi$ is the smallest $r\ge 1$ such that
$\pi$ is $r$-periodic; if no such $r$ exists, then we say that
$\pi$ is {\em unperiodic}.
\end{definition}

\begin{lemma}
\label{le_pi_and_pi_prime_have_the_same_period}
Let $\pi,\pi'$ be perfect matchings of bounded support.
Say that for some $L_1,L_2\in\integers$ and all $x\in\integers$ we have
$$
\pi(x)+\pi'(L_1-x) = L_2.
$$
Then either $\pi,\pi'$ are unperiodic, or they have the same period.
\end{lemma}
\begin{proof}
Say that $\pi$ is $r$-periodic, i.e.,
\begin{equation}\label{eq_condition_pi_is_r_periodic_in_same_period_proof}
\forall x\in\integers,\quad
\pi(x+r)=\pi(x)-r.
\end{equation} 
Then the assumption that $\pi(x)=L_2-\pi'(L_1-x)$ for all $x$ implies
$$
\forall x\in\integers,\quad
\pi'(L_1-x+r) = L_2 - \pi(x-r)=L_2 - \pi(x)-r;
$$
since $\pi(x)=L_2-\pi'(L_1-x)$, the RHS above becomes
$$
L_2 - \bigl(L_2-\pi'(L_1-x) \bigr) - r = \pi'(L_1-x)-r;
$$
setting $y=L_1-x$, it follows that 
\begin{equation}\label{eq_pi_prime_in_y_is_r_periodic_if_pi_is}
\forall y\in\integers,\quad
\pi'(y+r) = \pi'(y)-r,
\end{equation} 
and therefore $\pi'$ is $r$-periodic.
Arguing backwards, we see that
\eqref{eq_pi_prime_in_y_is_r_periodic_if_pi_is} implies
\eqref{eq_condition_pi_is_r_periodic_in_same_period_proof}; hence
if $\pi'$ is $r$-periodic, then so is $\pi$.
Hence $\pi$ and $\pi'$ are $r$-periodic for the exact same set of
$r\in\integers$.
It follows that if either $\pi$ or $\pi'$ is periodic, then they have
the same period, and if one of them is unperiodic then so is the other.
\end{proof}

\begin{lemma}
\label{le_the_Ls_such_that_pi_x_plus_pi_prime_etc}
Let $\pi,\pi'$ be perfect matchings of bounded support.
Say that $\pi$ has period $r\ge 1$, and that for some $L_1,L_2\in\integers$
we have
$$
\forall x, \quad 
\pi(x)+\pi'(L_1-x) = L_2
$$
Then we have that for all $L_1',L_2'\in\integers$ the following are 
equivalent:
\begin{enumerate}
\item 
\begin{equation}\label{eq_pi_pi_prime_related_by_L_one_prime_L_two_prime}
\forall x, \quad 
\pi(x)+\pi'(L_1'-x) = L_2'
\end{equation} 
\item 
For some $q\in\integers$ we have $L_1'=L_1+qr$ and $L_2'=L_2-qr$.
\end{enumerate}
\end{lemma}
\begin{proof}
(2) easily implies (1), since if $\pi$ is $r$-periodic, then so is
$\pi'$, and then for any $q\in\integers$ we have
$$
\pi(x)+\pi'(L_1'-x)=
\pi(x)+\pi'(L_1+qr-x) = \pi(x)+\pi'(L_1-x)-qr = L_2-qr = L_2'.
$$
Hence it suffices to show~(1) implies~(2).

If \eqref{eq_pi_pi_prime_related_by_L_one_prime_L_two_prime} holds and
$L_1'\ne L_1+kr$ for some $k\in\integers$, then we have
$L_1+kr < L_1' < L_1+kr+r$ for some $k\in\integers$;
set $r'=L_1'-L_1-kr$, so $1\le r'<r$.
Then, using the $r$-periodicity of $\pi'$, we have
$$
\forall x, \quad 
\pi(L_1-x)-L_2
=
\pi'(L_1'-x) - L_2' = \bigl( \pi'(L_1'-x-kr)+kr \bigr) - L_2'
= \pi'(L_1+r'-x) - L_2'+kr.
$$
Setting $y=L_1'-x$, we have for all $y$,
$$
\pi'(y)-\pi'(r'+y)=C
$$
for some constant $C$.  Now we claim that $C=r'$, for applying the
equation above to $y,y+r',\ldots,y+r'(a-1)$ for some $a\in\naturals$
and adding we get
$$
\pi'(y)-\pi'(ar'+y)=Ca;
$$
taking $a\to\infty$, and using the fact that $\pi'$ is bounded, we get
that $C=-r'$.
But then
$$
\forall y\in\integers,
\quad
\pi'(y)-\pi'(r'+y)=-r',
$$
and so $\pi'$ is $r'$-periodic for some $1\le r'<r$, and hence so
is $\pi$, which is impossible.
\end{proof}

Combining 
Lemmas~\ref{le_pi_and_pi_prime_have_the_same_period}
and~\ref{le_the_Ls_such_that_pi_x_plus_pi_prime_etc}
we immediately get an important conclusion.

\begin{lemma}
\label{le_consequence_of_fundamental_lemma}
Say that in Lemma~\ref{le_fundamental_two_bijections},
$\pi$ has period $r\ge 1$ and there exists $L_1,L_2\in\integers$
such that
$$
\forall x\in\integers,\quad
\pi(x)+\pi'(L_1-x) = L_2.
$$
Then for the graph $G'$ of Lemma~\ref{le_fundamental_two_bijections}
(and of Theorem~\ref{th_folinsbee_friedman_algorithm}),
for all $m\in\integers$:
\begin{enumerate}
\item
if $m=L_1+qr$ for some $q\in\integers$, then all edges from $m$ on the
left are adjacent to and only to the single vertex $L_2-qr$; and
\item 
otherwise, $m$ on the left is connected to some left or right vertex
of $G'$ of arbitrarily small value (and hence lies in the $v_0$
component of $G$).
\end{enumerate}
\end{lemma}
Note that in the above lemma (which assumes the hypotheses of
Lemma~\ref{le_fundamental_two_bijections}),
if $y$ is a 
vertex of $G'$ on the right, then is connected to some vertex on
the left 
(namely, if $b+b'=y$, then $y$ on the right is connected to
$\pi^{-1}(b)+(\pi')^{-1}(b')$ on the left, by (1c) of
Proposition~\ref{pr_equivalent_forms_of_left_right_connectivity}
applied symmetrically).
Hence Lemma~\ref{le_consequence_of_fundamental_lemma}
identifies all connected components of $G'$ and therefore
of $G$ in Theorem~\ref{th_folinsbee_friedman_algorithm}.

\subsection{Proof of Theorem~\ref{th_H_one_vanishes_tensor_prod}}
\label{su_proof_of_first_main_theorem}

\begin{proof}[Proof of Theorem~\ref{th_H_one_vanishes_tensor_prod}]
Say that $b^1(\cM_{W,\mec d} \otimes_{\cO_r} \cM_{W',\mec d'} ) >0$.
Then $\cM_{W,\mec d} \otimes_{\cO_r} \cM_{W',\mec d'}\isom
\cM_{W'',\mec d+\mec d'}$ where $W''=W\star_r W'$.
Then the graph $G'$ of the Folinsbee-Friedman algorithm
has a connected component whose vertices are disjoint 
from $\integers_{\le d_1+d_1'}$ on the left and 
$\integers_{\le d_2+d_2'}$ on
the right.
Then this connected component has a left vertex of $G'$,
say $L_1\in\integers_{\rm left}$.
Applying Lemma~\ref{le_fundamental_two_bijections}, we see that
conclusion~(1) must hold.  Hence
$\pi(x)+\pi'(L_1-x)=L_2$ for all $x\in\integers$.
It follows that for all $x,y\in\integers$,
$$
W(x,y)=1 \iff y=\pi(x) \iff y=L_2-\pi'(L_1-x)
\iff \pi'(L_1-x)=L_2-y
$$
$$
\iff
W'(L_1-x,L_2-y)=1.
$$
Hence $W(\mec d)=W'(\mec L-\mec d)$ for all $\mec d\in\integers^2$
and hence $W'=W^*_\mec L$.
\end{proof}

\subsection{Proof of Theorem~\ref{th_H_one_equals_k_in_rare_case}}
\label{su_proof_of_second_main_theorem}

\begin{proof}[Proof of Theorem~\ref{th_H_one_equals_k_in_rare_case}]
Let us prove the statement for a general $\mec K'=\mec d+\mec d'$.
Let $W''=W\star_r W'$.
Since $W'(\mec L-\mec d)=W(\mec d)$ for all $\mec d\in\integers^2$,
the proof in the last subsection shows that
$\pi(x)+\pi'(L_1-x)=L_2$ for all $x$.
It follows from
Lemma~\ref{le_consequence_of_fundamental_lemma} that the connected
components of $G'$ in Theorem~\ref{th_folinsbee_friedman_algorithm}
(with $W''$ replacing $W$) are either (1) connected to
some left $\integers_{\le K_1'}$ vertex or some right
$\integers_{\le K_2'}$ vertex, or
(2) are a single connected component consisting of a left vertex
$L_1+qr> K_1'$ and a right vertex $L_2-qr>K_2'$ (for some $q\in\integers$).
Combining this with Theorem~\ref{th_folinsbee_friedman_algorithm}
proves the claim regarding
the basis of $H^1$ for
$$
\cM_{W'',\mec d+\mec d'}\isom\cM_{W,\mec d}\otimes_{\cO_r}\cM_{W',\mec d'}
$$

For the special case of 
$\omega_{W,\mec K}=\cM_{W,\mec 0}\oplus_{\cO_r}\cM_{W',\mec K}$,
where $\mec K'=\mec K$, we know that the only $q\in\integers$
with $L_1+qr>K_1=L_1-1$ and $L_2-qr>K_2=L_2-1$ is $q=0$.
\end{proof}

\subsection{How Does $\omega_W$ Depend on $W$?}
\label{su_omega_depends_on_W}

We have noted that if $W$ is an $r$-periodic perfect matching,
then
$\omega_{W,\mec K}$ depends --- up to isomorphism ---
only on $W$.  It is interesting to ask how $\omega_W$ depends on $W$.
The remarks below show that when $\tilde W$ is a translation of $W$ or of
$W^*_\mec 0$, then $\omega_W\isom\omega_{\tilde W}$; otherwise we
don't know if $\omega_W\isom\omega_{\tilde W}$,
although below we argue that if $\omega_W\isom\omega_{\tilde W}$ then this
isomorphism would have to be somewhat ``exotic.''

Since 
$$
\omega_{W,\mec K} = \cM_{W,\mec 0}\otimes_{\cO_r} \cM_{W^*_\mec L,\mec K} ,
$$
let us restrict to $\mec L=\mec 0$ and $\mec K=\mec -1$.
Then
$$
\omega_{W,-\mec 1} \isom \cM_{W\star_r W^*_\mec 0,\mec -1}.
$$
But if $W''=W\star_r W^*_\mec 0$, then
$$
W''(\mec s) = \sum_{a_1=0}^{r-1}\sum_{a_2\in\integers}
W(\mec a) W^*_\mec 0(\mec s-\mec a)
=
\sum_{a_1=0}^{r-1}\sum_{a_2\in\integers}
W(\mec a) W(\mec a-\mec s)
$$
So if $A$ is the support of $W$ in $\integers^2$, then $W''$ is
supported on the set of $\mec s$ such that some $\mec a\in A$
has an $\mec a'\in A$ such that $\mec a'=\mec a-\mec s$, i.e.,
$\mec s=\mec a-\mec a'$.
Hence $W''$ is determined by the multiset of differences $A-A$,
where we consider $A-A$ as a multiset based on the values of $W''$ there 
(they may be greater than $1$).
For example, if $\tilde W$ is a translation of $W$, or a translation of 
$W^*_\mec 0$, then 
$$
W''=W\star_r W^*_\mec 0
\mbox{\quad and\quad}
\tilde W''=\tilde W\star_r {\tilde W}^*_\mec 0
$$
are equal, since the respective supports $A,\tilde A$ of $W,\tilde W$
satisfy $A-A=\tilde A-\tilde A$ as multisets.

Note that in the last paragraph, $\omega_W\isom\omega_{\tilde W}$ if
the supports of $A,\tilde A$ satisfy
$A-A=\tilde A-\tilde A$ as multisets.
We remark that it is well known that there are subsets 
$I,\tilde I\in\integers$
such that $I-I=\tilde I-\tilde I$ as multisets
and $\tilde I$ is not a translate of $I$ or of $-I$
(e.g., \cite{bloom,rosenblatt}).
However, we don't know if this is true, modulo $(r,-r)$, of sets
$A,\tilde A\subset\integers^2$ where $W,\tilde W$ corresponding to
$A,\tilde A$ are perfect matchings.

Note that if we can find such $A,\tilde A$ such that
$A-A=\tilde A-\tilde A$ as sets, not multisets
(the above multiplicity convention),
then we still get $\omega_W\isom\omega_{\tilde W}$, although via a
somewhat ``exotic'' isomorphism:
to see this, first note that
$\omega_W$ decomposes as the sum of two $\cO_r$ modules,
one that is nonzero only
on $B_3$, whose value there is the kernel, $K=K(\omega_W)$ of
$$
\omega_W(B_3)\to \omega_W(A_1)\oplus\omega_W(A_2).
$$
(this consists of $r(r-1)$ copies of $\cO_r(B_3)=k[v,1/v]$ as an
$k[v,1/v]$-module).
Hence if $A-A=\tilde A-\tilde A$ simply as sets,
then one could take an identity morphism from $\omega_W\to\omega_{\tilde W}$
on $A_1,A_2,B_1,B_2$, and take 
the identity on the $B_3$ values corresponding to the ``first'' nonzero
value 
of $W''=W\star_r W^*_\mec 0$ and $\tilde W\star_r {\tilde W}^*_\mec 0$.
One can then take 
an arbitrary $\cO_r$ module map from $K(\omega_W)$ to
$K(\omega_{\tilde W})$.
However, this is a rather ``exotic'' isomorphism.
Moreover, we don't know any such examples of $r$-periodic perfect matchings
$W,\tilde W$ whose supports, $A,\tilde A$
satisfy $A-A$ and $\tilde A-\tilde A$ agree as sets and not as
multisets.

We don't know if more ``exotic'' isomorphisms $\omega_W\to\omega_{\tilde W}$
can exist if $A-A\ne \tilde A-\tilde A$ as sets.

\section{Duality}
\label{se_duality}

Again, in this section $k$ will be a fixed field, and we will often
suppress $k$, writing, e.g., $\cO_r$ instead of $\cO_{r,k}$.

Let $W$ be an $r$-periodic weight, and let $\cF$ be an $\cO_r$-module.
In this section we will describe a natural (in $\cF$) pairing
\begin{equation}\label{eq_F_pairing_start_of_duality_section}
H^1_\cO(\cM_{W,\mec 0}\otimes \cF) \times 
\Hom_\cO(\cF,\cM_{W^*_\mec L,\mec K})
\to 
% H^1(\omega_W)\xrightarrow{\isom}
k 
\end{equation} 
whenever $\mec K=\mec L+\mec 1$ are elements of $\integers^2$.
We will show that when $\cF$ is a line bundle $\cL_{r,\mec d}$, then this
gives a perfect pairing of finite dimensional vector spaces.
Hence this gives an isomorphism
\begin{equation}\label{eq_duality_for_cL_mec_d}
\Hom_\cO(\cF,\cM_{W^*_\mec L,\mec K})
\xrightarrow{\isom}
H^1(\cM_{W,\mec 0}\otimes \cF)'
\end{equation} 
(where $\,'$ denotes the dual space as a $k$-vector space).
It will easily follow that we get an isomorphism
$$
H^0(\cM_{W^*_\mec L,\mec K-\mec d})
\xrightarrow{\isom}
H^1(\cM_{W,\mec d})',
$$
which is therefore a ``duality theorem,''
according to the discussion
around \eqref{eq_duality_in_main_results_section}.

In Section~\ref{se_strong_duality} we will strengthen this result in two
ways: 
(1) we will give a second ``duality'' formula that $\cF=\cL_{r,\mec d}$
satisfies, and
(2) we will show that there are more examples of $\cF$ for which both
duality statements hold, which includes certain ``skyscraper diagrams''
and any
{\em coherent} $\cO$-module.
Section~\ref{se_strong_duality} requires more background in homological
algebra; this section will proceed more ``naively.''

We begin by describing the ingredients we need to produce the
pairing \eqref{eq_pairing_cM_W_otimes_F_etc}

\begin{remark}
The $k$-vector spaces 
$$
\Hom_\cO(\cF,\cM_{W^*_\mec L}) 
\mbox{\quad and\quad}
\Hom_{\underline k}(\cF,\cM_{W^*_\mec L}) 
$$
are very different; in particular, we will later explain that
the second space infinite dimensional for $\cF=\cL_{r,\mec d}$.
Hence it is essential for us to work with Hom sets of $\cO$-modules
in \eqref{eq_F_pairing_start_of_duality_section} in our duality theorem.
\end{remark}

\subsection{A Yoneda Pairing}
\label{su_yoneda_pairing_simple_form_i_one_A_equals_cO}

In this subsection $\cO$ is an arbitrary diagram of $k$-algebras;
the only reason we insist on this is that we have only
defined $H^i(\cF)$ when $\cF$ is a $k$-diagram (of vector spaces).
However, had we defined $H^i(\cF)$ for any diagram of abelian groups,
$\cF$ (in the evident fashion), then the discussion below would be
valid for any diagram of rings, $\cO$.

There is a very simple way to define a map:
\begin{equation}\label{eq_simple_yoneda}
H^i(\cF) \times \Hom_\cO(\cF,\cG) \to H^1(\cG) ;
\end{equation} 
we will call this map a {\em Yoneda pairing}.
In Section~\ref{se_strong_duality} we will explain that the map we describe
is really an example of the well-known {\em Yoneda pairing}.
For now we just describe this map from scratch.

So let $\phi\from\cF\to\cG$ be a morphism 
of diagrams of $k$-vector spaces.
Then $\phi$ induces a map
$$
\cF(A_1)\oplus\cF(A_2) \xrightarrow{\phi(A_1)\oplus\phi(A_2)}
 \cG(A_1)\oplus\cG(A_2).
$$
Since $\phi$ is a morphism, one easily checks that this map yields a map
$$
H^1(\cF)=
\bigl( \cF(A_1)\oplus\cF(A_2) \bigr)/
\bigl( {\rm Image}(\cF(\rho_{\rm tot})) \bigr)
$$
to
$$
H^1(\cG)=
\bigl( \cG(A_1)\oplus\cG(A_2) \bigr)/
\bigl( {\rm Image}(\cG(\rho_{\rm tot})) \bigr) .
$$
This gives a pairing
$$
H^1(\cF) \times \Hom(\cF,\cG) \to H^1(\cG).
$$
To determine this pairing it suffices to determine how each
$\alpha\in H^1(\cF)$ and $\beta\in\Hom(\cF,\cG)$
are mapped to in $H^1(G)$, which is determined by the two morphisms
$$
\beta(A_1)\from \cF(A_1)\to\cG(A_1),
\quad
\beta(A_2)\from \cF(A_2)\to\cG(A_2),
$$
plus the equivalence classes of $\cG(A_1)\oplus\cG(A_2)$ modulo
the image of $\cG(\partial)$.

Of course, 
in case $\cF,\cG$ are $\cO$-modules where each value of $\cO$ is
a $k$-algebra, then
$$
\Hom_{\cO}(\cF,\cG)\subset \Hom_k(\cF,\cG),
$$
and the above Yoneda pairing is also a pairing
$$
H^1(\cF)\times \Hom_{\cO}(\cF,\cG) \to H^1(\cG).
$$

\subsection{Tensoring Hom}

Let $\cO$ be any diagram of rings,
and let $X=\{A_1,A_2,B_1,B_2,B_3\}$.
For any $\cO$-modules $\cA,\cB,\cC$ there is a natural map
\begin{equation}\label{eq_tensoring_a_morphism}
\Hom_\cO(\cA,\cB)\xrightarrow{{\rm id}_\cC\otimes\cdot}
\Hom_\cO(\cC\otimes\cA,\cC\otimes\cB)
\end{equation} 
where ${\rm id}_\cC$ is the identity map on $\cC$:
indeed, an element $\phi\in\Hom_\cO(\cA,\cB)$ is just a family
of maps $\phi=\{\phi_P\}_{P\in X}$, where 
$$
\phi_P \from \cA(P)\to\cB(P)
$$
is a map of $\cO(P)$-modules, and the $\{\phi_P\}$ are compatible
with the restriction maps of $\cA$ and $\cB$.  In this case there is a 
map
called ${\rm id}_{\cC}\otimes\phi$ which is a morphism
$\cC\otimes\cA\to\cC\otimes\cB$ whose map at $P\in X$ is the unique morphism
of $\cO(P)$-modules
taking $(m,a)$ to $(m,\phi(a))$
(we easily check that there is a unique such morphism).

\subsection{The Duality Map}

In this subsection we describe the duality map
\eqref{eq_duality_for_cL_mec_d} based on the
pairing
\eqref{eq_F_pairing_start_of_duality_section}.

The map \eqref{eq_tensoring_a_morphism} gives a
a natural map
$$
\Hom_\cO(\cF,\cM_{W^*_\mec L})
\to
\Hom\bigl(\cM_{W,\mec 0}\otimes\cF,
\cM_{W,\mec 0}\otimes \cM_{W^*_\mec L,\mec K}\bigr)
=
\Hom\bigl(\cM_{W,\mec 0}\otimes\cF,\omega_W).
$$
So fix an isomorphism
$$
H^1( \omega )\isom k .
$$
We therefore get a map
\begin{equation}\label{eq_pairing_cM_W_otimes_F_etc}
H^1(\cM_{W,\mec 0}\otimes\cF)\times
\Hom(\cF,\cM_{W^*_\mec L,\mec K})\to 
H^1(\cM_{W,\mec 0}\otimes\cF)\times
\Hom(\cM_{W,\mec 0}\otimes\cF,\omega)\to 
H^1(\omega)\to k.
\end{equation} 
This is our desired map \eqref{eq_F_pairing_start_of_duality_section}.

\begin{definition}\label{de_duality_meaning_the_weak_kind}
We say that an $\cO_r$-module $\cF$ {\em satisfies duality} if the
map
\begin{equation}\label{eq_duality_weak_in_the_definition_of_duality}
H^1(\cM_{W,\mec 0}\otimes\cF)\times
\Hom(\cF,\cM_{W^*_\mec L,\mec K})\to k
\end{equation} 
induced from \eqref{eq_pairing_cM_W_otimes_F_etc} is a perfect pairing.
\end{definition}

\subsection{The Main Computation}

In this subsection we show that 
\eqref{eq_pairing_cM_W_otimes_F_etc} is a perfect pairing when
$\cF=\cL_{r,\mec d}$.
This is a straightfoward calculation, but one has to carefully keep track
of things.

Let us first make some preliminary remarks and lemmas, which hold for any
$r$-periodic perfect matching, $W$, and choice of $\mec K\in\integers^2$
and $\mec L=\mec K+\mec 1$.

First, recall from Theorem~\ref{th_H_one_equals_k_in_rare_case}
that for $\omega=\omega_{W,\mec K}$,
$H^1(\omega)$ is by definition
$$
\omega(A) / {\rm Image}(\omega(\partial))
$$
and is the one-dimensional $k$-vector space
generated by $(x_1^{L_1},0)$, which in $H^1(\omega)$ is equivalent
to the vector $-(0,x_2^{L_2})$.

Second, we may identify
$$
\Hom(\cL_{r;\mec d},\cM_{W^*_\mec L,\mec K})
\isom
\Hom(\cO,\cM_{W^*_\mec L,\mec K-\mec d}).
$$

Third, 
$$
\Hom(\cO,\cM_{W^*_\mec L,\mec K-\mec d})
$$
may be described as a direct sum by the following general proposition.

\begin{proposition}\label{pr_Hom_cO_to_cM_W_prime_mec_d_etc}
Let
$W'$ be any $r$-periodic perfect matching, and $\pi'$ is the
associated permutation of $W'$.
For any $\mec d\in\integers^2$, let
$$
I_1 = \{ i_i \ | \ i_1\ge d_1,\ \pi'(i_1)\ge d_2 \}
$$
(which is finite since $\pi'$ is bounded).
Then for each
$\phi\in \Hom(\cO,\cM_{W',\mec d})$, the map
\begin{equation}\label{eq_phi_at_A_one_from_cO_to_cM_W_prime}
\phi(A_1) \from \cO(A_1) \to \cM_{W',\mec d}(A_1)
\end{equation} 
--- seeing as $\cO(A_1)=\cM_{W',\mec d}(A_1)=k[x_1^\pm]$ ---
must take $1\in \cO(A_1)$ to $p(x_1)\in \cM_{W',\mec d}(A_1)$ where $p$ is
the Laurent polynomial; then $p$ is necessarily of the form
\begin{equation}\label{eq_p_of_x_one_is_a_Laurent_poly_in_indices_I_one}
p(x_1)=\sum_{i_1\in I_1} c_{i_1} x_1^{i_1} .
\end{equation}
Conversely, for any $p$ as in
\eqref{eq_p_of_x_one_is_a_Laurent_poly_in_indices_I_one},
the map from $k[x^\pm]$ to itself (as a module over itself)
taking $1$ to  $p(x_1)$ determines a (unique) map
$\phi(A_1)$ in \eqref{eq_phi_at_A_one_from_cO_to_cM_W_prime},
which extends to a unique morphism $\phi$ as in
$\phi\in \Hom(\cO,\cM_{W',\mec d})$.
This sets up a bijection
$$
\Hom(\cO,\cM_{W',\mec d}) \to 
\biggl\{ p \ \biggm| \ p(x_1)=\sum_{i_1\in I_1} c_{i_1} x_1^{i_1} ,
\ c_{i_i}\in k
\biggr\} \isom k^{I_1}.
$$
\end{proposition}
\begin{proof}
Let $P\in X=\{A_1,A_2,B_1,B_2,B_3\}$.
First note that if $m\in\cM_{W',\mec d}(P)$, then there
is a unique morphism of $\cO(P)$-modules $\cO(P)\to \cM_{W',\mec d}(P)$
that maps $1$ to $m$.
Hence a $\phi\in\Hom(\cO,\cM_{W',\mec d})$
gives for each $P\in X$ an element $m_P\in \cM_{W',\mec d}(P)$
such that $\phi(P)(1)=m_P$.
Hence such a $\phi$ is equivalent to a family $\{m_P\}_{P\in X}$
such that the maps $1\mapsto m_P$ are compatible with the restriction
maps; i.e., give a commutative diagram 
as in Figure~\ref{fi_morphism_of_diagrams} with $\cF=\cO$ and
$\cG=\cM_{W'.\mec d}(P)$.

So consider a $\phi\in\Hom(\cO,\cM_{W',\mec d})$; then
$\phi(A_1)$ is a map of $k[x_1^\pm]$ to itself as a
$k[x_1^\pm]$; let
$p(x)=\phi(A_1) 1 \in k[x_1^\pm]$, and therefore
$$
p(x)=\sum_{i\in\integers} c_i x_1^i.
$$
Similarly $\phi(B_1)$ must take $1$ to $q(y_1)\in k[y_1]$.
Since $\phi$ is compatible with the restriction map
$\rho_{1,1}$, we have $\phi(B_1)$ must take $1$ to $q(y_1)$
and the commutativity of:
$$
\begin{tikzpicture}[scale=0.4]
\node (B11) at (0,4) {$\cO(B_1)=k[y_1]$};
\node (B12) at (20,4) {$\cM_{W',\mec d}(B_1)=k[y_1]$};
\draw[->] (B11) -- (B12) node [midway,above] {$1\mapsto q(y_1)$};
\node (A11) at (3,0) {$\cO(A_1)=k[x_1^\pm]$};
\node (A12) at (23,0) {$\cM_{W',\mec d}(A_1)=k[x_1^\pm]$};
\draw[->] (A11) -- (A12) node [midway,above] {$1\mapsto p(x_1)$};
\draw[->] (B11) -- (A11) node [midway,left] {$1\mapsto 1$};
\draw[->] (B12) -- (A12) node [midway,right] {$1\mapsto x_1^{d_1}$};
\end{tikzpicture}
$$
(see the top of Figure~\ref{fi_morphism_of_diagrams})
we see that the left arrow followed by the bottom arrow takes $1$ to
$p(x_1)$, and the top takes $1$ to $q(y_1)$, which the right arrow
takes to $x_1^{d_1}q(1/x_1)$.  Hence
$$
x_1^{d_1} q(1/x_1)=p(x_1).
$$
Hence
\begin{equation}\label{eq_phi_of_B_one_as_one_mapsto_q}
q(y_1)=y_1^{d_1}p(1/y_1) = y_1^{d_1} \sum_i c_i y_1^{-i}.
\end{equation} 
Since $q(y_1)\in k[y_1]$, we must have $c_i\ne 0$ implies that
$d_1-i\ge 0$, or, equivalently, $i\le d_1$.
Similarly chasing through the restrictions $\rho_{3,1}$ with the diagram:
$$
\begin{tikzpicture}[scale=0.4]
\node (B31) at (0,-4) {$\cO(B_3)=k[v^\pm]$};
% \node (B32) at (20,-4) {$\cM_{W',\mec d}(B_3)=k[v^\pm]^{\oplus r}$};
\node (B32) at (20,-4) {$\cM_{W',\mec d}(B_3)=k^{\oplus W}\isom k[v^\pm]^{\oplus r}$};
\draw[->] (B31) -- (B32) node [midway,above] {$1\mapsto \sum_i c_i(x_1^i,x_2^{\pi'(i)})$};
\node (A11) at (3,0) {$\cO(A_1)=k[x_1^\pm]$};
\node (A12) at (23,0) {$\cM_{W',\mec d}(A_1)=k[x_1^\pm]$};
\draw[->] (A11) -- (A12) node [midway,above] {$1\mapsto p(x_1)=\sum_i c_i x_1^i$};
\draw[->] (B31) -- (A11) node [midway,left] {$1\mapsto 1$};
\draw[->] (B32) -- (A12) node [midway,right] {$(x_1^i,x_2^{\pi'(i)}) \mapsto x_1^i$};
\end{tikzpicture}
$$
we see that $\phi(B_3)$ must be the map
$$
1 \mapsto \sum_i c_i \bigl( x_1^i , x_2^{\pi'(i)} \bigr),
$$
and chasing through $\rho_{3,2}$ we see that $\phi(A_2)$ must be the map
$$
1 \mapsto \sum_i c_i x_2^{\pi'(i)},
$$
and then chasing through $\rho_{2,2}$ as we did $\rho_{1,1}$, we see that
$\phi(B_2)$ must be the map $1\mapsto \tilde r(y_2)$ where
(compare with \eqref{eq_phi_of_B_one_as_one_mapsto_q})
$$
r(y_2) = y_2^{d_2} \sum_i c_i y_2^{-\pi'(i)}.
$$
Since $r(y_2)$ must lie in $k[y_2]$ we have $c_i\ne 0$ implies
$\pi'(i)\le d_2$.

Moreover, as long as $c_i\ne 0$ implies $i\le d_1$ and $\pi'(i)\le d_2$, then
$\phi$, which is determined by $\phi(A_1)$, is a valid morphism
$\cO\to\cM_{W',\mec d}$, seeing as each of the candidates for
$\phi(P)$ for 
$P\in X$ is indeed a valid morphism $\cO(P)\to\cM_{W',\mec d}(P)$ that
comprise $\phi$.
\end{proof}

We use Proposition~\ref{pr_Hom_cO_to_cM_W_prime_mec_d_etc} as follows.

\begin{corollary}\label{co_hom_c0_to_cM_etc}
Let $W$ be an $r$-periodic perfect matching, and
$\mec L,\mec K,\mec d\in\integers^2$ satisfy $\mec L=\mec K+\mec 1$.
Then a basis for
$\Hom\bigl(\cO_r,\cM_{W^*_\mec L,\mec K-\mec d}\bigr)$
is given by the unique 
$\phi\in\Hom\bigl(\cO_r,\cM_{W^*_\mec L,\mec K-\mec d}\bigr)$
such that
$\phi(A_1)$ is multiplication by $x_1^{i_1}$
where $i_1$ varies in the set
$$
I_1 = \{ i_1 \ | \ \mbox{\rm for the unique $\mec i=(i_1,i_2)$ with $W(\mec L-\mec i)=1$ we have $\mec i\le \mec K-\mec d$} \}.
$$
Hence, (substituting $\mec b=\mec L-\mec i$ and) setting
\begin{equation}\label{eq_B_one_for_Hom_basis}
b_1 = \{ b_1 \ | \ \mbox{\rm for the unique $\mec b=(b_1,b_2)$ with $W(\mec b)=1$ we have $\mec d+\mec 1\le \mec b$} \},
\end{equation} 
$\Hom\bigl(\cO_r,\cM_{W^*_\mec L,\mec K-\mec d}\bigr)$ has a basis
consisting of
\begin{equation}\label{eq_Hom_basis_phi_related_to_B_one}
\{ \phi\in \Hom\bigl(\cO_r,\cM_{W^*_\mec L,\mec K-\mec d}\bigr)
\ | \  \mbox{\rm $\phi(A_1)$ is multiplication by $x_1^{L_1-b_1}$} \} .
\end{equation} 
\end{corollary}

\begin{remark}\label{re_hom_cO_to_cM_versus_cL_to_cM_sub_minus_mec_d}
In Corollary~\ref{co_hom_c0_to_cM_etc}, if
$\phi\in \Hom(\cO_r,\cM_{W^*_\mec L,\mec K-\mec d})$ is given by
$\phi(A_1)$ maps $1\mapsto p(x_1)$, then the corresponding
$\phi'\in\Hom(\cL_{r,\mec d},\cM_{W^*_\mec L,\mec K})$ also has
$\phi'(A_1)$ maps $1\mapsto p(x_1)$, since the isomorphism
$$
\Hom(\cO_r,\cM_{W^*_\mec L,\mec K-\mec d})
\to
\Hom(\cL_{r,\mec d},\cM_{W^*_\mec L,\mec K})
$$
is just tensoring with $\cL_{r,\mec d}$ on both sides, which doesn't
affect $A_1$ since $\cL_{r,\mec d}(A_1)=\cO(A_1)$.
\end{remark}

\begin{theorem}
\label{th_weak_duality_for_cL_mec_d}
Let $k$ be a field and
$W\from\integers^2\to\{0,1\}$ be a perfect matching with associated
bijection $\pi$.  Let $\mec L\in\integers^2$ set $\mec K=\mec L-\mec 1$.
For an arbitrary $\cO_r$-module, $\cF$, consider the pairing of $k$-diagrams
\begin{equation}\label{eq_F_pairing}
H^1(\cM_{W,\mec 0}\otimes\cF) \times
\Hom_{\cO_r}(\cF,\cM_{W^*_{\mec L},\mec K})
\to 
H^1(\omega_{W,\mec K})
\xrightarrow{\isom} k
\end{equation} 
given by composing the map
\begin{align*}
\Hom_{\cO_r}(\cF,\cM_{W^*_{\mec L},\mec K})
& \to
\Hom_{\cO_r}(\cM_{W,\mec 0}\otimes\cF,
\cM_{W,\mec 0}\otimes\cM_{W^*_{\mec L},\mec K}) \\
& =
\Hom_{\cO_r}(\cM_{W,\mec 0}\otimes\cF,\omega_W)
\end{align*}
and the Yoneda map
$$
H^1(\cF\otimes\cM_{W,\mec 0}) \times
\Hom_{\cO_r}(\cM_{W,\mec 0}\otimes\cF,\omega_W)
\to
H^1(\omega_W)
$$
and the isomorphism
$$
H^1(\omega_W)\to k
$$
which takes
$$
(x_1^{L_1},0)\in \omega_W(A)
$$
to $1\in k$
(and therefore takes $(0,x_2^{L_2})$ to $-1$).
Consider the particular
case where $\cF$ is the line bundle $\cL_{\mec d}=\cL_{r,\mec d}$
for some $\mec d\in\integers^2$.
Then
\begin{enumerate}
\item 
If $(x_1^a,0)\in\cM_{W,\mec d}(A)$, and if
$\phi\in \Hom_{\cO_r}(\cO_r,\cM_{W^*_{\mec L},\mec K-\mec d})$ has
$\phi(A_1)=x_1^b$ for some $b\in\integers$, then under the pairing
\eqref{eq_F_pairing} and the natural isomorphisms
$$
\Hom_{\cO_r}(\cO_r,\cM_{W^*_{\mec L},\mec K-\mec d})
\isom
\Hom_{\cO_r}(\cL_{r,\mec d},\cM_{W^*_{\mec L},\mec K}),
$$
$((x_1^a,0),\phi)$ is taken to 
$(x_1^{a+b},0)\in \omega_{W,\mec K}(A)$.
\item
Hence $(x_1^a,\phi)$ above is taken to a non-zero element
of $H^1(\omega_{W,\mec K})$ iff $a+b=L_1$.
% \item
% The above~(1) and~(2)hold for 
% $$
% A_1, \ L_1, \ (x_1^a,0),\ (x_1^{a+b},0)
% $$
% everywhere replaced with, respectively, 
% $$
% A_2, \ L_2, \ (0,x_2^a),\ (0,-x_2^{a+b}) 
% $$
% (note that $(x_1^{L_1},0)=(0,-x_2^{L_2})$ in $\omega_W(A)$,
% so the map $(x_1^{L_1},0)\mapsto 1$ also takes
% $(0,x_2^{L_2})$ to $-1$).
\item
The pairing \eqref{eq_F_pairing} for $\cF=\cL_{k,\mec d}$ is a
perfect pairing.
Hence $\cL_{k,\mec d}$ satisfies duality
(Definition~\ref{de_duality_meaning_the_weak_kind}.)
\end{enumerate}
\end{theorem}
\begin{proof}
To prove (1), consider how a
$\phi$
fits into the pairing
\eqref{eq_F_pairing}: according to
Remark~\ref{re_hom_cO_to_cM_versus_cL_to_cM_sub_minus_mec_d},
we identify $\phi$ with the element
\begin{equation}\label{eq_phi_prime_in_Hom_L_r_d_M_W_star_etc}
\phi' \in 
\Hom(\cL_{r,\mec d}, \cM_{W^*_{\mec L},\mec K})\isom
\Hom(\cO,\cM_{W^*_{\mec L},\mec K-\mec d}),
\end{equation} 
where $\phi'(A_1)$ is multiplication by $x_1^b$ in
$$
\Hom_{k[x_1^\pm]}(k[x_1^\pm],k[x_1^\pm]).
$$
Hence
$$
({\rm Id}_{\cM_{W,\mec 0}}\otimes\phi')(A_1) 
\from \Hom_{k[x_1^\pm]}(k[x_1^\pm]\otimes k[x_1^\pm],
k[x_1^\pm]\otimes k[x_1^\pm])
\isom
\Hom_{k[x_1^\pm]}(k[x_1^\pm],k[x_1^\pm])
$$
is again multiplication by $x_1^b$,
and therefore takes $x_1^a\in k[x_1^\pm]=\cM_{W,d}(A_1)$
to $x_1^{a+b}\in \omega_{W,\mec K}(A_1)$.
Hence $(x_1^a,0)\in\cM_{W,\mec d}(A)$ is taken to
$(x_1^{a+b},0)\in\omega_W(A)$.

(2) follows from Theorem~\ref{th_H_one_equals_k_in_rare_case},
since the only $(x_1^c,0)$ that is nonzero in
$H^1(\omega_W)$ is $(x_1^{L_1},0)$.

% (3)~follows by the analogous argument on values of $A_2$.

To prove (3) let us first show that
$H^1(\cM_{W,\mec d})$ the following basis:
\begin{equation}\label{eq_basis_for_H_one_cM_W_mec_d}
\{ (x_1^{a_1},0) \ | \mbox{\rm for the unique $\mec a=(a_1,a_2)$ with $W(\mec a)=1$ we have $\mec d+\mec 1 \le \mec a$} \};
\end{equation} 
the reader can either prove this from scratch, or use the following
material from
\cite{folinsbee_friedman_euler}:
first, we have
\begin{equation}\label{eq_cM_W_mec_d_isom_sum_indicators}
\cM_{W,\mec d}\isom\cI_\mec d^{\oplus W}=
\bigoplus_{W(\mec a)=1} \cI_{\mec d\ge \mec a}
\end{equation} 
(see Definitions~6.1 and~6.2 for this notation,
and Proposition~6.1 of \cite{folinsbee_friedman_euler}
for the isomorphism).
We then have
$$
H^1(\cM_{W,\mec d}) \isom \bigoplus_{W(\mec a)=1} H^1(\cI_{\mec d\ge \mec a})
$$
(see end Section~4.1 of \cite{folinsbee_friedman_euler}).
Now according to Example~5.2 of \cite{folinsbee_friedman_euler},
$H^1(\cI_{\mec d\ge\mec a})$ is nonzero iff $\cI_{\mec d\ge\mec a}$
is a copy of $\underline k_{/B_1,B_2}$, which holds iff
$\mec d+\mec 1\le\mec a$;
moreover, the definition of $\cI_{\mec d\ge\mec a}$ shows that
if $\mec d+\mec 1\le\mec a$, then $H^1(\cI_{\mec d\ge\mec a})$
is generated by $(1,0)$ in $\cI_{\mec d\ge\mec a}(A_1)$,
which corresponds to $(x_1^{a_1},0)$ in $\cM_{W,\mec d}$
under the isomorphism of \eqref{eq_cM_W_mec_d_isom_sum_indicators}.
Hence \eqref{eq_basis_for_H_one_cM_W_mec_d} is a basis for
$H^1(\cM_{W,\mec d})$.

Now, compare
\eqref{eq_basis_for_H_one_cM_W_mec_d} with
\eqref{eq_B_one_for_Hom_basis}:
the set of $\mec a$ with $W(\mec a)=1$ and $\mec d+\mec 1\ge \mec a$
is a finite set (since $W(\mec a)=0$ if $\deg(\mec a)$ is sufficiently
large), and is the same condition that $\mec b$ in
\eqref{eq_B_one_for_Hom_basis} satisfies.
By~(2) above, $(x^{a_1+L_1-b_1},0)\in H^1(\omega_{W,\mec K})$ is nonzero iff
$a_1=b_1$; hence
\eqref{eq_F_pairing} with $\cF=\cL_{r,\mec d}$
is a perfect pairing
of finite dimensional $k$-vector spaces.
\end{proof}

\subsection{More Diagrams $\cF$ Satisfying Duality}

At this point we can conclude that other diagrams satisfy
\eqref{eq_F_pairing}.
We will only outline the ideas, because in the next section we
will prove much stronger results.

Say that $0\to\cF_1\to\cF_2\to\cF_3\to 0$ is an exact sequence
of $\cO_r$-modules, and both $\cF_1,\cF_2$ satisfy
\eqref{eq_F_pairing}.  Then we claim the same is true of
$\cF_3$.  To see this, one has to verify that our definition
of cohomology groups and the Yoneda pairing guarentees that we have
a commutative diagram
\newcommand{\joelsSmallCommutativeDiagramViaPairing}[3]{
\begin{tikzpicture}[scale=0.40]
\node (A0) at (#1 * 0.3,#2) {#3 $0$};
\node (A1) at (#1 * 1,#2) {#3 $\Hom(\cF_3,\cM_{W^*_\mec L,\mec K})$};
\node (A2) at (#1 * 2,#2) {#3 $\Hom(\cF_2,\cM_{W^*_\mec L,\mec K})$};
\node (A3) at (#1 * 3,#2) {#3 $\Hom(\cF_1,\cM_{W^*_\mec L,\mec K})$};
\draw[->] (A0) -- (A1);
\draw[->] (A1) -- (A2);
\draw[->] (A2) -- (A3);
\node (B0) at (#1 * 0.3,#2 *0) {#3 $0$};
\node (B1) at (#1 * 1,#2 *0) {#3 $H^1(\cM_{W,\mec 0}\otimes\cF_3)^*$};
\node (B2) at (#1 * 2,#2 *0) {#3 $H^1(\cM_{W,\mec 0}\otimes\cF_2)^*$};
\node (B3) at (#1 * 3,#2 *0) {#3 $H^1(\cM_{W,\mec 0}\otimes\cF_1)^*$};
\draw[->] (B0) -- (B1);
\draw[->] (B1) -- (B2);
\draw[->] (B2) -- (B3);
\draw[->] (A1) -- (B1);
\draw[->] (A2) -- (B2);
\draw[->] (A3) -- (B3);
\end{tikzpicture}
}
\begin{equation}\label{eq_short_commutative_duality_diagram}
\joelsSmallCommutativeDiagramViaPairing{10}{3}{}
\end{equation} 
If so, then the five-lemma shows that $\cF_3$ also satisfies
\eqref{eq_F_pairing}.
Let us give an important example.

\begin{definition}\label{de_cS_one_two}
For $i=1,2$, the {\em small skyscraper at $B_i$}, denoted $\cS_i$,
is the $\cO_r$-module whose values are $0$
everywhere except at $B_i$, where its value is
$k[y_1]/y_1 k[y_1]$.
\end{definition}
Hence $\cS_1,\cS_2$ are $\cO_r$-modules (whose definition is independent of
$r$), and as $k$-diagrams they
are the diagrams whose values are $0$ everywhere except at $B_i$,
where its value is $k$.
Note that $\cS_i$ is an example of a ``skyscraper diagram'' in the sense of
Subsubsection~\ref{su_skyscraper_diagrams_explained}.

\begin{example}
For each $r\in\naturals$ and $\mec d\in\integers^2$,
there is an exact sequence of $\cO_r$-modules
\begin{equation}\label{eq_skyscraper_S_one_exact_sequence_first}
0 \to \cL_{r,\mec d} \to \cL_{r,\mec d+\mec e_1} \to \cS_1 \to 0.
\end{equation} 
where the map $\cL_{r,\mec d} \to \cL_{r,\mec d+\mec e_1}$ is the inclusion,
and $\cS_1$ is therefore the quotient.
By the five-lemma, $\cS_1$ also satisfies duality.
There is an analogous short exact sequence with $\cS_2$, and similarly
$\cS_2$ satisfies duality.
\end{example}
We remark that 
one can equally well see that the $\cS_j$ satisfy duality by
verifying that for $j\in[2]$, both
$$
\Hom(\cS_j,\cM_{W^*_\mec L,\mec K}),
\quad
H^1(\cM_{W,\mec 0}\otimes\cS_j)
$$
vanish.

It follows that any diagram that can be written as the cokernel of a morphism
$\cF_1\to\cF_2$ where $\cF_1$ are sums of $\cL_{r,\mec d}$ and
$\cS_j$ above also satisfies \eqref{eq_F_pairing}.

We remark that the commutativity of
the diagram \eqref{eq_short_commutative_duality_diagram}
follows from results in the next section,
which appeal to results about the Yoneda pairing
using the powerful framework of derived
categories.

\section{Strong Duality}
\label{se_strong_duality}

Again, in this section $k$ will be a fixed field, and we will often
suppress $k$, writing, e.g., $\cO_r$ instead of $\cO_{r,k}$.

In Theorem~\ref{th_weak_duality_for_cL_mec_d} we showed that pairing
\eqref{eq_F_pairing_start_of_duality_section} is a perfect pairing
for $\cF=\cL_{r,\mec d}$, and we therefore get an isomorphism
\eqref{eq_duality_for_cL_mec_d} for this value of $\cF$.
The point of this section is to give
a stronger duality theorem: namely, we will build a pairing
$$
H^i(\cM_{W,\mec 0}\otimes\cF) \times
\Ext^{1-i}_{\cO_r}(\cF,\cM_{W^*_\mec L,\mec K})
\to k
$$
for any $\cF$ and $i=0,1$, which therefore gives a morphism
\begin{equation}\label{eq_strong_duality_arising_from_perfect_pairing}
\Ext^{1-i}_{\cO_r}(\cF,\cM_{W^*_\mec L,\mec K})
\to
H^i(\cM_{W,\mec 0}\otimes\cF)'
\end{equation} 
for $i=0,1$.
We will say that $\cF$ satisfies {\em strong duality} if this map
is an isomorphism for $i=0,1$;
for $i=1$, \eqref{eq_strong_duality_arising_from_perfect_pairing}
turns out to be just the map \eqref{eq_duality_for_cL_mec_d},
and so strong duality is, indeed, a stronger property than 
duality (Definition~\ref{de_duality_meaning_the_weak_kind}).

In this section we give examples of $\cO_r$-modules that satisfy
strong duality, namely
(1) the skyscraper diagrams $\cS_j$ for $j=1,2$
(Definition~\ref{de_cS_one_two}), and
(2) $\cL_{r,\mec d}$ for any $\mec d\in\integers^2$.
This is proven in Theorem~\ref{th_strong_duality_for_skyscrapers_and_cL_mec_d},
which is based on
Theorem~\ref{th_strong_duality_for_cL_mec_d_degree_suff_small}.

It is also a standard type of result that if
$0\to\cF_1\to\cF_2\to\cF_3\to 0$ is a short exact sequence, and
any two of the $\cF_i$ satisfy strong duality, then so does the third.
This is proven in Theorem~\ref{th_duality_commutative_diagram}.

To prove these theorems we wll need to first exploit the
{\em flatness} of $\cM_{W,\mec d}$; we explain
flatness in Subsection~\ref{su_flatness}, and its application to
Ext groups in
Subsection~\ref{su_flat_module_and_ext_groups}.
We also need a more general Yoneda pairing than in Section~\ref{se_duality};
this we develop in
Subsection~\ref{su_yodena_pairing_derived_category}.
The remaining subsections of this section prove
Theorems~\ref{th_duality_commutative_diagram},
\ref{th_strong_duality_for_cL_mec_d_degree_suff_small},
and~\ref{th_strong_duality_for_skyscrapers_and_cL_mec_d}.

Parts of this section will also refer to facts regarding
homological algebra and injective and projective resolutions
of $k$-diagrams developed in Appendix~\ref{ap_yoneda_pairing_etc}.

Throughout this section we will work either with $\cO_r$-modules or
$\cO$-modules where $\cO$ is a more general structure (either a 
diagram of rings or a diagram of $k$-algebras).
For brevity we omit $\cO_r$ or $\cO$ in subscripts when confusion is
unlikely; for example, we often write
$\Hom(\cF,\cG)$ instead of $\Hom_{\cO}(\cF,\cG)$ or $\Hom_{\cO_r}(\cF,\cG)$,
and similarly for $\Ext^i(\cF,\cG)$.

\subsection{Flatness}
\label{su_flatness}

Flatness is a standard concept in homological algebra: if $R$ is a ring
and $M$ an $R$-module, we say $M$ is {\em flat} if for every exact 
sequence of $R$-modules $0\to A\to B\to C\to 0$ we have that
$$
0 \to M\otimes_R A \to M\otimes_R B \to M\otimes_R C \to 0
$$
(obtained by tensoring with ${\rm id}_M$)
is also exact.
It is a standard result that a direct sum of copies of any ring, $R$,
is (a projective $R$-module and) a flat $R$-module;
see \cite{hilton} Proposition~III.7.4 (and its proof),
\cite{rotman} Proposition~3.46, or 
\cite{weibel} beginning of Section~3.2.

\begin{definition}
Let $\cO$ be a diagram of $k$-algebras.  We say that an $\cO$-module,
$\cM$, is {\em flat} if for every exact sequence of $\cO$-modules
$$
0\to\cF_1\to\cF_2\to\cF_3\to 0
$$
we have that
$$
0\to\cM\otimes_\cO\cF_1\to\cM\otimes_\cO\cF_2\to\cM\otimes_\cO\cF_3\to 0
$$
is exact.
\end{definition}
Here, of course, the maps $\cM\otimes_\cO\cF_i\to\cM\otimes_\cO\cF_{i+1}$
are the maps taking $(m,f)$ to $(m,f')$ where $f$ maps to $f'$ in the 
arrow $\cF_i\to\cF_{i+1}$.

\begin{lemma}
Let $\cO$ be a diagram of $k$-algebras.  Let $\cM$ be an $\cO$-module
such that for any 
$P\in X=\{A_1,A_2,B_1,B_2,B_3\}$,
$\cM(P)$ is isomorphic to a direct sum of copies of $\cO(P)$.
Then $\cM(P)$ is flat.
\end{lemma}
\begin{proof}
Let $0\to\cF_1\to\cF_2\to\cF_3\to 0$ be a short exact sequence.
Then for each $P\in X=\{A_1,A_2,B_1,B_2,B_3\}$ we have an exact sequence
$$
0\to\cF_1(P)\to\cF_2(P)\to\cF_3(P)\to 0.
$$
Since $\cM(P)$ is isomorphic to a direct sum of copies of $\cO(P)$,
$\cM(P)$ is a (projective and) flat $\cO(P)$-module.  Therefore
$$
0\to\cM(P) \otimes_{\cO} \cF_1(P)
\to\cM(P) \otimes_{\cO} \cF_2(P)
\to\cM(P) \otimes_{\cO} \cF_3(P)\to 0.
$$
Since $P\in X$ is arbitrary, we have that
$$
0\to\cM\otimes\cF_1\to
\cM\otimes\cF_2\to
\cM\otimes\cF_3\to 0
$$
is exact.
\end{proof}
Of course, more generally $\cM$ is a flat $\cO$-module if
for all $P\in X$, $\cM(P)$ is a flat $\cO(P)$-module.

\subsection{The Map $\Hom(\cF,\cG)\to\Hom(\cM\otimes\cF,\cM\otimes\cG)$
and Ext Groups for Flat $\cM$}
\label{su_flat_module_and_ext_groups}

In this subsection we will describe a map
$$
\Ext^i(\cF,\cG) \to \Ext^i(\cM\otimes\cF,\cM\otimes\cG)
$$
that results from standard homological algebra whenever $\cM$ is flat.

At this point we will assume
the homological algebra in
\cite{weibel}, namely 
Theorem~2.7.6 on page~63, which states that
$\Ext^n_R(A,B)$, defined to be the right derived functors
of $B\mapsto\Hom(A,B)$ for $A$ fixed, is isomorphic to the
right derived functors of $A\mapsto\Hom(A,B)$ for $B$ fixed.
Note that by
Subsection~\ref{su_enough_injectives_and_projectives},
for any diagram of rings, $\cO$,
the category of $\cO$-modules has enough injectives and
projectives.
Hence, by the Freyd-Mitchell embedding theorem 
(see, e.g., \cite{weibel}, Theorem~1.61, page~25),
Theorem~2.7.6 (which holds for $R$-modules for a ring, $R$)
also holds for $\cO$-modules.

Let $\cO$ be any diagram of rings.  Then for $\cO$-modules $\cF,\cG$,
one defines $\Ext^i(\cF,\cG)$ for $i\ge 0$ as the right derived
functors of $\cG\mapsto \Hom(\cF,\cG)$ (with $\cF$ fixed);
by Theorem~2.7.6 of \cite{weibel} (page~63),
these functors are isomorphic to the right derived functors of
$\cF\mapsto\Hom(\cF,\cG)$ with $\cG$ fixed.
The groups $\cF\mapsto \Ext^i(\cF,\cG)$ is a universal set of
$\delta$-functors, since the category of $\cO$-modules has enough
projectives.

If $\cM$ is any $\cO$-modules, we get a natural map
\begin{equation}\label{eq_hom_cF_cG_to_tensoring_with_cM}
\Hom(\cF,\cG)\to\Hom(\cM\otimes\cF,\cM\otimes\cG)
\end{equation} 
by tensoring with ${\rm id}_M$; i.e., for each
$\phi\from\cF\to\cG$ we give a map
$$
{\rm id}_M\otimes\phi
\from \cM\otimes\cF \to \cM\otimes\cG
$$
which for each
$P\in X={A_1,A_2,B_1,B_2,B_3}$
is the map
$$
\cM(P)\otimes\cF(P) \xrightarrow{{\rm id}_M\otimes\phi(P)}
\cM(P)\otimes\cG(P)
$$
(taking $m\otimes f$ to $m\otimes(\phi(P)(f))$).

Now assume that $\cM$ is flat.  
Then the functors
$$
\cF\mapsto \Ext^i(\cM\otimes\cF,\cM\otimes\cG) 
$$
for $i\ge 0$ give another $\delta$-functor, since if
$0\to\cF_1\to\cF_2\to\cF_3\to 0$ is short exact, then
$0\to\cM\otimes\cF_1\to\cM\otimes\cF_2\to\cM\otimes\cF_3\to 0$
is short exact, and we get a long exact sequence
of groups $\Ext^i(\cM\otimes\cF_j,\cM\otimes\cG)$.

Since the functors $\cF\to\Ext^i(\cF,\cG)$ form a universal 
$\delta$-functor, the map
$$
\Ext^0(\cF,\cG)\to\Ext^0(\cM\otimes\cF,\cM\otimes\cG)
$$
(given by \eqref{eq_hom_cF_cG_to_tensoring_with_cM}), gives 
rise to a unique set of maps (as $i$ varies)
$$
\Ext^i(\cF,\cG)\to\Ext^i(\cM\otimes\cF,\cM\otimes\cG)
$$
of $\delta$-functors.
By definition this means that for every short exact sequence
$$
0\to\cF_1\to\cF_2\to\cF_3\to 0
$$
we have vertical arrows
%%%%%%%%%%%%%%%%%%%%%%%%%%%%%%%%%%%%%%%%%%%%%%%%%%%%%%%%%% 
\newcommand{\joelsDeltaFunctorCommutativeDiagram}[3]{
\begin{tikzpicture}[scale=0.50]
\node (A0) at (#1 * 0.3,#2) {#3 $0$};
\node (A1) at (#1 * 1,#2) {#3 $\Hom(\cF_3,\cG)$};
\node (A2) at (#1 * 2,#2) {#3 $\Hom(\cF_2,\cG)$};
\node (A3) at (#1 * 3,#2) {#3 $\Hom(\cF_1,\cG)$};
\node (A4) at (#1 * 4,#2) {#3 $\Ext^1(\cF_3,\cG)$};
\node (A5) at (#1 * 5,#2) {#3 $\Ext^1(\cF_2,\cG)$};
\node (A6) at (#1 * 6,#2) {#3 $\Ext^1(\cF_1,\cG)$};
\node (A7) at (#1 * 6.7,#2) {#3 $0$};
\draw[->] (A0) -- (A1);
\draw[->] (A1) -- (A2);
\draw[->] (A2) -- (A3);
\draw[->] (A3) -- (A4);
\draw[->] (A4) -- (A5);
\draw[->] (A5) -- (A6);
\draw[->] (A6) -- (A7);
\node (B0) at (#1 * 0.3,#2 *0) {#3 $0$};
\node (B1) at (#1 * 1,#2 *0) {#3 \tiny$\Hom(\cM\otimes\cF_3,\omega)$};
\node (B2) at (#1 * 2,#2 *0) {#3 \tiny$\Hom(\cM\otimes\cF_2,\omega)$};
\node (B3) at (#1 * 3,#2 *0) {#3 \tiny$\Hom(\cM\otimes\cF_1,\omega)$};
\node (B4) at (#1 * 4,#2 *0) {#3 \tiny$\Ext^1(\cM\otimes\cF_3,\omega)$};
\node (B5) at (#1 * 5,#2 *0) {#3 \tiny$\Ext^1(\cM\otimes\cF_2,\omega)$};
\node (B6) at (#1 * 6,#2 *0) {#3 \tiny$\Ext^1(\cM\otimes\cF_1,\omega)$};
\node (B7) at (#1 * 6.7,#2 *0) {#3 $0$};
\draw[->] (B0) -- (B1);
\draw[->] (B1) -- (B2);
\draw[->] (B2) -- (B3);
\draw[->] (B3) -- (B4);
\draw[->] (B4) -- (B5);
\draw[->] (B5) -- (B6);
\draw[->] (B6) -- (B7);
\draw[->] (A1) -- (B1);
\draw[->] (A2) -- (B2);
\draw[->] (A3) -- (B3);
\draw[->] (A4) -- (B4);
\draw[->] (A5) -- (B5);
\draw[->] (A6) -- (B6);
\end{tikzpicture}
}
%%%%%%%%%%%%%%%%%%%%%%%%%%%%%%%%%%%%%%%%%%%%%%%%%%%%%%%%%% 
$$
\joelsDeltaFunctorCommutativeDiagram{5}{3}{\SMALL}
$$
where $\omega=\cM\otimes\cG$.

\subsection{The Yoneda Pairing}
\label{su_yodena_pairing_derived_category}

Let $\cO$ be any diagram of rings.  Then the
{\em Yoneda pairing} is a map
\begin{equation}\label{eq_yoneda_pairing_start_of_Yoneda_subsection}
\Ext^i_\cO(A,B) \times \Ext^j_\cO(B,C) \to \Ext_\cO^{i+j}(A,C)
\end{equation} 
which is functorial in
$A,B,C$ in an appropriate sense.  

We will only need the pairing in the cases $i=0$ and $j=1$ and 
$i=1$ and $j=0$.  In these cases
there are two standard ways to describe this pairing.

The first way is simpler to work with: it involves working in the
derived category\footnote{
  Here we are working with the full derived category of $\cO$-modules,
  $\cD(\cO)$.  However, since all the $\cO$-modules of interest to us
  have finite projective and injective resolutions, we could equally
  well work in $\cD^+(\cO)$, $\cD^-(\cO)$, or $\cD^b(\cO)$ of the 
  derived category of $\cO$-modules with boundedness conditions;
  see, e.g., the beginnig of Subsection~III.2.5 of \cite{gelfand}
  for definitions.
  } of $\cO$-modules, whose foundations do most of
the work for us.  This is spelled out in
\cite{gelfand} Remark~III.5.4(b), page~166: namely,
we can identify $\Ext_\cO^i(X,Y)$ with 
$\Hom_{\cD(\cO)}(X[k],Y[i+k])$,
where $\cD(\cO)$ is the derived category of $\cO$-modules.
(For background on this and more references, see 
Subsection~\ref{su_background_on_Ext_and_derived_category}.)
Then for $\cO$-modules $A,B,C$, the map 
\eqref{eq_yoneda_pairing_start_of_Yoneda_subsection} is the composition
of $\Hom$'s in the derived category
$$
\Hom_{\cD(\cO)}(X[k],Y[k+i]) \times \Hom_{\cD(\cO)}(Y[k+i],Z[k+i+j])
\to
\Hom_{\cD(\cO)}(X[k],Z[k+i+j]) .
% \isom \Ext^{i+j}_\cO(X,Z).
$$
The wonderful thing about this approach is that we automatically know
that $\Hom$ is associative, which is the main computation we need
to prove
Proposition~\ref{pr_Yoneda_pairing_commutes}.\footnote{
  Another advantage of the derived category approach is that
  the cases $i=0$ and $j=0$ do not have to be treated differently
  from the $i,j>0$ cases, which the classical Yoneda Ext pairing
  requires.
  }

Therefore this is the approach we take to the Yoneda pairing.

The more classical approach to the Yoneda pairing --- which we won't use ---
is to define $E(X,Y)$ as the class of
``extensions'' of $X$ by $Y$, i.e., exact sequences
\begin{equation}\label{eq_Yoneda_ext_one}
0\to Y\to F\to X\to 0,
\end{equation} 
and to then show that (1) $E(X,Y)\isom \Ext^1(X,Y)$, and
(2) $E(X,Y)$ is a bifunctor in $X$ and $Y$; for the details to this
approach, see, for example,
\cite{hilton}, Sections~III.1 and~III.2.
This approach requires a lot less foundations, but the computations
that we need
are longer computations.
[To turn $E(X,Y)$ into a bifunctor, one needs to pullback 
\eqref{eq_Yoneda_ext_one} along a
morphism $X'\to X$ and pushout \eqref{eq_Yoneda_ext_one}
along a morphism $Y\to Y'$
(see \cite{hilton}, Section~III.1); this makes some computations we need
more involved.]\footnote{
  To construct the pairing 
  \eqref{eq_yoneda_pairing_start_of_Yoneda_subsection}, one can splice
  together sequences in the ``Yoneda Ext,'' which generalizes $E(X,Y)$
  above, in the cases $i,j>0$; however, since we only need the 
  cases $i+j=1$, we are never
  splicing together Yoneda Ext sequences.
  }

\subsubsection{First Type of Result: Prior Computations Are Correct}
\label{susu_prior_computations_are_correct}

In Definition~\ref{de_diagram_k_vs}
we defined $H^i(\cF)$, and in
Subsection~\ref{su_yoneda_pairing_simple_form_i_one_A_equals_cO}
we defined a pairing
$$
H^1(\cF)\times \Hom(\cF,\cG)\to H^1(\cG).
$$
To use the machinery of the Yoneda pairing, we have to check that
both these definitions coincide with the definitions given by the Yoneda
pairing above; i.e., we have to verify that:
\begin{enumerate}
\item 
$H^i(\cF) \eqdef \Ext^i_\cO(\cO,\cF)$ agrees with the definition of
$H^i(\cF)$ in Definition~\ref{de_diagram_k_vs}; we prove this in
Subsection~\ref{su_our_first_H_i_cF_definitions_consistent}.
\item
The pairing
$$
\Ext^1(\cO,\cF)\times \Ext^0(\cF,\cG)\to \Ext^1(\cO,\cG),
$$
given by the Yoneda pairing (as computed in the derived category)
agrees with the pairing
used in Theorem~\ref{th_weak_duality_for_cL_mec_d}
(i.e., in Subsection~\ref{su_yoneda_pairing_simple_form_i_one_A_equals_cO})
$$
H^1(\cF)\times \Hom(\cF,\cG)\to H^1(\cG),
$$
when defining $\Hom(\cF,\cG)\eqdef \Ext^0(\cF,\cG)$
(and $H^i(\cF) \eqdef \Ext^i_\cO(\cO,\cF)$);
we prove this in 
Subsection~\ref{su_first_definition_of_Yoneda_pairing_is_consistent}.
\end{enumerate}

\subsubsection{Second Needed Result: Two ``Adjointness'' Properties}

The second result we need is that
the pairing \eqref{eq_yoneda_pairing_start_of_Yoneda_subsection}
satisfies two properties that might be called ``adjointness'' properties.

To describe these properties,
for each $A,B,C$ and $i,j$, for each
$\alpha\in\Ext^i(A,B)$ and $\beta\in \Ext^j(B,C)$, let
$$
\langle \alpha, \beta\rangle
=\langle \alpha,\beta\rangle_{A,B,C,i,j}
$$
denote the image of $(\alpha,\beta)$ under the map
\eqref{eq_yoneda_pairing_start_of_Yoneda_subsection}.

\begin{enumerate}
\item 
If $f\from B_1\to B_2$ is any map, and
$$
f_* \from \Ext^i(A,B_1)\to \Ext^i(A,B_2)
\quad\mbox{and}\quad
f^* \from \Ext^j(B_2,C)\to \Ext^j(B_1,C)
$$
are the natural morphisms (of $\Ext$ as a bifunctor), then
for all $\alpha\in\Ext^i(A,B_1)$ and $\gamma\in \Ext^j(B_2,C)$
we need to prove that
\begin{equation}\label{eq_adjoint_first}
\langle \alpha, f^*\gamma \rangle = \langle f_*\alpha,\gamma \rangle.
\end{equation} 
In fact, we will only need the case $i+j=1$ (but this holds in general).
\item
If $0\to B_1\to B_2\to B_3\to 0$ is a short exact sequence,
and
$$
\delta_*\from \Ext^i(A,B_3)\to\Ext^{i+1}(A,B_1)
\quad\mbox{and}\quad
\delta^*\from \Ext^{j-1}(B_3,C)\to\Ext^j(B_1,C)
$$
are the connecting $\delta$-maps, then for all
$\alpha\in\Ext^i(A,B_3)$ and $\gamma\in\Ext^j(B_1,C)$ we have
\begin{equation}\label{eq_adjoint_second}
\langle \alpha, \delta^*\gamma \rangle = 
\langle \delta_* \alpha, \gamma \rangle.
\end{equation} 
\end{enumerate} 

For proofs of~\eqref{eq_adjoint_first}
and~\eqref{eq_adjoint_second}, see respectively
Subsections~\ref{su_proof_of_first_adjointness_property}
and~\ref{su_proof_of_second_adjointness_property}.

\subsubsection{Our Interest in Adjointness}

We use the above adjointness properties for one reason alone.

\begin{proposition}\label{pr_Yoneda_pairing_commutes}
Let $C$ be an $\cO_r$-module, and say that $H^1(C)\isom k$ and fix one
such isomorphism.  Then for any short exact sequence of $\cO_r$-modules
$0\to \cG_1\to \cG_2\to \cG_3\to 0$, the Yoneda pairing gives vertical
arrows in the diagram
%%%%%%%%%%%%%%%%%%%%%%%%%%%%%%%%%%%%%%%%%%%%%%%%%%%%%%%%%% 
\newcommand{\joelsCommutativeDiagramViaPairing}[3]{
\begin{tikzpicture}[scale=0.40]
\node (A0) at (#1 * 0.3,#2) {#3 $0$};
\node (A1) at (#1 * 1,#2) {#3 $\Hom(\cG_3,C)$};
\node (A2) at (#1 * 2,#2) {#3 $\Hom(\cG_2,C)$};
\node (A3) at (#1 * 3,#2) {#3 $\Hom(\cG_1,C)$};
\node (A4) at (#1 * 4,#2) {#3 $\Ext^1(\cG_3,C)$};
\node (A5) at (#1 * 5,#2) {#3 $\Ext^1(\cG_2,C)$};
\node (A6) at (#1 * 6,#2) {#3 $\Ext^1(\cG_1,C)$};
\node (A7) at (#1 * 6.7,#2) {#3 $0$};
\draw[->] (A0) -- (A1);
\draw[->] (A1) -- (A2);
\draw[->] (A2) -- (A3);
\draw[->] (A3) -- (A4);
\draw[->] (A4) -- (A5);
\draw[->] (A5) -- (A6);
\draw[->] (A6) -- (A7);
\node (B0) at (#1 * 0.3,#2 *0) {#3 $0$};
\node (B1) at (#1 * 1,#2 *0) {#3 $H^1(\cG_3)'$};
\node (B2) at (#1 * 2,#2 *0) {#3 $H^1(\cG_2)'$};
\node (B3) at (#1 * 3,#2 *0) {#3 $H^1(\cG_1)'$};
\node (B4) at (#1 * 4,#2 *0) {#3 $H^0(\cG_3)'$};
\node (B5) at (#1 * 5,#2 *0) {#3 $H^0(\cG_2)'$};
\node (B6) at (#1 * 6,#2 *0) {#3 $H^0(\cG_1)'$};
\node (B7) at (#1 * 6.7,#2 *0) {#3 $0$};
\draw[->] (B0) -- (B1);
\draw[->] (B1) -- (B2);
\draw[->] (B2) -- (B3);
\draw[->] (B3) -- (B4);
\draw[->] (B4) -- (B5);
\draw[->] (B5) -- (B6);
\draw[->] (B6) -- (B7);
\draw[->] (A1) -- (B1);
\draw[->] (A2) -- (B2);
\draw[->] (A3) -- (B3);
\draw[->] (A4) -- (B4);
\draw[->] (A5) -- (B5);
\draw[->] (A6) -- (B6);
\end{tikzpicture}
}
%%%%%%%%%%%%%%%%%%%%%%%%%%%%%%%%%%%%%%%%%%%%%%%%%%%%%%%%%% 
$$
\joelsCommutativeDiagramViaPairing{5}{3}{\SMALL}
$$
which is a commutative diagram
(where $\,'$ is the dual space),
and where we write $\Hom(\cG_j,C)$ in place of 
$\Ext^0(\cG_j,C)$, and $H^i(\cG_j)$ in place of $\Ext^i(\cO_r\cG_j)$.
\end{proposition}
In this proposition it is important that $\Hom(\cG_j,C)$
and $H^i(\cG_j)$ are really shorthand notation for the Ext groups
$\Ext^0(\cG_j,C)$ and $\Ext^i(\cG_j,C)$; in
the proof of Theorem~\ref{th_strong_duality_for_cL_mec_d_degree_suff_small}
below, it is crucial to use
the results of Subsubsection~\ref{susu_prior_computations_are_correct}
to see that the Yoneda pairing computed in
Theorem~\ref{th_weak_duality_for_cL_mec_d} agrees with the Yoneda
pairing on Ext groups as defined in this section
(in terms of the derived category).

To prove this
proposition we will introduce the following easy lemma.
\begin{lemma}
\label{le_pairings_and_commutative_diagram_abstract}
% Let $A,B,C,D$ be $k$-vector spaces.
% Let $\cA,\cB,\cC,\cD$ be $k$-vector spaces.
Let $U,V,W,X$ be $k$-vector spaces.
Let $g\from U\to W$ and $h\from X\to V$ be linear maps, and let
$h'\from V'\to X'$ be the dual linear map.
Let
$$
\mu=\langle \cdot,\cdot \rangle_{VU}\from V\times U\to k, \quad
\nu=\langle \cdot,\cdot \rangle_{XW}\from X\times W\to k
$$
be two bilinear forms of
$k$-vector spaces, and let 
$\tilde\mu\from U\to V'$ and
$\tilde\nu\from W\to X'$.
Then the following are
equivalent:
\begin{enumerate}
\item for each $u\in U$ and
$x\in X$ we have
% $\nu(\alpha,g\gamma)=\mu(h\alpha,\delta)$;
$\langle x,g u\rangle_{XW}=\langle hx,u\rangle_{VU}$;
and
\item the following diagram commutes:
$$
\begin{tikzpicture}[scale=0.30]
\node (A1) at (0,4) {$U$};
\node (A2) at (6,4) {$W$};
\draw[->] (A1) -- (A2) node [midway,above] {$g$};
\node (B1) at (0,0) {$V'$};
\node (B2) at (6,0) {$X'$};
\draw[->] (B1) -- (B2) node [midway,above] {$h'$};
\draw[->] (A1) -- (B1) node [midway,left] {$\tilde\mu$};
\draw[->] (A2) -- (B2) node [midway,right] {$\tilde\nu$};
\end{tikzpicture}
$$
\end{enumerate}
\end{lemma}
\begin{proof}
Let us write down condition~(2), i.e.,
what it means for the diagram above to commute:
so let 
% $\alpha\in \cA$.  Then $f\alpha\in \cC$, and the image of $f\alpha$
$u\in U$.  Then $gu\in W$, and the image of $gu$
under
$\tilde\nu$ is the linear functional $x\mapsto \langle x,gu\rangle_{XW}$.
On the other hand, the image of $u$ under $\tilde\mu$ is the
linear functional $v\mapsto \langle v,u \rangle_{VU}$; the image of this
linear functional under $h'$ is therefore the linear map
$x\mapsto \langle hx,u\rangle_{VU}$.  Hence the above diagram commutes
iff for all $u\in U$ the two linear functionals
$$
x \mapsto \langle x,gu\rangle_{XW}
\quad\mbox{and}\quad
x \mapsto \langle hx,u \rangle_{VU}
$$
are the same linear functional on $X$.  This holds iff
condition~(1) of the proposition holds.
\end{proof}

\begin{proof}[Proof of Proposition~\ref{pr_Yoneda_pairing_commutes}]
The horizontal maps are given by long exact sequences of Ext groups
associated to any short exact sequence.
The vertical maps are given by the Yoneda pairing and the fact that
--- by Subsubsection~\ref{susu_prior_computations_are_correct} ---
we can define $H^i(G_j)$ to be $\Ext^i(\cO,G_j)$, 
and define $\Hom(\cG_j,C)$ to be $\Ext^0(\cG_j,C)$,
and then our computation
of $H^i(G_j)$ as $\Ext^i(\cO,G_j)$ agrees with that used in earlier
sections, namely with Definition~\ref{de_diagram_k_vs}.  We therefore
have to check that the diagram commutes.

% Stuff commented out below are now a proposition
% First, we claim the following general result: let 
% $A\times B\to k$ and $C\times D\to k$ be two bilinear forms on
% $k$-vector spaces,
% and let $f\from A\to C$ and $g\from B\to D$ be linear maps, and
% $g'\from D'\to B'$ be the dual linear map. 
First, 
according to 
\eqref{eq_adjoint_first}, if $B_1,B_2,C$ are any elements of the
derived category and $f\from B_1\to B_2$ a morphism,
for any $\alpha\in \Ext^{1-i}(\cO,B_1)$ and
$\gamma\in \Ext^i(B_2,C)$ we have
$$
\langle \alpha,f^*\gamma \rangle
=
\langle f_*\alpha,\gamma \rangle \in \Ext^1(\cO,C)\isom H^1(C),
$$
where $\langle \cdot,\cdot\rangle$ in
$\langle \alpha,f^*\gamma \rangle$ is the Yoneda pairing
$$
\Ext^{1-i}(\cO,B_1)\times \Ext^{i}(B_1,C) \to \Ext^1(\cO,C),
$$
and similarly 
the $\langle \cdot,\cdot\rangle$ in
$\langle f_*\alpha,\gamma \rangle$ is the Yoneda pairing
$$
\Ext^{1-i}(\cO,B_2)\times \Ext^{i}(B_2,C)\to \Ext^1(\cO,C).
$$
Now apply Lemma~\ref{le_pairings_and_commutative_diagram_abstract}
with $U,V,W,X$ there being, respectively,
$$
\Ext^i(B_2,C),\ H^{1-i}(B_2)\eqdef\Ext^{1-i}(\cO,B_2),
\ \Ext^i(B_1,C),\ H^{1-i}(B_1)\eqdef \Ext^{1-i}(\cO,B_1)
$$
(the two instances of $\eqdef$ above are using the results of
Subsubsection~\ref{susu_prior_computations_are_correct});
we get a commutative diagram:
$$
\begin{tikzpicture}[scale=0.30]
\node (A1) at (0,4) {$\Ext^i(B_2,C)$};
\node (A2) at (8,4) {$\Ext^i(B_1,C)$};
\node (B1) at (0,0) {$H^{1-i}(B_2)'$};
\node (B2) at (8,0) {$H^{1-i}(B_1)'$};
\draw[->] (A1) -- (B1);
\draw[->] (A1) -- (A2) node [midway,above] {$f^*$};
\draw[->] (B1) -- (B2) node [midway,below] {$(f_*)'$};
\draw[->] (A2) -- (B2);
\end{tikzpicture}
$$
Taking $f\from B_1\to B_2$ to be $\cG_2\to\cG_3$
(in the short exact sequence) gives a commutative diagram
$$
\begin{tikzpicture}[scale=0.30]
\node (A1) at (0,4) {$\Ext^i(\cG_3,C)$};
\node (A2) at (8,4) {$\Ext^i(\cG_2,C)$};
\node (B1) at (0,0) {$H^{1-i}(\cG_3)'$};
\node (B2) at (8,0) {$H^{1-i}(\cG_2)'$};
\draw[->] (A1) -- (B1);
\draw[->] (A1) -- (A2);
\draw[->] (B1) -- (B2);
\draw[->] (A2) -- (B2);
\end{tikzpicture}
$$
and similarly with $f\from B_1\to B_2$ to be $\cG_1\to\cG_2$ gives a 
commutative digram
$$
\begin{tikzpicture}[scale=0.30]
\node (A1) at (0,4) {$\Ext^i(\cG_2,C)$};
\node (A2) at (8,4) {$\Ext^i(\cG_1,C)$};
\node (B1) at (0,0) {$H^{1-i}(\cG_2)'$};
\node (B2) at (8,0) {$H^{1-i}(\cG_1)'$};
\draw[->] (A1) -- (B1);
\draw[->] (A1) -- (A2);
\draw[->] (B1) -- (B2);
\draw[->] (A2) -- (B2);
\end{tikzpicture}
$$
Taking $i=1$ we see that the two leftmost squares of the diagram
in the proposition commute,
and taking $i=0$ gives the two rightmost squares.

Hence it suffices to check that the middle square commutes.
But this follows from the analogous argument as above,
and with
\eqref{eq_adjoint_second}
replacing
\eqref{eq_adjoint_first}, and applying
Lemma~\ref{le_pairings_and_commutative_diagram_abstract} with
$U,V,W,X$ being
$$
\Hom(\cG_1,C)\eqdef\Ext^0(\cG_1,C), \ H^1(\cG_1)\eqdef \Ext^1(\cO,\cG_1),
\ \Ext^1(\cG_3,C),\ H^0(\cG_3)\eqdef \Ext^0(\cO,\cG_3).
$$
\end{proof}

\subsection{The Duality Commutative Diagram}

\begin{theorem}
\label{th_duality_commutative_diagram}
Let $\cO$ be a diagram of $k$-algebras, let $\cM$ be a flat $\cO$-module,
and let $\cM'$ be any $\cO$-module.
Say that $\omega=\cM\otimes\cM'$ satisfies $H^1(\omega)\isom k$,
and fix such an isomorphism.
Then for $i=0,1$ and any $\cO$-module, $\cF$, there is a natural map
$$
\Ext^i(\cF,\cM')\to H^{1-i}(\cM\otimes\cF)'
$$
such that for any short exact sequence
$0\to\cF_1\to\cF_2\to\cF_3$ of $\cO$-modules, the two resulting
long exact sequences (given by horizontal arrows) fit into a
commutative diagram:
\newcommand{\joelsDualityCommutativeDiagram}[3]{
\begin{tikzpicture}[scale=0.40]
\node (A0) at (#1 * 0.3,#2) {#3 $0$};
\node (A1) at (#1 * 1,#2) {#3 $\Hom(\cF_3,\cM')$};
\node (A2) at (#1 * 2,#2) {#3 $\Hom(\cF_2,\cM')$};
\node (A3) at (#1 * 3,#2) {#3 $\Hom(\cF_1,\cM')$};
\node (A4) at (#1 * 4,#2) {#3 $\Ext^1(\cF_3,\cM')$};
\node (A5) at (#1 * 5,#2) {#3 $\Ext^1(\cF_2,\cM')$};
\node (A6) at (#1 * 6,#2) {#3 $\Ext^1(\cF_1,\cM')$};
\node (A7) at (#1 * 6.7,#2) {#3 $0$};
\draw[->] (A0) -- (A1);
\draw[->] (A1) -- (A2);
\draw[->] (A2) -- (A3);
\draw[->] (A3) -- (A4);
\draw[->] (A4) -- (A5);
\draw[->] (A5) -- (A6);
\draw[->] (A6) -- (A7);
\node (B0) at (#1 * 0.3,#2 *0) {#3 $0$};
\node (B1) at (#1 * 1,#2 *0) {#3 $H^1(\cM\otimes\cF_3)'$};
\node (B2) at (#1 * 2,#2 *0) {#3 $H^1(\cM\otimes\cF_2)'$};
\node (B3) at (#1 * 3,#2 *0) {#3 $H^1(\cM\otimes\cF_1)'$};
\node (B4) at (#1 * 4,#2 *0) {#3 $H^0(\cM\otimes\cF_3)'$};
\node (B5) at (#1 * 5,#2 *0) {#3 $H^0(\cM\otimes\cF_2)'$};
\node (B6) at (#1 * 6,#2 *0) {#3 $H^0(\cM\otimes\cF_1)'$};
\node (B7) at (#1 * 6.7,#2 *0) {#3 $0$};
\draw[->] (B0) -- (B1);
\draw[->] (B1) -- (B2);
\draw[->] (B2) -- (B3);
\draw[->] (B3) -- (B4);
\draw[->] (B4) -- (B5);
\draw[->] (B5) -- (B6);
\draw[->] (B6) -- (B7);
\draw[->] (A1) -- (B1);
\draw[->] (A2) -- (B2);
\draw[->] (A3) -- (B3);
\draw[->] (A4) -- (B4);
\draw[->] (A5) -- (B5);
\draw[->] (A6) -- (B6);
\end{tikzpicture}
}
%%%%%%%%%%%%%%%%%%%%%%%%%%%%%%%%%%%%%%%%%%%%%%%%%%%%%%%%%% 
$$
\joelsDualityCommutativeDiagram{5}{3}{\Tiny}
$$
Moreover, one has the following ``two out of three'' principle:
say that two of $\cF_1,\cF_2,\cF_3$ satisfy strong duality, i.e., for
two of $j=1,2,3$ we have
$$
\Ext^i(\cF_j,\cM')\to H^{1-i}(\cM\otimes\cF_j)'
$$
is an isomorphism for both $i=0,1$; then all the $\cF_1,\cF_2,\cF_3$
satisfy strong duality.
\end{theorem}
\begin{proof}
We combine the commutative diagram at the end of
Subsection~\ref{su_flat_module_and_ext_groups},
replacing $\cG$ with $\cM'$ (hence $\omega=\cM\otimes\cM'$)
with the commutative diagram
in Proposition~\ref{pr_Yoneda_pairing_commutes},
setting $\cG_i=\cM\otimes\cF_i$ and $C=\omega$.
This gives the desired commutative diagram.

The ``two out of three'' principle follows from the five-lemma.
\end{proof}

\subsection{The Fundamental Lemma}

\begin{lemma}\label{le_vanishing_both_strong_duality_deg_mec_d_small}
Fix an $r$-periodic perfect matching $W$ and an $\mec M\in\integers^2$.  For any
$\mec d$ with $\deg(\mec d)$ sufficiently small, we have
\begin{enumerate}
\item 
$H^0(\cM_{W,\mec M}\otimes \cL_{r,\mec d}) = 0$, and
\item 
$\Ext^1_{\cO_r}(\cL_{r,\mec d},\cM_{W,\mec M})=0$. 
\end{enumerate}
\end{lemma}
\begin{proof}
(1): we have
$$
\beta^0(\cM_{W,\mec M}\otimes \cL_{r,\mec d}) = \beta^0(\cM_{W,\mec M+\mec d})
= \bigl| 
\{ \mec d' \ | \ \mec d'\le \mec M+\mec d,\ W(\mec d')=1 \}
\bigr|
$$
which vanishes for $\deg(\mec d)$ sufficiently small, since $\mec M$ is
fixed and $W$ has
bounded support.

(2): 
fix any $\mec L=\mec K+\mec 1$.
By Proposition~III.6.7 of \cite{hartshorne} we have\footnote{
  To apply Proposition~III.6.7 of \cite{hartshorne} as is, we have
  to verify that $\cL_{r,-\mec d}$ 
  equals the dual sheaf of $\cL_{r,\mec d}$, i.e.,
  to verify that $\cL^\vee\eqdef\SHom_{\cO_r}(\cL_{r,\mec d},\cO_r)$
  equals $\cL_{r,-\mec d}$.
  We leave this to the interested reader, or notice that we have
  essentially alrealy verified this in
  Proposition~\ref{pr_easy_cL_is_invertible_tensoring_with_cL_etc},
  \eqref{eq_hom_L_otimes_cF_move_L_to_other_side}: indeed, this equation
  implies that
  for all $\cO_r$-modules, $\cF$, we have
  $\Hom(\cL_{r,\mec d}\otimes\cF,\cO)\isom \Hom(\cF,\cL_{r,-\mec d})$:
  moreover,
  Exercise~II.5.1 of \cite{hartshorne}, implies that for any ringed
  space $(X,\cO)$, and locally free $\cO$-module $\cL$ we have
  $\Hom(\cL\otimes\cF,\cO)\isom \Hom(\cF,\cL^\vee)$ where
  $\cL^\vee$ is the dual sheaf of $\cL$;
  hence $\cL_{r,\mec d}^\vee$ and $\cL_{r,-\mec d}$ represent the same
  functor, and are therefore isomorphic by Yoneda's lemma.
  One can also entirely circumvent Proposition~III.6.7 of 
  \cite{hartshorne}, by using the proofs of Lemma~III.6.6 
  and Proposition~III.6.7, which results in a more direct argument:
  namely, $\Ext^i(\cF\otimes\cL_{r,\mec d},\cG)\isom
  \Ext^i(\cF,\cL_{r,-\mec d}\otimes\cG)$ for all $i$ since they
  agree for $i=0$ by
  Proposition~\ref{pr_easy_cL_is_invertible_tensoring_with_cL_etc},
  \eqref{eq_hom_L_otimes_cF_move_L_to_other_side},
  and both vanish for all $\cG$ injective using
  \eqref{eq_hom_L_otimes_cF_move_L_to_other_side}
  and the exactness of tensoring by $\cL_{r,\mec d}$, which results from
  \eqref{eq_hom_L_otimes_on_both_sides_does_nothing}.
  }
$$
\Ext^1_{\cO_r}(\cL_{r,\mec d},\cM_{W,\mec M})
\isom \Ext^1_{\cO_r}(\cO_r,\cL_{r,-\mec d}\otimes\cM_{W,\mec M}),
$$
and by 
Proposition~\ref{pr_easy_cL_is_invertible_tensoring_with_cL_etc},
the right-hand-side is
$$
\isom
\Ext^1_{\cO_r}(\cO_r,\cM_{W,\mec M-\mec d}),
$$
and
according to \cite{hartshorne}, Proposition~III.6.3(c)
$$
\Ext^1_{\cO_r}(\cO_r,\cM_{W,\mec M-\mec d})
\isom
H^1(\cM_{W,\mec M-\mec d})
$$
which for any fixed $\mec L=\mec K+\mec 1$ is isomorphic to
$$
H^0(\cM_{W^*_\mec L,\mec K-\mec M+\mec d})
\isom
H^0(\cM_{W^*_\mec L,\mec K-\mec M}\otimes\cL_{r,\mec d})
$$
but by (1) (with $W,\mec M$ respectively replaced by
$W^*_\mec L,\mec K-\mec M$), this vanishes for $\deg(\mec d)$ 
sufficiently small (since $\mec K,\mec M$ are fixed).
\end{proof}

\subsection{Strong Duality for $\cL_\mec d$ with $\deg(\mec d)$ Sufficiently
Small}

\begin{theorem}
\label{th_strong_duality_for_cL_mec_d_degree_suff_small}
Let $W$ be an $r$-periodic perfect matching,
let $\cM=\cM_{W,\mec 0}$, and for some $\mec L=\mec K+\mec 1$
let $\cM'=\cM_{W^*_\mec L,\mec K}$.
Then for all $\mec d$ with $\deg(\mec d)$ sufficiently small we have 
that $\cF=\mec L_\mec d$ satisfies strong duality, i.e., the map
\begin{equation}\label{eq_strong_duality_equation_for_cL_mec_d_small}
\Ext^i(\cF,\cM')\to H^{1-i}(\cM\otimes\cF)'
\end{equation} 
is an isomorphism for $i=0,1$.
\end{theorem}
\begin{proof}
For $i=0$ this follows from
Theorem~\ref{th_weak_duality_for_cL_mec_d}
and Subsubsection~\ref{susu_prior_computations_are_correct}.
According to 
Lemma~\ref{le_vanishing_both_strong_duality_deg_mec_d_small},
both sides of 
\eqref{eq_strong_duality_equation_for_cL_mec_d_small} are $0$.
\end{proof}

\subsection{Strong Duality for Skyscrapers and All $\cL_\mec d$}

\begin{theorem}
\label{th_strong_duality_for_skyscrapers_and_cL_mec_d}
The skyscraper diagrams $\cS_1$ and $\cS_2$ satisfy strong duality, as
well as all diagrams $\cL_\mec d$.
\end{theorem}
\begin{proof}
Consider the short exact sequence:
\begin{equation}\label{eq_skyscraper_S_one_exact_sequence}
0 \to \cL_\mec d \to \cL_{\mec d+\mec e_1} \to \cS_1 \to 0,
\end{equation}
By Theorem~\ref{th_strong_duality_for_cL_mec_d_degree_suff_small},
both $\cL_\mec d$ and $\cL_{\mec d+\mec e_1}$ satisfy strong duality
for $\deg(\mec d)$ sufficiently small.
Taking such a $\mec d$, the ``two out of three principle''
(Theorem~\ref{th_duality_commutative_diagram}) implies
that $\cS_1$ also satisfies strong duality.

Similarly $\cS_2$ satisfies strong duality.

Now consider an arbitrary $\mec d\in\integers^2$.  According to
Theorem~\ref{th_strong_duality_for_cL_mec_d_degree_suff_small}, for
$a_0\in\integers$ sufficiently large we have
$\cL_{\mec d-a_0\mec e_1}$ satisfies strong duality.
So fix such an $a_0$.
Then, considering the exact sequence 
$$
0 \to \cL_{\mec d-a\mec e_1} \to \cL_{\mec d-(a-1)\mec e_1} \to \cS_1 \to 0
$$
with $a=a_0,a_0-1,a_0-2,\ldots$, and repeatedly
applying the ``two out of three''
principle, we see that $\cL_{\mec d-a'\mec e_1}$ 
for $a'=a_0-1,a_0-2,\ldots,0$.
\end{proof}

\begin{remark}
Since the value of 
$\cS_1$ vanishes at all points except at $B_1$, where its value
is $k$, it follows that $H^0(\cS_1)\isom k$ and $H^1(\cS_1)=0$.
Similarly for $\cS_2$.
\end{remark}

\begin{remark}
Any $\cO_r=\cO_{r,k}$-module
can also be regarded as a sheaf of $k$-vector spaces, and
hence a $\underline k$ module.
A general theorem tells us that for any $k$-diagram, $\cF$,
$H^i_{\underline k}(\cF)\isom H^i_{\underline\cO}$,
and hence we can compute $H^i$ as $\underline k$-modules or
$\cO_r$-modules.
However,
$$
\Hom_{\cO_r}(\cF,\cM_{W^*_\mec L}) 
\mbox{\quad and\quad}
\Hom_{\underline k}(\cF,\cM_{W^*_\mec L}) 
$$
are very different; in particular, we will later explain that
the second space infinite dimensional for $\cF=\cL_\mec d$.
Hence it is essential for us to work with $\cO_r$-modules to prove
our strong duality theorems.
\end{remark}

\appendix
\section{The Yoneda Pairing, Skyscrapers, and CoSkyscrapers}
\label{ap_yoneda_pairing_etc}

In this section we address a number of points in homological algebra.
The reader familiar with algebraic geometry will likely find all these
points believable --- if not already known to them --- and can safely
skip the details.

Again, in this section $k$ will be a fixed field, and we will often
suppress $k$, writing, e.g., $\cO_r$ instead of $\cO_{r,k}$.
Also, we sometimes suppress the subscript $r$ in $\cO_r$, or
suppress $\cO$ altogether (e.g., writing $\Ext^i(\cF,\cG)$ instead of
$\Ext^i_{\cO_r}(\cF,\cG)$) if confusion is unlikely.

The results in this section actually hold where $\cO$ is an arbitrary
diagram of rings; however, in 
Subsections~\ref{su_our_first_H_i_cF_definitions_consistent}
and onward,
we insist that $\cO$ is a diagram of $k$-algebras simply because
our definitions of $H^i(\cF)$, of $\Ext$ groups, etc.,
insist on working with $k$-diagrams
(of vector spaces).

\subsection{Background on Ext groups and the Derived Category}
\label{su_background_on_Ext_and_derived_category}

For some of the facts we need about Ext groups, we refer to
\cite{hartshorne}, Sections~III.1 and~III.6.
For facts about the derived category, see
\cite{gelfand} Chapter~III
and \cite{weibel} Chapter~10.

Note that \cite{gelfand} defines $\Ext^i(\cF,\cG)$ as
$\Hom_\cD(\cF[0],\cG[i])$, where $\cD=\cD(\cO)$ is the derived category
of $\cO$-modules and
$\cG[i]$ denotes the complex that is all zeros except for a $\cG$ in
position $-i$
(see Remark~III.5.4(b), page~166 there);
only later does \cite{gelfand} show that this definition
coincides with the usual derived functor definition of $\Ext^i(\cF,\cG)$
(see Subsection~III.6.14, page~194 there).
By contrast, \cite{weibel} defines $\Ext^i(\cF,\cG)$ in terms of derived
functors of $\Ext^i(\cF,\cdot)$ (Definition~2.5.2, page~50),
and shows that this agrees with those of $\Ext^i(\cdot,\cG)$
(Theorem~2.7.6, page~63) (and it follows from the Freyd Full Embedding
Theoerm\footnote{
  See the comments in \cite{hartshorne},
  Section~III.1, page~203, or \cite{kashiwara}, Theorem~9.6.10, page~238.
  Moreover, below we will show that the category
  of $\cO_{r,k}$-modules have enough injectives and projectives.
  }
that the same holds in the category of $\cO'$-modules for any ringed
space $(X,\cO')$, provided that this category has enough projectives
and injectives.  Since the category of $\cO_{r,k}$-modules is equivalent
to such a category, we conclude this holds for $\cO_{r,k}$-modules.
Then \cite{weibel} shows that 
$\Ext^i(\cF,\cG)$ and agrees with the $i$-th hyperext of $\cF[0],\cG[0]$
(Corollary~10.7.5, page~400 there; \cite{weibel} writes $\cF$ for both
$\cF$ and the complex $\cF[0]$).

\subsection{Skyscrapers and Coskyscrapers for $\cO$-modules}

We will need some a projective resolution for $\cO$, and to know that
the category of $\cO$-modules have enough injectives and enough projectives.
This will
involve what we call {\em coskyscraper} $\cO$-modules
which are a sort of dual diagram
to the commonly used {\em skyscraper $\cO$-modules} that one uses in
sheaf theory.

[All statements above are well-known to experts, but it seems difficult
to find this anywhere, at least in our simple context.
Once one realizes that the category of diagrams of $\cO$-modules
is equivalent to sheaves of $\cO'$-modules for a ringed space
$(X,\cO')$, then it is well known that there are enough injectives.]

What we do in this subsection is valid over a general diagram of
rings $\cO$; let us fix notation, setting
\begin{equation}\label{eq_diagram_of_rings_notation}
R_i=\cO(B_i), \ S_j=\cO(A_j),
\ \sigma_{ij}=\cO(\rho_{ij}),
\end{equation} 
and see Figure~\ref{fi_diagram_of_rings_notation}.
\begin{figure}[ht]
$$
\begin{tikzpicture}[scale=0.30]
\tikzmath{
  \y = 4;
  \x = 15;
}
\node (B1) at (0,\y) {$\cO(B_1)=R_1$};
\node (B3) at (0,0) {$\cO(B_3)=R_3$};
\node (B2) at (0,-\y) {$\cO(B_2)=R_2$};
\node (A1) at (\x,\y/2) {$\cO(A_1)=S_1$};
\node (A2) at (\x,-\y/2) {$\cO(A_2)=S_2$};
\draw[->] (B1) -- (A1) node [midway,above] {\Small $\sigma_{11}=\cO(\rho_{11})$};
\draw[->] (B3) -- (A1);
\draw[->] (B2) -- (A2);
\draw[->] (B3) -- (A2);
\end{tikzpicture}
$$
\caption{In this section we work with a general diagram of rings, $\cO$.
We introduce the notation
$R_i=\cO(B_i)$, $S_j=\cO(A_j)$,
and $\sigma_{ij}=\cO(\rho_{ij})$.
}
\label{fi_diagram_of_rings_notation}
\end{figure}
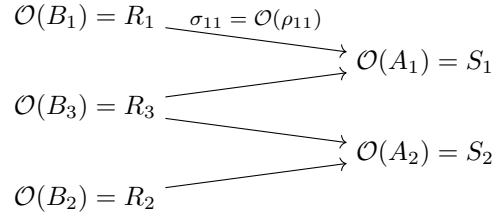

\subsubsection{Coskyscraper Diagrams}

\begin{proposition}\label{pr_coskyscrapers}
Let $\cO$ be a diagram of rings
For $i\in[3]$ and $M$ and $R_i$-module, define the {\em coskyscraper at
$B_i$ of value $M$}, denoted
${\rm CoSky}(B_i,M)$, to be the $\cO$-module depicted below:
$$
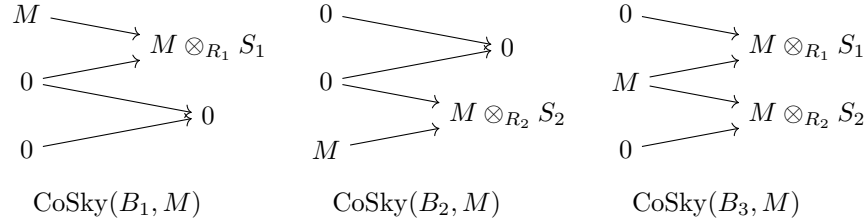
\begin{tikzpicture}[scale=0.30]
\tikzmath{
  \y = 3;
  \x = 8;
}
\node (B1) at (0,\y) {$M$};
\node (B3) at (0,0) {$0$};
\node (B2) at (0,-\y) {$0$};
\node (A1) at (\x,\y/2) {$M\otimes_{R_1}S_1$};
\node (A2) at (\x,-\y/2) {$0$};
\draw[->] (B1) -- (A1);
\draw[->] (B3) -- (A1);
\draw[->] (B2) -- (A2);
\draw[->] (B3) -- (A2);
\node at (\x *0.5,-\y*1.8) {${\rm CoSky}(B_1,M)$};
\end{tikzpicture}
\quad
\begin{tikzpicture}[scale=0.30]
\tikzmath{
  \y = 3;
  \x = 8;
}
\node (B1) at (0,\y) {$0$};
\node (B3) at (0,0) {$0$};
\node (B2) at (0,-\y) {$M$};
\node (A1) at (\x,\y/2) {$0$};
\node (A2) at (\x,-\y/2) {$M\otimes_{R_2}S_2$};
\draw[->] (B1) -- (A1);
\draw[->] (B3) -- (A1);
\draw[->] (B2) -- (A2);
\draw[->] (B3) -- (A2);
\node at (\x *0.5,-\y*1.8) {${\rm CoSky}(B_2,M)$};
\end{tikzpicture}
\quad
\begin{tikzpicture}[scale=0.30]
\tikzmath{
  \y = 3;
  \x = 8;
}
\node (B1) at (0,\y) {$0$};
\node (B3) at (0,0) {$M$};
\node (B2) at (0,-\y) {$0$};
\node (A1) at (\x,\y/2) {$M\otimes_{R_1}S_1$};
\node (A2) at (\x,-\y/2) {$M\otimes_{R_2}S_2$};
\draw[->] (B1) -- (A1);
\draw[->] (B3) -- (A1);
\draw[->] (B2) -- (A2);
\draw[->] (B3) -- (A2);
\node at (\x *0.5,-\y*1.8) {${\rm CoSky}(B_3,M)$};
\end{tikzpicture}
$$
Then for any $\cF$ we have an isomorphism
that is functorial in $\cF$:
\begin{equation}
\label{eq_hom_from_coskyscraper_to_diagram_and_hom_on_B_values}
\Hom_\cO\bigl( {\rm CoSky}(B_i,M),\cF\bigr)
\isom
\Hom_{\cO(B_i)}\bigl( M,\cF(B_i)\bigr)
\end{equation} 
which takes $\phi\in \Hom_\cO\bigl( {\rm CoSky}(B_i,M),\cF\bigr)$ to
$\phi(B)$.
Similarly for any $j\in[2]$, an $S_j$-module $M$,  and the diagrams
$$
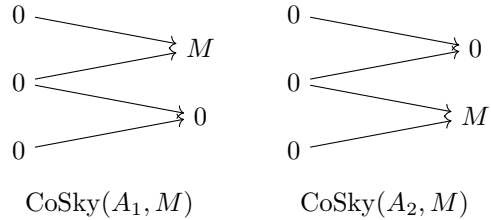
\begin{tikzpicture}[scale=0.30]
\tikzmath{
  \y = 3;
  \x = 8;
}
\node (B1) at (0,\y) {$0$};
\node (B3) at (0,0) {$0$};
\node (B2) at (0,-\y) {$0$};
\node (A1) at (\x,\y/2) {$M$};
\node (A2) at (\x,-\y/2) {$0$};
\draw[->] (B1) -- (A1);
\draw[->] (B3) -- (A1);
\draw[->] (B2) -- (A2);
\draw[->] (B3) -- (A2);
\node at (\x *0.5,-\y*1.8) {${\rm CoSky}(A_1,M)$};
\end{tikzpicture}
\qquad
\begin{tikzpicture}[scale=0.30]
\tikzmath{
  \y = 3;
  \x = 8;
}
\node (B1) at (0,\y) {$0$};
\node (B3) at (0,0) {$0$};
\node (B2) at (0,-\y) {$0$};
\node (A1) at (\x,\y/2) {$0$};
\node (A2) at (\x,-\y/2) {$M$};
\draw[->] (B1) -- (A1);
\draw[->] (B3) -- (A1);
\draw[->] (B2) -- (A2);
\draw[->] (B3) -- (A2);
\node at (\x *0.5,-\y*1.8) {${\rm CoSky}(A_2,M)$};
\end{tikzpicture}
$$
and the isomorphism
\begin{equation}
\label{eq_hom_from_coskyscraper_to_diagram_and_hom_on_A_values}
\Hom_\cO\bigl( {\rm CoSky}(A_j,M),\cF\bigr)
\isom
\Hom_{\cO(A_j)}\bigl( M,\cF(A_j)\bigr)
\end{equation} 
taking $\phi$ to $\phi(A_j)$.
\end{proposition}
(To understand the meaning of $M\otimes_{R_i} S_j$,
notice that if $M$ is an $R_i$-module, and $\rho_{ij}$ is a restriction
map for some $i\in[3]$ and $j\in[2]$, then $S_j$ is an $R_i$-module via
the map $\sigma_{ij}\from R_i\to S_j$.
Hence one can take the tensor product $M\otimes_{R_i} S_j$,
which is an $S_j$-module with $S_j$ acting on the
second factor.)
\begin{proof}
We leave the details to the reader.
For some details, see
\cite{folinsbee_friedman_two_vertex}, Section~10.1.
\end{proof}
We remark that this proposition follows from the general construction
of adjoints, \cite{sga4.1}, Exp.~I, 5.1; see
\cite{friedman_cohomology}, Section~2.1, specifically the discussion
around equation~(6) (this is done for diagrams of $k$-vector spaces,
and the above proposition
is the generalization to $\cO$-modules for a diagram of
rings, $\cO$).
We also remark that the above proposition is very similar to what one
sees in algebraic geometry.

Here is an immediate consequence of 
Proposition~\ref{pr_coskyscrapers}.

\begin{theorem}\label{th_projectives_from_coskyscrapers}
For some
$P\in X=\{A_1,A_2,B_1,B_2,B_3\}$, let
$M$ be a projective $\cO(P)$-module.
Then ${\rm CoSky}(P,M)$ is a projective $\cO$-module.
\end{theorem}
It is a standard result that $M$ is a projective $R$-module
iff it is a direct summand of a free $R$-module, i.e., of a
direct sum of copies of $R$.

\subsubsection{Skyscraper Diagrams}
\label{su_skyscraper_diagrams_explained}

Similarly to coskyscraper diagrams, for each 
$P\in X=\{A_1,A_2,B_1,B_2,B_3\}$, if $M$ is an $\cO(P)$, then for
any $\cF$ one has the isomorphism
\begin{equation}\label{eq_hom_to_skyscraper_from_diagram_and_hom_on_P_values}
\Hom_\cO\bigl( \cF,{\rm Sky}(P,M)\bigr)
\isom
\Hom_{\cO(P)}\bigl( \cF(P),M\bigr)
\end{equation} 
where ${\rm Sky}(P,M)$ is the {\em skyscraper diagram at $P$ with value
$M$} which are defined by:
$$
\begin{tikzpicture}[scale=0.30]
\tikzmath{
  \y = 3;
  \x = 8;
}
\node (B1) at (0,\y) {$M$};
\node (B3) at (0,0) {$0$};
\node (B2) at (0,-\y) {$0$};
\node (A1) at (\x,\y/2) {$0$};
\node (A2) at (\x,-\y/2) {$0$};
\draw[->] (B1) -- (A1);
\draw[->] (B3) -- (A1);
\draw[->] (B2) -- (A2);
\draw[->] (B3) -- (A2);
\node at (\x *0.5,-\y*1.8) {${\rm Sky}(B_1,M)$};
\end{tikzpicture}
\qquad
\begin{tikzpicture}[scale=0.30]
\tikzmath{
  \y = 3;
  \x = 8;
}
\node (B1) at (0,\y) {$0$};
\node (B3) at (0,0) {$M$};
\node (B2) at (0,-\y) {$0$};
\node (A1) at (\x,\y/2) {$0$};
\node (A2) at (\x,-\y/2) {$0$};
\draw[->] (B1) -- (A1);
\draw[->] (B3) -- (A1);
\draw[->] (B2) -- (A2);
\draw[->] (B3) -- (A2);
\node at (\x *0.5,-\y*1.8) {${\rm Sky}(B_3,M)$};
\end{tikzpicture}
\qquad
\begin{tikzpicture}[scale=0.30]
\tikzmath{
  \y = 3;
  \x = 8;
}
\node (B1) at (0,\y) {$M$};
\node (B3) at (0,0) {$M$};
\node (B2) at (0,-\y) {$0$};
\node (A1) at (\x,\y/2) {$M$};
\node (A2) at (\x,-\y/2) {$0$};
\draw[->] (B1) -- (A1);
\draw[->] (B3) -- (A1);
\draw[->] (B2) -- (A2);
\draw[->] (B3) -- (A2);
\node at (\x *0.5,-\y*1.8) {${\rm Sky}(A_1,M)$};
\end{tikzpicture}
$$
and similarly for ${\rm Sky}(B_2,M)$ and ${\rm Sky}(A_2,M)$.
Here we note that if $M$ is a module over $\cO(A_j)=S_j$,
and $\rho_{ij}$ is a restriction,
then $M$ is automatically an $R_i$ module via the map
$\sigma_{ij}\from R_i\to S_j$.
Again we leave the proof of 
\eqref{eq_hom_to_skyscraper_from_diagram_and_hom_on_P_values} to the reader, 
and refer the reader to 
\cite{folinsbee_friedman_two_vertex}, Section~10.1, for a diagram
and helpful discussion.

We therefore get the following ``dual theorem'' to
Theorem~\ref{th_projectives_from_coskyscrapers}.

\begin{theorem}\label{th_injectives_from_skyscrapers}
For some
$P\in X=\{A_1,A_2,B_1,B_2,B_3\}$, let
$M$ be an injective $\cO(P)$-module.
Then ${\rm Sky}(P,M)$ is an injective $\cO$-module.
\end{theorem}

\subsection{$\cO$-modules Have Enough Injectives and Projectives}
\label{su_enough_injectives_and_projectives}

For any ring, $R$, the category of $R$-modules has enough injectives
and projectives (see \cite{weibel}, Sections~2.2 and~2.3).

If $\cF$ is any $\cO$-module, let us show that there is an injection
$\cF\to\cI$ where $\cI$ is an injective $\cO$-module.
For each $P\in X=\{A_1,A_2,B_1,B_2,B_3\}$ let
$\cF(P)\to I_P$ be an injection to an injective $\cO(P)$-module, $I_P$.
Then
$$
\cF\to \cI\eqdef \bigoplus_{P\in X} {\rm Sky}(P,I_P)
$$
be the natural map 
using \eqref{eq_hom_to_skyscraper_from_diagram_and_hom_on_P_values}.
The map $\cF\to\cI$ is injective since $\cF\to {\rm Sky}(P,I_P)$ is
injective at $P$.  Since $I_P$ is injective, then 
\eqref{eq_hom_to_skyscraper_from_diagram_and_hom_on_P_values} shows that
${\rm Sky}(P,I_P)$ is an injective $\cO$-module;
hence $\cI$ an injective $\cO$-module.

Hence the category of $\cO$-modules has enough injectives.
Similarly any $\cO$-module, $\cF$, has a surjective map $\cP\to \cF$
where 
$$
\cP = \bigoplus_{P\in X} {\rm CoSky}(P,J_P)
$$
where $J_P\to \cF(P)$ is a surjection and $J_P$ is a projective
$\cO(P)$ module.  Hence $\cP$ is a projective $\cO$-module.
Hence the category of $\cO$ also has enough projectives.

\subsection{A Projective Resolution of $\cO$}
\label{su_a_projective_resolution_of_cO}

In this subsection we build a useful projective resolution of $\cO$.

To do so, for each $i\in[3]$, the identity map
$R_i\to R_i=\cO(B_i)$ determines via
\eqref{eq_hom_from_coskyscraper_to_diagram_and_hom_on_B_values}
a map
$$
{\rm CoSky}(B_i,R_i) \xrightarrow{\phi_i} \cO .
$$
Taking direct sums we get a morphism $d_0=\phi_1\oplus\phi_2\oplus\phi_3$ from
$$
\cP_0 = {\rm CoSky}(B_1,R_1) \oplus {\rm CoSky}(B_2,R_2) \oplus
{\rm CoSky}(B_3,R_3)
\xrightarrow{d_0=\phi_1\oplus\phi_2\oplus\phi_3}\cO,
$$
which we illustrate as
$$
\begin{tikzpicture}[scale=0.30]
\tikzmath{
  \y = 3;
  \x = 8;
}
\node (B1) at (0,\y) {$R_1$};
\node (B3) at (0,0) {$R_3$};
\node (B2) at (0,-\y) {$R_2$};
\node (A1) at (\x,\y/2) {$S_1\oplus S_1$};
\node (A2) at (\x,-\y/2) {$S_2\oplus S_2$};
\draw[->] (B1) -- (A1);
\draw[->] (B3) -- (A1);
\draw[->] (B2) -- (A2);
\draw[->] (B3) -- (A2);
\node (Name) at (\x *0.5,-\y*2.0) {\large$\cP_0$};
\tikzmath{
  \xoffset = 20;
}
\node (oB1) at (\xoffset+0,\y) {$R_1$};
\node (oB3) at (\xoffset+0,0) {$R_2$};
\node (oB2) at (\xoffset+0,-\y) {$R_3$};
\node (oA1) at (\xoffset+\x,\y/2) {$S_1$};
\node (oA2) at (\xoffset+\x,-\y/2) {$S_2$};
\draw[->] (oB1) -- (oA1);
\draw[->] (oB3) -- (oA1);
\draw[->] (oB2) -- (oA2);
\draw[->] (oB3) -- (oA2);
\node (oName) at (\xoffset+\x *0.5,-\y*2.0) {\large$\cO$};
\draw[->] (B1) -- (oB1) [dashed];
\draw[->] (B2) -- (oB2) [dashed];
\draw[->] (B3) -- (oB3) [dashed];
\draw[->] (A1) -- (oA1) [dashed];
\draw[->] (A2) -- (oA2) [dashed];
\draw[->] (Name) -- (oName) [dashed] node [midway,above] {\large$d_0$};
\end{tikzpicture}
$$
(here the morphism $d_0$ and the $d_0(P)$ are illustrated in dashed lines,
and the restriction maps of each of $\cP_0$ and $\cO$ are solid lines).
We therefore see the kernel of $d_0$ is $\cP_1$, given by
the kernels of the individual $d_0(P)$:
$$
\begin{tikzpicture}[scale=0.30]
\tikzmath{
  \y = 3;
  \x = 8;
}
\node (B1) at (0,\y) {$0$};
\node (B3) at (0,0) {$0$};
\node (B2) at (0,-\y) {$0$};
\node (A1) at (\x,\y/2) {$S_1$};
\node (A2) at (\x,-\y/2) {$S_2$};
\draw[->] (B1) -- (A1);
\draw[->] (B3) -- (A1);
\draw[->] (B2) -- (A2);
\draw[->] (B3) -- (A2);
\node (Name) at (\x *0.5,-\y*1.5) {$\cP_1$};
\end{tikzpicture}
$$
Therefore
$$
\cP_1 \isom {\rm CoSky}(A_1,S_1) \oplus {\rm CoSky}(A_2,S_2)
$$
is projective,
and we get the exact sequence
of $\cO$-modules:
\begin{equation}\label{eq_projective_resolution_of_cO}
0\to \cP_1\xrightarrow{d_1}\cP_0\xrightarrow{d_0}\cO\to 0
\end{equation} 
where we may take $d_1(A_1)$ to map $1\in S_1$ to $(1,-1)\in S_1\oplus S_1$
(or to $(s,-s)$ for any $s$ that is a unit in the ring $S_1$),
and similarly we may take $d_1(A_2)$ to map $1\in S_2$ to
$(1,-1)\in S_2\oplus S_2$.

\subsection{Our First $H^i(\cF)$ Definition is Consistent}
\label{su_our_first_H_i_cF_definitions_consistent}

In this subsection we prove that our definition of $H^i(\cF)$
for a $k$-diagram, $\cF$, is actually 
$H^i(\cF)\eqdef\Ext^i_{\cO_r}(\cO_r,\cF)$
as defined by the 
derived category.
% , and that the Yoneda pairing defined in
% Section~\ref{se_duality} is the true Yoneda pairing defined in
% Section~\ref{se_strong_duality}.

For any $\cF$, we may compute $H^i(\cF)=\Ext^i_{\cO_r}(\cO_r,\cF)$
from any projective resolution of $\cO$; taking the resolution
\eqref{eq_projective_resolution_of_cO}, and hence
$H^i(\cF)$ are the homology groups of
\begin{equation}\label{eq_equation_that_gives_tau_from_Hom_cP_zero_cF_etc}
0\to \Hom_\cO(\cP_0,\cF) \xrightarrow{\tau} \Hom_\cO(\cP_1,\cF) \to 0,
\end{equation} 
i.e., the cokernel of $\tau$.
(This is a classical theorem about Ext; see
\cite{weibel}, Theorem~2.7.6, page~63;
to see its equality in the context of the derived category,
see the discussion around \eqref{eq_chain_map_cPs_to_cF_in_position_zero}
below.)

In view of
\eqref{eq_hom_from_coskyscraper_to_diagram_and_hom_on_A_values}
and
\eqref{eq_hom_from_coskyscraper_to_diagram_and_hom_on_B_values}
$\tau$ is canonically isomorphic to the map
$$
\bigoplus_{i=1}^3 \Hom_{R_i}\bigl(R_i,\cF(B_i)\bigr)
\to \bigoplus_{j=1}^2 \Hom_{S_j}\bigl(S_j,\cF(A_j)\bigr)
$$
whose cokernel is that of the map
\begin{equation}\label{eq_sum_of_cF_Bs_to_sum_cF_As}
\bigoplus_{i=1}^3 \cF(B_i)
\to \bigoplus_{j=1}^2 \cF(A_j).
\end{equation} 
In view of the fact that we may take for $j=1,2$ the maps
$$
d_1(A_j) \from S_j \to S_j^2=\cP_0(A_j)\isom\cP_0(B_j)\oplus\cP_0(B_3)
$$
to be the maps
$$
d_1(A_j) 1  = (1,-1)  ,
$$
the map in
\eqref{eq_sum_of_cF_Bs_to_sum_cF_As} can be taken to be the map
$$
\bigoplus_{i=1}^3 \cF(B_i)
\to \bigoplus_{j=1}^2 \cF(A_j)
$$
which is the sum of maps taking, for $j\in[2]$,
$\cF(B_j)$ to $\cF(A_j)$ via the identity map, and
taking $\cF(B_3)$ to $\cF(A_j)$ via minus the identity map.

But this agrees with our definition of $\cF(\partial)$ in
\eqref{eq_formula_for_cF_partial} in Definition~\ref{de_diagram_k_vs}.
Hence the two definitions agree.

\subsection{Our First Definition of the Yoneda Pairing is Consistent}
\label{su_first_definition_of_Yoneda_pairing_is_consistent}

In Subsection~\ref{su_yoneda_pairing_simple_form_i_one_A_equals_cO}
we defined the Yoneda pairing.  Here we show that this agrees with
the Yoneda pairing as defined in
Subsection~\ref{su_yodena_pairing_derived_category}.

So let $\cF,\cG$ be $\cO$-modules, which we view as elements of
the derived category, $\cD(\cO)$, as complexes that live only in
the $0$-place (i.e., we write $\cF,\cG$ for $\cF[0],\cG[0]$\footnote{
  This is the convention in \cite{gelfand}.
  }).
Note that $\Hom_{\cD(\cO)}(\cF,\cG)$ is exactly
the set of morphisms $\Hom(\cF,\cG)$ in place $0$
(see \cite{gelfand}, Proposition~2, Section~III.5, page~164),
although we don't really need this fact.

First, we can compute $\Hom_{\cD(\cO)}(\cO,\cF)$ by taking a 
projective resolution of $\cO$ (see \cite{gelfand}, Theorem~21, page~179,
and Exercise~1, page~183, in
Section~III.5, or \cite{weibel}, Corollary~10.4.7, page~388).
Hence $H^1(\cF)$ is the set of chain maps up to homotopy of:
\begin{equation}\label{eq_chain_map_cPs_to_cF_in_position_zero}
\begin{tikzpicture}[scale=0.4]
\node (A0) at (0,3) {$0$};
\node (A1) at (4,3) {$\cP_1$};
\node (A2) at (8,3) {$\cP_0$};
\node (A3) at (12,3) {$0$};
\draw[->] (A0) -- (A1);
\draw[->] (A1) -- (A2);
\draw[->] (A2) -- (A3);
\node (B0) at (0,0) {$0$};
\node (B1) at (4,0) {$\cF$};
\node (B2) at (8,0) {$0$};
\node (B3) at (12,0) {$0$};
\draw[->] (B0) -- (B1);
\draw[->] (B1) -- (B2);
\draw[->] (B2) -- (B3);
\draw[->] (A0) -- (B0);
\draw[->] (A1) -- (B1) node[midway,left] {$\alpha$};
\draw[->] (A2) -- (B2);
\draw[->] (A3) -- (B3);
\end{tikzpicture}
\end{equation} 
and each such chain map is determined by $\alpha\from\cP_1\to\cF$.
Since the homotopy maps are determined by a map $\cP_0\to\cF$,
we conclude that the set of chain maps are precisely the cokernel
of $\tau$ in
\eqref{eq_equation_that_gives_tau_from_Hom_cP_zero_cF_etc}.

Now if $\phi\from\cF\to\cG$ is any map, then the Yoneda pairing
gives a map in the derived category
\begin{equation}\label{eq_chain_map_cF_to_cG_in_position_zero}
\begin{tikzpicture}[scale=0.4]
\node (A0) at (0,3) {$0$};
\node (A1) at (4,3) {$\cF$};
\node (A2) at (8,3) {$0$};
\node (A3) at (12,3) {$0$};
\draw[->] (A0) -- (A1);
\draw[->] (A1) -- (A2);
\draw[->] (A2) -- (A3);
\node (B0) at (0,0) {$0$};
\node (B1) at (4,0) {$\cG$};
\node (B2) at (8,0) {$0$};
\node (B3) at (12,0) {$0$};
\draw[->] (B0) -- (B1);
\draw[->] (B1) -- (B2);
\draw[->] (B2) -- (B3);
\draw[->] (A0) -- (B0);
\draw[->] (A1) -- (B1) node[midway,left] {$\phi$};
\draw[->] (A2) -- (B2);
\draw[->] (A3) -- (B3);
\end{tikzpicture}
\end{equation} 
Hence the Yoneda pairing composes
\eqref{eq_chain_map_cF_to_cG_in_position_zero} and
\eqref{eq_chain_map_cPs_to_cF_in_position_zero}, which gives the
map
\begin{equation}\label{eq_chain_map_cPs_to_cG_in_position_zero}
\begin{tikzpicture}[scale=0.4]
\node (A0) at (0,3) {$0$};
\node (A1) at (4,3) {$\cP_1$};
\node (A2) at (8,3) {$\cP_0$};
\node (A3) at (12,3) {$0$};
\draw[->] (A0) -- (A1);
\draw[->] (A1) -- (A2);
\draw[->] (A2) -- (A3);
\node (B0) at (0,0) {$0$};
\node (B1) at (4,0) {$\cG$};
\node (B2) at (8,0) {$0$};
\node (B3) at (12,0) {$0$};
\draw[->] (B0) -- (B1);
\draw[->] (B1) -- (B2);
\draw[->] (B2) -- (B3);
\draw[->] (A0) -- (B0);
\draw[->] (A1) -- (B1) node[midway,left] {$\phi\alpha$};
\draw[->] (A2) -- (B2);
\draw[->] (A3) -- (B3);
\end{tikzpicture}
\end{equation} 
Hence the Yoneda pairing has $\phi$ operating by composition as a map
\begin{equation}\label{eq_phi_circ_map_from_Hom_cP_one_cF_to_same_cG}
\phi\circ\cdot \from \Hom(\cP_1,\cF)\to \Hom(\cP_1,\cG)
\end{equation} 
taking $\alpha$ to $\phi\circ\alpha$.
Now we check that the isomorphisms
\eqref{eq_hom_from_coskyscraper_to_diagram_and_hom_on_A_values}
are functorial in $\cF$, and therefore
under the identifications
$$
\Hom(\cP_1,\cF)\isom 
\bigoplus_{j=1}^2 \cF(A_j),
\qquad
\Hom(\cP_1,\cG)\isom 
\bigoplus_{j=1}^2 \cG(A_j),
$$
composition in $\phi$ in \eqref{eq_phi_circ_map_from_Hom_cP_one_cF_to_same_cG}
becomes the map
$$
\bigoplus_{j=1}^2 \cF(A_j)
\to
\bigoplus_{j=1}^2 \cG(A_j)
$$
that is just the sum of 
composition in $\phi(A_1)$ and $\phi(A_2)$.

But this is just how we defined the Yoneda pairing in
Subsection~\ref{su_yoneda_pairing_simple_form_i_one_A_equals_cO}.
Hence this definition is consistent with how we are defining it in
terms of the derived category.

\subsection{The Yoneda Pairing and Ext Functoriality}
\label{su_yoneda_pairing_and_ext_functoriality}

To prove \eqref{eq_adjoint_first}, we will need the following
facts: if $\cF_1,\cF_2,\cG$ are $\cO$-modules, and
$f\from \cF_1\to \cF_2$ is a morphism, then there is
a map
$$
f_* \from \Ext^i(\cG,\cF_1) \to \Ext^i(\cG,\cF_2) ,
$$
and a map
$$
f^* \from \Ext^i(\cF_2,\cG)\to\Ext^i(\cF_1,\cG);
$$
this is the usual functoriality of Ext groups.
We will need to know that $f_*$ agrees with the composition
in the derived category,
\begin{equation}\label{eq_Yoneda_pairing_gives_functoriality_f_lower_star}
\Hom_\cD(\cG[-i],\cF_1) \times \Hom_\cD(\cF_1,\cF_2)
\to \Hom_\cD(\cG[-i],\cF_2);
\end{equation} 
this is true because
\eqref{eq_Yoneda_pairing_gives_functoriality_f_lower_star} is
the functoriality of $\Hom_\cD(\cG[-i],\cF)$ in the variable, $\cF$,
this is proven in \cite{gelfand}, Subsection~III.6.14.
Similarly
$f^*$ agrees with the composition
$$
\Hom_\cD(\cF_1,\cF_2) \times \Hom_\cD(\cF_2,\cG[i]) \to
\Hom_\cD(\cF_1,\cG[i])
$$
(see the last paragraph of \cite{gelfand}, Subsection~III.6.14).

\subsection{The First Adjointness Property}
\label{su_proof_of_first_adjointness_property}

In this subsection we prove
\eqref{eq_adjoint_first}.

So let $f\from \cF_1\to\cF_2$ be a morphism of $\cO$-modules, and let
$\cG$ be another $\cO$-module.  Then for $i=0,1$ we have a map
$$
\Ext^i(\cO,\cF_1)\times\Hom(\cF_1,\cF_2)\times\Ext^{1-i}(\cF_2,\cG)
\to \Ext^1(\cO,\cG)=H^1(\cG)
$$
given by composing the Yoneda pairings which are a composition of
Hom sets in the derived category $\cD(\cO)$:
$$
\Hom(\cO[-i],\cF_1)\times 
\Hom(\cF_1,\cF_2) \times
\Hom(\cF_2,\cG[1-i])
\to \Hom(\cO[-i],\cG[1-i])\isom H^1(\cG).
$$

So let 
$f\from\cF_1\to\cF_2$ be a morphism (which we also view as a 
morphism in $\Hom_\cD(\cF_1,\cF_2)$), and let
$$
\alpha\in \Hom(\cO[-i],\cF_1),
\quad
\gamma \in \Hom(\cF_2,\cG[1-i]).
$$
From the associativity of $\Hom$ in the derived category we have that
\begin{equation}
\label{eq_associativity_of_gamma_f_alpha}
(\gamma\circ f) \circ \alpha
=
\gamma\circ (f\circ \alpha).
\end{equation} 
Then from Subsection~\ref{su_yoneda_pairing_and_ext_functoriality}
we know that $f\circ\alpha=f_*\alpha$ as $f_*$ acts on Ext groups,
and $\gamma\circ f=f^*\gamma$.
Hence \eqref{eq_associativity_of_gamma_f_alpha} implies that in the
Yoneda pairing, we have \eqref{eq_adjoint_first}.

\subsection{Proof of \eqref{eq_adjoint_second}}
\label{su_proof_of_second_adjointness_property}

To verify \eqref{eq_adjoint_second}, we need a similar functoriality.
Let us explain.

Let
$0\to\cF_1\to\cF_2\to\cF_3\to 0$ be an exact sequence.  Then there
corresponds a distinguished triangle in the derived category
$$
\cF_1[0]\to\cF_2[0]\to\cF_3[0]\xrightarrow{\delta} \cF_1[1].
$$
The Yoneda pairing gives a map
$$
\Hom(\cO,\cF_3)\times \Hom(\cF_3,\cF_1[1]) \times \Hom(\cF[1],\cG[1])
$$
by composition which is associative, i.e., for any
$$
\alpha\in\Hom(\cO,\cF_3),
\ \delta\in\Hom(\cF_3,\cF_1[1]),
\ \gamma\in\Hom(\cF[1],\cG[1]),
$$
we have
$$
\gamma\circ(\delta\circ\alpha)
=(\gamma\circ\delta)\circ\alpha .
$$

So it suffices to show that 
\begin{enumerate}
\item
$\delta\circ\alpha=\delta_*\alpha$
is the usual $\delta$-connecting map in the long exact sequence
\begin{equation}\label{eq_usual_connecting_map}
\cdots \to \Ext^0(\cO,\cF_3) \to \Ext^1(\cO,\cF_1) \to \cdots
\end{equation} 
obtained from the short exact sequence
$0\to\cF_1\to\cF_2\to\cF_3\to 0$; and
\item 
similarly for $\gamma\circ\delta=\delta^*\gamma$.
\end{enumerate}
We will prove~(1); (2)~is proven analogously.
Note that~(1) and~(2) likely
follow from the general theory of the long exact
sequence for the Ext groups and that for distinguished triangles (could
these two long exact sequences really be different?); however,
we have not found a place in the literature that states this formally.
Hence we include a proof.
  
In the short exact sequence
$0\to\cF_1\to\cF_2\to\cF_3\to 0$, the map 
$$
\delta\from \cF_3[0]\to\cF_1[1]
$$
is given by the roof 
\begin{equation}\label{eq_the_roof_that_gives_the_delta_map}
\begin{tikzpicture}[scale=0.30]
\node (1) at (0,0) {$0\xrightarrow{\qquad}\cF_3$};
\node (2) at (8,4) {$\cF_1\xrightarrow{\qquad}\cF_2$};
\node (3) at (16,0) {$\cF_1\xrightarrow{\qquad} 0$};
\draw[->](7.5,3) to (0,1);
\draw[->](8.5,3) to (16,1);
\end{tikzpicture}
\end{equation} 
(see \cite{kashiwara}, just before Proposition~13.1.15, page~323,
or \cite{gelfand} Lemma~III.3.3, the map $\delta=\delta(f)$);
the left arrow is a quasi-isomorphism and the right is a morphism.
So given an
$$
\alpha\in\Hom_\cD(\cO,\cF_3),
$$
we take a projective resolution $0\to\cP_1\to\cP_0\to 0$ of $\cO$,
and we may realize $\alpha$ as a map of complexes
$$
\begin{tikzpicture}[scale=0.30]
\node (1) at (0,4) {$\cP_1$};
\node (2) at (4,4) {$\cP_0$};
\node (3) at (0,0) {$0$};
\node (4) at (4,0) {$\cF_3$};
\draw[->] (1) to (2);
\draw[->] (1) -- (3) node[midway,left] {$0$};
\draw[->] (2) -- (4) node[midway,right] {$\alpha$};
\draw[->] (3) to (4);
\end{tikzpicture}
$$
let us show that the image of $\alpha$ in $\cF_1\to 0$ is just the
usual connecting map.  Since $\cP_0$ is projective and $\cF_2\to\cF_3$
is surjective, we have a map
$\beta\from\cP_0\to\cF_2$ with a commutative diagram
$$
\begin{tikzpicture}[scale=0.30]
\node (2) at (4,4) {$\cP_0$};
\node (3) at (0,0) {$\cF_2$};
\node (4) at (4,0) {$\cF_3$};
\draw[->] (2) to (4);
\draw[->] (2) -- (3) node[midway,left] {$\beta\,$};
\draw[->] (3) to (4);
\end{tikzpicture}
$$
which gives a map $\tau$
$$
\begin{tikzpicture}[scale=0.30]
\node (1) at (0,4) {$\cP_1$};
\node (2) at (4,4) {$\cP_0$};
\node (3) at (0,0) {$\cF_2$};
\node (4) at (4,0) {$\cF_3$};
\draw[->] (1) to (2);
\draw[->] (1) -- (3) node[midway,left] {$\tau$};
\draw[->] (2) to (4);
\draw[->] (2) to (3);
\draw[->] (3) to (4);
\end{tikzpicture}
$$
Since $\tau$ followed by $\cF_2\to\cF_3$ is the map
$\cP_1\to\cP_0\to\cF_3$ which is the zero map, the image of $\tau$
is in the kernel of $\cF_2\to\cF_3$, which therefore gives an arrow
$\sigma$
$$
\begin{tikzpicture}[scale=0.30]
\node (1) at (0,4) {$\cP_1$};
\node (2) at (4,4) {$\cP_0$};
\node (0) at (-4,0) {$\cF_1$};
\node (3) at (0,0) {$\cF_2$};
\node (4) at (4,0) {$\cF_3$};
\draw[->] (1) to (2);
\draw[->] (1) -- (3) node[midway,left] {$\tau$};
\draw[->] (1) -- (0) node[midway,left] {$\sigma$};
\draw[->] (2) to (4);
\draw[->] (2) to (3);
\draw[->] (3) to (4);
\draw[->] (0) to (3);
\end{tikzpicture}
$$
We therefore get a diagram
$$
\begin{tikzpicture}[scale=0.30]
\node (1) at (0,0) {$\cP_1\xrightarrow{\qquad}\cP_0$};
\node (2) at (8,4) {$\cP_1\xrightarrow{\qquad}\cP_0$};
\node (3) at (16,0) {$0\xrightarrow{\qquad}\cF_3$};
\node (4) at (24,4) {$\cF_1\xrightarrow{\qquad}\cF_2$};
\node (5) at (16,8) {$\cP_1\xrightarrow{\qquad}\cP_0$};
\draw[->](7.5,3) to (0,1);
\draw[->](8.5,3) to (15.5,1);
\draw[->](15.5,7) to (8,5);
\draw[->](16.5,7) to (24,5);
\draw[->](24,3) to (16.5,1);
\end{tikzpicture}
$$
whose leftward arrows are all quasi-isomorphisms.
We can therefore compose the map $\cP_1\to\cP_0$ to 
$0\to\cF_3$ to $\cF_1\to 0$
(the latter from \eqref{eq_the_roof_that_gives_the_delta_map})
in the derived category
from the diagram:
$$
\begin{tikzpicture}[scale=0.30]
\node (1) at (0,0) {$\cP_1\xrightarrow{\qquad}\cP_0$};
\node (2) at (8,4) {$\cP_1\xrightarrow{\qquad}\cP_0$};
\node (3) at (16,0) {$0\xrightarrow{\qquad}\cF_3$};
\node (4) at (24,4) {$\cF_1\xrightarrow{\qquad}\cF_2$};
\node (5) at (16,8) {$\cP_1\xrightarrow{\qquad}\cP_0$};
\node (6) at (32,0) {$\cF_1\xrightarrow{\qquad}0$};
\draw[->](7.5,3) to (0,1);
\draw[->](8.5,3) to (15.5,1);
\draw[->](15.5,7) to (8,5);
\draw[->](16.5,7) to (24,5);
\draw[->](23.5,3) to (16.5,1);
\draw[->](24.5,3) to (32,1);
\end{tikzpicture}
$$
The composition is therefore the map of complexes, from the
top centre
complex of the above diagram to the lower right complex, i.e., the map:
$$
\begin{tikzpicture}[scale=0.30]
\node (1) at (0,4) {$\cP_1$};
\node (2) at (4,4) {$\cP_0$};
\node (3) at (0,0) {$\cF_1$};
\node (4) at (4,0) {$0$};
\draw[->] (1) to (2);
\draw[->] (1) -- (3) node[midway,left] {$\sigma$};
\draw[->] (2) -- (4) node[midway,right] {$0$};
\draw[->] (3) to (4);
\end{tikzpicture}
$$
Now we have to check that the map $\alpha\mapsto\sigma$
is the usual connecting map 
\eqref{eq_usual_connecting_map} that we obtain from
$$
\begin{tikzpicture}[scale=0.30]
\node (0) at (0,4) {$\Hom(\cP_1,\cF_1)$};
\node (1) at (10,4) {$\Hom(\cP_1,\cF_2)$};
\node (2) at (20,4) {$\Hom(\cP_1,\cF_3)$};
\draw[->] (0) to (1);
\draw[->] (1) to (2);
\node (3) at (0,0) {$\Hom(\cP_0,\cF_1)$};
\node (4) at (10,0) {$\Hom(\cP_0,\cF_2)$};
\node (5) at (20,0) {$\Hom(\cP_0,\cF_3)$};
\draw[->] (3) to (4);
\draw[->] (4) to (5);
\draw[->] (3) to (0);
\draw[->] (4) to (1);
\draw[->] (5) to (2);
\end{tikzpicture}
$$
So consider the usual connecting map (or ``snake lemma'') procedure:
if $\alpha\in\Hom(\cP_0,\cF_3)$ lies in $\Ext^0(\cO,\cF_3)$,
then to $\alpha$ there corresponds and element of $\Hom(\cP_0,\cF_2)$,
and $\beta$ above is precisely one such example.
Then $\beta$ maps precisely to $\tau$ as above in $\Hom(\cP_1,\cF_2)$,
which maps to $0$ in $\Hom(\cP_1,\cF_3)$ and therefore
comes precisely from what we above called $\sigma\in\Hom(\cP_1,\cF_1)$.
Hence the map $\alpha\mapsto\sigma$ is precisely the usual
connecting map or ``snake lemma'' construction.

%    Bibliography styles amsplain or harvard are also acceptable.
\providecommand{\bysame}{\leavevmode\hbox to3em{\hrulefill}\thinspace}
\providecommand{\MR}{\relax\ifhmode\unskip\space\fi MR }
% \MRhref is called by the amsart/book/proc definition of \MR.
\providecommand{\MRhref}[2]{%
  \href{http://www.ams.org/mathscinet-getitem?mr=#1}{#2}
}
\providecommand{\href}[2]{#2}

% \bibliography{bibrefs}

% \bibliography{bibJan16}

\end{document}